\documentclass{article}

\usepackage[utf8]{inputenc}
\usepackage{amsmath,amsfonts,amssymb}
\usepackage{amsthm,upref}
\usepackage[a4paper]{geometry}
\usepackage{color} 
\usepackage[dvipsnames]{xcolor}
\usepackage{graphicx}
\usepackage{subfigure}
\usepackage{tabularx,float}
\usepackage[hidelinks]{hyperref}
\usepackage{comment} 
\usepackage{tabularx} 
\usepackage{titlefoot} 
\usepackage{fancyhdr}

\newtheorem{thm}{Theorem}[section]
\newtheorem{definition}[thm]{Definition}

\newtheorem{proposition}[thm]{Proposition}
\newtheorem{corollary}[thm]{Corollary}
\newtheorem{lemma}[thm]{Lemma}
\newtheorem{remark}[thm]{Remark}
\newtheorem{assumption}{Assumption}

\newcommand{\eps}{\varepsilon}
\newcommand{\R}{\mathbb{R}}

\newcommand{\transpose}{\textnormal{T}}
\newcommand{\me}{\textrm{e}}
\newcommand{\D}{\textrm{D}}

\newcommand\rev[1]{{\color{black}{#1}}} 

\title{Extending Discrete Geometric Singular Perturbation Theory to Non-Hyperbolic Points}
\author{S.~Jelbart\thanks{Corresponding author. Email: samuel.jelbart@asc.tuwien.ac.at} \& C.~Kuehn}
\date{\today}

\begin{document}
	\maketitle
	
	\begin{abstract}
		We extend the recently developed \textit{discrete geometric singular perturbation theory} to the non-normally hyperbolic regime. Our primary tool is the \textit{Takens embedding theorem}, which provides a means of approximating the dynamics of particular maps with the time-1 map of a formal vector field. First, we show that the so-called \textit{reduced map}, which governs the slow dynamics near slow manifolds in the normally hyperbolic regime, can be locally approximated by the \rev{time-1} map of the reduced vector field which appears in continuous-time geometric singular perturbation theory. In the non-normally hyperbolic regime, we show that the dynamics of fast-slow maps with a unipotent linear part can be locally approximated by the time-1 map induced by a fast-slow vector field in the same dimension, which has a nilpotent singularity of the corresponding type. The latter result is used to describe (i) the local dynamics of two-dimensional fast-slow maps with non-normally singularities of regular fold, transcritical and pitchfork type, and (ii) dynamics on a (potentially high\rev{-}dimensional) local center manifold in $n$-dimensional fast-slow maps with regular contact or fold submanifolds of the critical manifold. In general, our results show that the dynamics near a large and important class of singularities in fast-slow maps can be described via the use of formal embedding theorems which allow for their approximation by the time-1 map of a fast-slow vector field featuring a loss of normal hyperbolicity.
	\end{abstract}
	
	
	

	\section{Introduction}
	\label{sec:introduction}

	\textit{Geometric singular perturbation theory (GSPT)} is a powerful and established approach to the mathematical analysis of multiple time-scale systems of ordinary differential equations (ODEs); see \cite{Fenichel1979,Jones1995,Kuehn2015,Wechselberger2019} for seminal works, well-known overviews and books on GSPT. The aim of this article is to continue the development of a corresponding theory for discrete fast-slow systems induced by repeated iteration of a map. We follow \cite{Jelbart2022a} in referring to this theory as \textit{discrete geometric singular perturbation theory (DGSPT)}, and we consider this work to be a natural sequel and complement to the mathematical formalism developed therein.
	
	DGSPT applies to the study of discrete dynamical systems induced by iteration of fast-slow maps or difference equations in the general form
	\begin{equation}
		\label{eq:non_stnd_maps}
		z \mapsto H(z,\eps) = z + h(z) + \eps G(z,\eps) ,
	\end{equation}
	where $z \in \R^n$ for an integer $n \geq 2$, the functions $H(z,\eps)$, $h(z)$ and $\eps G(z,\eps)$ are $C^r$-smooth for a positive integer $r \geq 1$, $\eps$ is a perturbation parameter satisfying $0 < \eps \ll 1$, and the zero set $\{z \in \R^n : h(z) = 0\}$ is assumed to contain a $C^r$-smooth, $(k<n)$-dimensional \textit{critical manifold} $S$, which forms a manifold of fixed points for the limiting map $z \mapsto H(z,0)$. The direct ODE counterpart to the class of fast-slow maps defined by \eqref{eq:non_stnd_maps} is the class defined by ODE systems of the form
	\[
	\frac{\textup{d} z}{\textup{d} t} = h(z) + \eps G(z,\eps) .
	\]
	A recent formulation of GSPT for fast-slow ODE systems in this class can be found in \cite{Wechselberger2019}, see also \cite{Goeke2014,Kruff2019,Lax2020,Lizarraga2020c,Lizarraga2020}. We emphasise that similarly to the theory for ODEs, the class of fast-slow maps defined by \eqref{eq:non_stnd_maps} contains, but is not limited to, the well-known class of fast-slow maps in \textit{standard form}
	\begin{equation}
		\label{eq:stnd_form_maps}
		\begin{split}
			x &\mapsto x + f(x,y,\eps) , \\
			y &\mapsto y + \eps g(x,y,\eps) ,
		\end{split}
	\end{equation}
	where $(x,y) \in \R^{n-k} \times \R^k$, $0 < \eps \ll 1$ and the functions $f$ and $\eps g$ are $C^r$-smooth.
	
	\
	
	The primary contribution of \cite{Jelbart2022a} was to provide the following for general fast-slow maps \eqref{eq:non_stnd_maps}:
	\begin{enumerate}
		\item[(I)] `Singular theory' for the identification and analysis of simplified fast and slow limiting problems;
		\item[(II)] A set of perturbation theorems analogous to those provided by \textit{Fenichel theory} \cite{Fenichel1979,Jones1995,Wiggins2013} in the fast-slow ODE setting, assuming a regularity condition known as \textit{normal hyperbolicity}.
	\end{enumerate}
	The singular theory developed in (I) applies to the identification and analysis of the \textit{layer map}, which approximates fast dynamics in regions of phase space bounded away from the critical manifold $S$, and the \textit{reduced map}, which approximates the dynamics near $S$. A notion of normal hyperbolicity was formulated in terms of the distribution of multipliers associated with the linearisation of the layer map $z \mapsto H(z,0)$ along $S$. Specifically, \rev{the layer map} is normally hyperbolic \rev{at $z \in S$} 
	if the Jacobian matrix $\D H(z,0)$ has $k$ \textit{trivial multipliers} that are identically equal to $1$ (their corresponding eigenvectors span the tangent space $\textup{T}_zS$), and $n-k$ \textit{non-trivial multipliers} $\mu_j(z)$ satisfying $\mu_j(z) \in \mathbb C \setminus S^1$, i.e.~$|\mu_j(z)| \neq 1$. It was shown that under normally hyperbolic assumptions, classical invariant manifold theorems dating back to \cite{Hirsch1970} could be specialised in order to characterise the persistence of compact normally hyperbolic submanifolds of $S$ as locally invariant \textit{slow manifolds} $S_\eps$ for all sufficiently small $\eps > 0$ (the authors actually used more recent, and more concrete invariant manifold theory due to Nipp \& Stoffer \cite{Nipp2013}). The center-stable/unstable manifolds associated to $S$, as well as their asymptotic rate foliations, were also shown to persist in a manner which is (for the most part) qualitatively similar to Fenichel theory.
	
	\
	
	Although we shall also derive new results in the normally hyperbolic regime, the primary aim of this work is to make steps towards the extension of DGSPT into the non-normally hyperbolic regime, which is not considered in \cite{Jelbart2022a}. More precisely, we consider local dynamics near non-normally hyperbolic singularities on $S$. On the linear level, codimension-1 singularities \rev{of $H(z,0)$} at $z \in S$ can be classified and divided into three distinct types:
	\begin{enumerate}
		\item[(i)] Fold/contact-type singularities with a single non-trivial multiplier of $DH(z,0)$ at $1$;
		\item[(ii)] Flip/period-doubling type singularities with a single non-trivial multiplier of $DH(z,0)$ at $-1$;
		\item[(iii)] Neimark-Sacker/torus-type singularities with a pair of nontrivial multipliers of $DH(z,0)$ on $S^1$ with non-zero imaginary part.
	\end{enumerate}
	Due to the coarseness of this `classification', (i) also includes common singularity types 
	such as transcritical and pitchfork singularities. The same is true of (ii)-(iii). For analytical purposes, many authors choose to group singularities in (ii)-(iii) together -- they are sometimes referred to collectively as `oscillatory singularities' -- because a number of important analytical methods and approaches apply equally to the study of dynamics near both singularity types. They also exhibit dynamical similarities. Numerous authors have shown that many singularities in (ii)-(iii) are associated with \textit{delayed stability loss}; see \cite{Baesens1991,Baesens1995,Fruchard2003,Fruchard2009,Neishtadt1996,Neishtadt1987} for important works in this direction. Similarly to the well-known delayed stability loss phenomena associated with delayed Hopf bifurcations in fast-slow ODE systems \cite{Baer1989,Hayes2016,Kuehn2015,Neishtadt1988,Neishtadt1987}, the existence of delay is sensitive to smoothness and noise.\footnote{\rev{The similarities between discrete and continuous systems here appear to be more than just qualitative: comparing the asymptotic estimates in the references mentioned here shows that the magnitude of the delay associated with dynamic flip and Neimark-Sacker bifurcations are in \textit{quantitative} agreement with the corresponding estimates for delayed Hopf bifurcations in continuous-time systems.}}
	
	We will focus on the dynamics near fold/contact-type singularities in (i), which appear to have received less attention in the literature. We believe that one important reason for this is that the analytical methods used to study singularities in (ii)-(iii) rely on a certain regularity property which tends to be violated at points where a non-trivial multiplier of the linearisation is equal to $1$. Thus, alternative methods are needed. A number of works have shown that an adaptation of the well-known \textit{geometric blow-up method} \cite{Dumortier1996,Jardon2021,Krupa2001a,Krupa2001b} can be applied to \rev{the} study of discretized fast-slow ODEs with singularities in (i) \cite{Arcidiacono2019,Engel2019,Nipp2013,Nipp2009}. The approach adopted in these works relies in an important way on scaling properties of the discretization parameter, but there are significant obstacles to the extension of this approach to the study of general fast-slow maps, i.e.~to the study of maps that are not obtained via discretization. In order to understand the dynamics near singularities in (i) in a more general setting, we adopt an alternative approach which relies on 
	the \textit{Takens embedding theorem} \rev{\cite{Chen1965,Chow1994,Gramchev2005,Takens1973}}, a formal embedding theorem which, under certain conditions on the linearisation, guarantees the existence of an $n$-dimensional vector field with a time-$1$ map which has the same local Taylor expansion as the original map \eqref{eq:non_stnd_maps}. Rather than tackling the problem of local dynamics directly, we show that in many cases, the Takens embedding theorem can be used to obtain a local approximation of the map by a time-1 map induced by the flow of a fast-slow ODE system in the same dimension. We prove and apply a number of formal embedding theorems of this kind in both the normally hyperbolic and non-normally hyperbolic regimes.\footnote{It is important to distinguish \textit{formal embeddings}, which involve agreement between formal power series associated with a map and the time-$1$ map of a vector field, from \textit{exact embeddings}, which involve an equality (not only in series) between the map and the time-$1$ map. We refer to Theorem \ref{thm:embedding_nilpotent} and Remark \ref{rem:embedding} for a more precise distinction.}
	
	The most important result we obtain in the normally hyperbolic regime pertains to the reduced map, which approximates the slow dynamics near normally hyperbolic submanifolds of $S$ to leading order in $\eps$. We show that the reduced map can be approximated by the time-1 map induced by the \textit{reduced vector field}, which governs the slow dynamics on normally hyperbolic submanifolds of the critical manifold in fast-slow ODE systems \cite{Fenichel1979,Kuehn2015,Wechselberger2019}. This is particularly useful in practice due to the fact that the reduced vector field can be determined explicitly in terms of the functions in \eqref{eq:non_stnd_maps} and their derivatives. This shows that in many cases, the slow dynamics in fast-slow maps can be analysed using the 
	reduced vector field which appears in fast-slow ODE theory.
	
	The most important result we obtain in the non-normally hyperbolic regime pertains to the dynamics near singularities in a large and important subclass of singularities of type (i), known as \textit{unipotent singularities}, for which the linearisation $\D H(z,0)$ at such a point is unipotent, i.e.~$\D H(z,0) = \mathbb I_n + \Lambda$ for a nilpotent matrix $\Lambda$. We show that here too, the local dynamics of the map can be approximated by the time-1 map induced by a fast-slow vector field in the same dimension. The approximating vector field has a critical manifold which is $C^r$-close to the critical manifold of the map, and (assuming $r$ is large enough) is in many cases expected to have a (continuous-time) nilpotent singularity of the corresponding type. We show, for example, that the dynamics of planar fast-slow maps in standard form \eqref{eq:stnd_form_maps} with singularities of regular fold, transcritical and pitchfork type, can be locally approximated by the time-1 map induced by planar fast-slow ODEs in standard form with singularities of regular fold, transcritical and pitchfork type, respectively. We then combine these results with (already known) results on the ODE counterparts to these problems, which can be found in \cite{Krupa2001a,Krupa2001c}, in order to characterise the extension of the attracting slow manifold through a neighbourhood of the singularity in the map.
	
	Finally, we show that the dynamics near regular contact type singularities \rev{of maps} in $\R^n$, which may be viewed as the counterpart to regular fold type singularities for problems in general non-standard form \eqref{eq:non_stnd_maps} (they reduce to regular fold singularities after a local transformation to standard form \eqref{eq:stnd_form_maps}), can be understood in a similar way. In general, regular contact points are not unipotent in $n$-dimensional fast-slow maps with $\textup{codim} (S) = n - k \geq 2$. One can, however, apply a center manifold reduction in order to obtain a fast-slow map on a $(k+1)$-dimensional center manifold with a regular contact point on a critical manifold which is codimension-$1$ in $\R^{k+1}$. 
	This is analogous to the center manifold reduction applied in \cite{Wechselberger2012,Wechselberger2019} to study dynamics near regular \rev{fold/contact} submanifolds in fast-slow ODEs. In contrast to the original (higher dimensional) map, the linearised problem within the center manifold is unipotent, allowing for the application of the embedding-based methods and results that we derived for unipotent singularities. This allows us to show that the dynamics on the center manifold can be approximated by the time-1 map of a $(k+1)$-dimensional fast-slow ODE system with a regular contact point. This is advantageous, given that the dynamics of the latter have already been described in \cite{Wechselberger2012,Wechselberger2019} (see also \cite{Lizarraga2020c} for more on contact points in fast-slow ODEs).
	
	\
	
	The manuscript is structured as follows. In Section \ref{sec:setup_and_definitions} we introduce basic notions and provide an overview of DGSPT in the normally hyperbolic regime. We also present two new formal embedding theorems, one of which characterises the close relationship between the reduced map and the reduced vector field from continuous-time GSPT. The rest of the paper is devoted to the approximation of local dynamics near non-normally hyperbolic singularities \rev{of the map} on $S$. We begin with the linear classification of codimension-1 singularities in Section \ref{sec:nilpotent_embedding}. Our main approximation result on the dynamics near unipotent singularities is stated and proven in Section \ref{sub:existence_of_formal_embeddings_for_nilpotent_singularities}. This result is used to generate approximations and results on the local geometry and dynamics of $2$-dimensional, standard form fast-slow maps with regular fold, transcritical and pitchfork type singularities in Section \ref{sec:2d_applications}. The approximation of dynamics near regular contact points in $n$-dimensional fast-slow maps in general (nonstandard) form \eqref{eq:non_stnd_maps} is treated in Section \ref{sec:regular_contact_points}. We conclude in Section \ref{sec:summary_and_outlook} with a summary and outlook. Finally, a number of important results are collected in the Appendix, including an explicit characterisation the Takens embedding theorem in the form which is relevant for our purposes in Appendix \ref{app:Takens_embedding_theorem}. 

	\section{DGSPT and formal embeddings for the slow dynamics}
	\label{sec:setup_and_definitions}
	
	We consider $\eps$-families of maps in the form
	\begin{equation}
		\label{eq:general_maps_original}
		z \mapsto \bar z = z + h(z) + \eps G(z,\eps) , 
	\end{equation}
	where $z \in \R^n$ for $n \geq 2$, $\eps \in [0,\eps_0]$ is a small perturbation parameter, and the functions $h : \R^n \to \R^n$ and $\eps G : \R^n \times [0,\eps_0] \to \R^n$ are $C^r$-smooth in $z$ and (for the latter) $\eps$. In the following we require only that $r \geq 1$, however additional smoothness will be required in later sections in order to define and study the dynamics near certain singularities.
	
	DGSPT applies to the class of maps \eqref{eq:general_maps_original} defined by the following additional assumptions, which have direct counterparts in the definition of fast-slow ODE systems (see e.g.~\cite{Wechselberger2019}):
	
	\begin{assumption}
		\label{ass:fast-slow}
		The map \eqref{eq:general_maps_original} is fast-slow in the sense of \cite[Def.~2.1]{Jelbart2022a}. More precisely, the set of fixed points $C := \{ z \in \R^n : h(z) = 0 \}$ \rev{of the map $z \mapsto z + h(z)$} contains a $(k<n)$-dimensional, regularly embedded submanifold $S \subset \R^n$. We assume for simplicity that $S$ is connected.
	\end{assumption}
	
	\begin{assumption}
		\label{ass:factorisation}
		The $\eps$-independent terms can be factorised via
		\begin{equation}
			\label{eq:factorisation}
			h(z) = N(z) f(z) ,
		\end{equation}
		where $N(z)$ is an $n \times (n-k)$ matrix with full column rank, and $f(z)$ is an $(n-k) \times 1$ column vector. Singular points $z_\ast \notin S$ which satisfy $N(z_\ast) f(z_\ast) = 0$, if they exist, are assumed to be isolated.
	\end{assumption}
	
	Using Assumptions \ref{ass:fast-slow} and \ref{ass:factorisation} we may rewrite \eqref{eq:general_maps_original} in the general form
	\begin{equation}
		\label{eq:general_maps}
		z \mapsto \bar z = H(z,\eps) = z + N(z) f(z) + \eps G(z,\eps) ,
	\end{equation}
	for which the fixed point manifold $S$ described in Assumption \ref{ass:fast-slow}, which shall be referred to as the \textit{critical manifold}, can be written as
	\begin{equation}
		\label{eq:critical_manifold}
		S := \left\{ z \in \R^n : f(z) = 0 \right\} .
	\end{equation}
	It is worthy to note that locally, 
	Assumption \ref{ass:factorisation} follows from Assumption \ref{ass:fast-slow}. This can be shown after a direct application of Hadamard's lemma; see \cite[Rem.~2]{Lax2020} and Appendix \ref{app:Hadamard}. 
	For this reason, Assumption \ref{ass:factorisation} is not strictly necessary in the statement of local results. We have decided to retain it here and throughout, however, since it simplifies the presentation and allows us to provide an overview of DGSPT which may also apply to the study of global dynamics.
	
	\begin{remark}
		The particular form of the factorisation \eqref{eq:factorisation} can often be obtained by inspection in applications, however we refer to \cite[App.~3]{Goeke2014} for an algebraic procedure for determining it (locally). Globally, such a factorisation is not always attainable; see \cite{deMaesschalck2021,Wechselberger2019} for counterexamples in the continuous-time setting.
	\end{remark}
	
	\begin{remark}
		\label{rem:standard_form}
		Consider the map \eqref{eq:general_maps} in the special case that $z = (x,y) \in \R^{n-k} \times \R^k$, with
		\[
		N(z) = \begin{pmatrix}
			\mathbb I_{n-k} \\
			\mathbb O_{k,n-k}
		\end{pmatrix} , 
		\qquad f(z) = \tilde f(x,y,0) , \qquad 
		G(z,\eps) = 
		\begin{pmatrix}
			\eps^{-1}( \tilde f(x,y,\eps) - \tilde f(x,y,0)) \\
			\tilde g(x,y,\eps)
		\end{pmatrix}, 
		\]
		where $\mathbb O_{j,k}$ denotes the $j \times k$ zero matrix for $(j,k) \in \mathbb N_+^2$ (we shall use this notation throughout). In this case one obtains a fast-slow map in the well-known \textit{standard form}
		\begin{equation}
			\label{eq:standard_form}
			\begin{split}
				x \mapsto \bar x &= x + \tilde f(x,y,\eps) , \\
				y \mapsto \bar y &= y + \eps \tilde g(x,y,\eps) .
			\end{split}
		\end{equation}
		Thus the class of maps defined by \eqref{eq:general_maps} and Assumptions \ref{ass:fast-slow} and \ref{ass:factorisation} includes, but is not limited to, the class of fast-slow maps in standard form.
	\end{remark}

	\subsection{DGSPT in the normally hyperbolic regime}
	\label{sub:dgspt}
	
	Here we provide an overview of the theory established in \cite{Jelbart2022a}, which applies to general fast-slow maps \eqref{eq:general_maps} under Assumptions \ref{ass:fast-slow} and \ref{ass:factorisation}.

	\subsubsection{The layer map}
	\label{sub:layer_map}
	
	The limiting problem associated to the fast dynamics is obtained by taking $\eps \to 0$ in \eqref{eq:general_maps}, leading to the so-called \textit{layer map}
	\begin{equation}
		\label{eq:layer_map}
		z \mapsto \bar z = z + N(z) f(z) .
	\end{equation}
	The critical manifold $S$, recall \eqref{eq:critical_manifold}, is contained in the fixed point set of \eqref{eq:layer_map}. Since $S$ is $(k<n)$-dimensional, the linearisation along $S$, i.e.~
	\begin{equation}
		\label{eq:layer_linearisation}
		\D (z + N(z) f(z)) = \mathbb I_n + N(z) \D f(z) , \qquad z \in S,
	\end{equation}
	has at least $k$ multipliers which are identically equal to $1$. These are referred to as the \textit{trivial multipliers}. We denote the remaining $n-k > 0$ \textit{non-trivial multipliers} by $\mu_j(z)$, $j = 1, \ldots, n-k$. Linear stability along $S$ is characterised and classified in terms of the distribution of the non-trivial multipliers $\mu_j(z)$ in the complex plane. 
	
	\begin{definition}
		\label{def:normal_hyperbolicity}
		We say that \rev{the map \eqref{eq:layer_map}} is \textit{normally hyperbolic} \rev{at a point $z \in S$} if there are no non-trivial multipliers on the unit circle, i.e.~if
		\[
		|\mu_j(z)| \neq 1, \qquad \forall j = 1, \ldots, n-k.
		\]
		\rev{The map is normally hyperbolic along a} subset or submanifold \rev{$S_{\textup{nh}} \subseteq S$ if it} is normally hyperbolic \rev{at each point in $S_{\textup{nh}}$.}
	\end{definition}
	
	For calculations it is worthy to note that the $n-k$ non-trivial multipliers can be identified with the multipliers of the $(n-k) \times (n-k)$ square matrix $\mathbb I_{n-k} + \D f(z) N(z)$, by \cite[Prop.~2.10]{Jelbart2022a}. Thus, the problem of checking for normal hyperbolicity, or equivalently, identifying a loss of normal hyperbolicity, reduces to the question of whether or not the spectrum of $\mathbb I_{n-k} + \D f(z) N(z)$ contains points on the unit circle. 
	
	The local dynamics of the map \eqref{eq:general_maps} is also \rev{affected} by \textit{superstability}, which is related to non-invertibility of the map \eqref{eq:general_maps}, and occurs near points on $S$ \rev{where the linearisation has} at least one non-trivial multiplier equal to zero.
	
	\begin{proposition}
		\label{prop:invertibility}
		Consider a point $z \in S$. The map \eqref{eq:general_maps} is a diffeomorphism on a neighbourhood $\mathcal U$ about $z$ in $\R^n$ for all $\eps \in [0,\eps_0)$ if and only if $\eps_0 > 0$ is sufficiently small and
		\begin{equation}
			\label{eq:invertibility_cond}
			\mu_j(z) \neq 0 , \qquad \forall j = 1, \ldots, n-k .
		\end{equation}
	\end{proposition}
	
	\begin{proof}
		Let $z \in S$. By the inverse function theorem, the extended map
		\[
		F(z,\eps) :=
		\begin{pmatrix}
			z + N(z) f(z) + \eps G(z, \eps) \\
			\eps
		\end{pmatrix} 
		\]
		is a local diffeomorphism on a neighbourhood $\mathcal U \times [0,\eps_0)$ about $(z,0)$ in $\R^n \times \R_{\geq 0}$ if and only if the determinant of the Jacobian
		\[
		\D F(z,0) = \mathbb I_{n+1} + 
		\begin{pmatrix}
			N \D f (z) & G(z,0) \\
			\mathbb O_{1,n} & 1
		\end{pmatrix}
		\]
		is non-zero. Now let $\{\mu_i(z)\}_{i = 1}^{n+1}$ denote the multipliers of $\D f(z,0)$, where the first $n-k$ denote the non-trivial multipliers. The remaining multipliers are given by $\mu_i(z) = 1$, where $i = n-k+1, \ldots , n+1$. Using this fact together with \eqref{eq:invertibility_cond} yields
		\[
		\det \D f(z,0) 
		= \prod_{i=1}^{n+1} \mu_i(z)
		= \prod_{i=1}^{n-k} \mu_i(z) \neq 0 ,
		\]
		which implies the result.
	\end{proof}
	
	It is worthy to note that \rev{the map can be normally hyperbolic \textit{and} superstable (and therefore non-invertible) at a point} 
	$z \in S$. 
	The notion of normal hyperbolicity does not depend on invertibility.

	\subsubsection{Slow manifolds and dynamics for $0 < \eps \ll 1$}
	\label{ssub:Fenichel_theorems}
	
	Similarly to Fenichel theory for fast-slow ODEs in the continuous-time setting (see e.g.~\cite{Fenichel1979,Jones1995,Kuehn2015,Wechselberger2019,Wiggins2013}), the theory developed in \cite{Jelbart2022a} applies under normally hyperbolic assumptions. For the purposes of this overview, we restrict ourselves to the statement pertaining to the persistence of compact normally hyperbolic submanifolds of $S$ as locally invariant \textit{slow manifolds} for $0 < \eps \ll 1$. For detailed statements on the persistence and perturbation of the local center-stable/unstable manifolds $W_{\textrm{loc}}^{\textrm{s}/\textrm{u}}(S)$, as well as their asymptotic rate foliations, we refer to \cite[Sec.~3]{Jelbart2022a}. In the following we write
	\begin{equation}
		\label{eq:S_n}
		S_{\textrm{nh}} := \left\{ z \in S : |\mu_j(z)| \neq 1, \ \forall j = 1, \ldots n-k \right\} ,
	\end{equation}
	which defines the normally hyperbolic subset of $S$. Generically, $S_{\textrm{nh}} \subseteq S$ is a union of connected $k$-dimensional submanifolds of $S$ such that $\overline{S_{\textrm{nh}}} = S$.
	
	\begin{thm}
		\label{thm:slow_manifolds}
		\textup{\cite[Thm.~3.1]{Jelbart2022a}} Consider the map \eqref{eq:general_maps} under Assumptions \ref{ass:fast-slow} and \ref{ass:factorisation}, and let $\widetilde S \subseteq S_{\textup{nh}}$ be a compact, connected submanifold of $S_{\textup{nh}}$. Then there is an $\eps_0 > 0$ such that for all $\eps \in (0,\eps_0)$, there exists a compact and connected slow manifold $\widetilde S_\eps$ which is
		\begin{itemize}
			\item[(i)] $O(\eps)$-close and diffeomorphic to $\widetilde S$;
			\item[(ii)] $C^r$-smooth in both $z$ and $\eps$;
			\item[(iii)] locally invariant under \eqref{eq:general_maps}, i.e.~the restricted map \eqref{eq:general_maps}$|_{\widetilde S_\eps}$ is invertible and satisfies the following:
			\begin{itemize}
				\item[(a)] If $z \in \widetilde S_\eps$ and $\bar z^j \in \mathcal U$ for all $j = 1, \ldots, l$, where $\mathcal U$ is a neighbourhood of $\widetilde S$, then $\bar z^j \in \widetilde S_\eps$ for all $j = 1, \ldots, l$;
				\item[(b)] If $z \in \widetilde S_\eps$ and $\bar z^{-j} \in \mathcal U$ for all $j = 1, \ldots, l$, where $\mathcal U$ is a neighbourhood of $\widetilde S$, then $\bar z^{-j} \in \widetilde S_\eps$ for all $j = 1, \ldots, l$.
			\end{itemize}
		\end{itemize}
	\end{thm}
	
	The existence of locally invariant slow manifolds allows for a rigorous fast-slow decomposition of the dynamics. To leading order, the fast dynamics can be analysed using properties of the layer map \eqref{eq:layer_map} described in Section \ref{sub:layer_map} above. We turn now to the slow dynamics, which can be analysed after restriction to the lower dimensional slow manifolds described by Theorem \ref{thm:slow_manifolds}.

	\subsubsection{The reduced map}
	\label{sub:reduced_map}
	
	The \textit{reduced map}, which governs the dynamics on slow manifolds to leading order in $\eps$, can be defined using the following result from \cite{Jelbart2022a}.
	
	\begin{proposition}
		\label{prop:reduced_map}
		\textup{\cite[Prop.~2.16]{Jelbart2022a}}
		Let $\widetilde S_{\eps}$ denote a slow manifold perturbing from a compact submanifold $\widetilde{S} \subseteq S_{\textup{nh}}$, as described by Theorem \ref{thm:slow_manifolds}. Then for all $\eps \in [0,\eps_0)$ with $\eps_0 > 0$ sufficiently small we have
		\begin{equation}
			\label{eq:slow_map}
			z \mapsto z + \eps \Pi_{\mathcal N}^{S_{\textup{nh}}} G(z,0) + O(\eps^2) , \qquad 
			z \in \widetilde{S} ,
		\end{equation}
		where $\Pi_{\mathcal N}^{S_{\textup{nh}}}$ denotes the oblique projection along the linear fast fibre bundle
		\[
		\mathcal N := \bigcup_{z \in S_{\textup{nh}}} \textup{span}\{N^i(z)\}_{j=1, \ldots, n-k} 
		\]
		onto the tangent bundle $\textup{T}S_{\textup{nh}}$, where $N^i(z)$ denotes the $i$'th column of the matrix $N(z)$. It has the matrix representation
		\begin{equation}
			\label{eq:projection_operator}
			\Pi_{\mathcal N}^{S_{\textup{nh}}} = \mathbb I_n - N (\textup{D} f N)^{-1} \textup{D} f \big|_{S_{\textup{nh}}} .
		\end{equation}
	\end{proposition}
	
	Proposition \ref{prop:reduced_map} provides the explicit asymptotic form of the map \eqref{eq:general_maps} after restriction to a locally invariant slow manifold $\widetilde S_{\eps}$, up to $O(\eps^2)$. \rev{Indeed, one can view the vector field $\Pi_{\mathcal N}^{S_{\textrm{nh}}} G(z,0)$ as the limiting reduced problem for the fast-slow map \eqref{eq:general_maps}, since it is generated by the slow map \eqref{eq:slow_map} as $\eps \to 0$. More precisely, if we denote the right-hand side of \eqref{eq:slow_map} by $\mathcal G(z,\eps)$, then $\mathcal G(z,\eps)$ generates the vector field $\Pi_{\mathcal N}^{S_{\textrm{nh}}} G(z,0)$ on $\widetilde S$ when $\eps \to 0$ in the sense that the derivative at $\eps = 0$ is given by
		\begin{equation}
			\label{eq:convergence}
			\lim_{\eps \to 0} \frac{\mathcal G(z,\eps) - \mathcal G(z,0)}{\eps} = 
			\Pi^{S_{\textup{nh}}}_{\mathcal N} G(z,0) , \qquad z \in \widetilde S .
		\end{equation}
		Intuitively, we think of the dynamics induced by the slow map \eqref{eq:slow_map} as `converging' to a flow as the distance between iterates, which is generically $O(\eps)$, tends to zero.
		
		Alternatively, we may define the \textit{reduced map}} 
	as the leading order approximation for the slow dynamics, i.e.~by truncating \eqref{eq:slow_map} at order $O(\eps)$:
	\begin{equation}
		\label{eq:reduced_map}
		z \mapsto z + \eps \Pi_{\mathcal N}^{S_{\textrm{nh}}} G(z,0) , \qquad 
		z \in \widetilde S .
	\end{equation}
	Interestingly, the reduced map \eqref{eq:reduced_map} coincides with the Euler discretization of the ODE problem
	\begin{equation}
		\label{eq:reduced_vf_fast}
		z' = \frac{d z}{d t} = \eps \Pi_{\mathcal N}^{S_{\textrm{nh}}} G(z,0) , \qquad z \in \widetilde S ,
	\end{equation}
	with step-size $h=\eps$ or, equivalently,
	\begin{equation}
		\label{eq:reduced_vf}
		\dot z = \frac{d z}{d \tau} = \Pi_{\mathcal N}^{S_{\textrm{nh}}} G(z,0) , \qquad z \in \widetilde S .
	\end{equation}
	Because of \rev{this close relationship between the slow map \eqref{eq:slow_map} and the reduced vector field $\Pi^{S_{\textup{nh}}}_{\mathcal N} G(z,0)$}, a number of important dynamical questions about the reduced map \eqref{eq:reduced_map} can be reformulated as questions about the reduced vector field \eqref{eq:reduced_vf}. We illustrate the point with the following simple result.
	
	\begin{proposition}
		\rev{There exists an $\eps_0 > 0$ such that for all $\eps \in (0,\eps_0)$, t}he reduced map \eqref{eq:reduced_map} has a fixed point at $z = z_0$ if and only if the reduced vector field \eqref{eq:reduced_vf} has an equilibrium at $z = z_0$. The multipliers of the linearised reduced map satisfy the hyperbolicity condition $||\mu_i| - 1| > c > 0$ for an arbitrarily small but fixed constant $c$ if and only if $z_0$ is a hyperbolic equilibrium of the reduced vector field \eqref{eq:reduced_vf}.
	\end{proposition}
	
	\begin{proof}
		Since
		\[
		z = z_0 \mapsto z_0 + \eps \Pi^{S_{\textrm{nh}}}_{\mathcal N} G(z_0,0) = 0 
		\quad \iff \quad 
		\Pi^{S_{\textrm{nh}}}_{\mathcal N} G(z_0,0) = 0 ,
		\]
		the reduced map \eqref{eq:reduced_map} has a fixed point at $z = z_0$ if and only if the reduced vector field has an equilibrium at $z = z_0$. The multipliers of the reduced map at $z = z_0$ are given by $\mu_j = 1 + \eps \lambda_j$, where $j = 1, \ldots, n$ and $\lambda_j$ are the eigenvalues of the matrix $\D (\Pi_{\mathcal N}^{S_{\textrm{nh}}}G(z,0))(z_0)$. The hyperbolicity statement follows from the fact that
		\[
		|\mu_i|^2 = 1 + 2 \eps \textup{Re} (\lambda_i) + O(\eps^2) ,
		\]
		as $\eps \to 0$.
	\end{proof}
	
	In the following section we shall derive an even stronger relationship between the map \eqref{eq:slow_map} -- and therefore also the reduced map \eqref{eq:reduced_map} -- and the reduced vector field \eqref{eq:reduced_vf}. 
	
	\begin{remark}
		\label{rem:invertibility}
		The maps \eqref{eq:slow_map} and \eqref{eq:reduced_map} are diffeomorphisms even if the original map \eqref{eq:general_maps} is not invertible, i.e.~\eqref{eq:slow_map} and \eqref{eq:reduced_map} are also invertible at superstable points on $S$. This feature follows from the fact that \eqref{eq:general_maps}$|_{\widetilde S_{\eps}}$ is invertible due to Theorem \ref{thm:slow_manifolds} Assertion (iii). It can also be seen as a direct consequence of the inverse function theorem, since the Jacobian matrices associated with the restricted maps in \eqref{eq:slow_map} and \eqref{eq:reduced_map} are $O(\eps)$-close to $\mathbb I_n$. To illustrate the point, consider the simple example $(x,y) \mapsto (x^2, y + \eps)^\transpose$, which can be written in general form \eqref{eq:general_maps} with $N(x,y) = (1,0)^\transpose$, $f(x,y) = - x + x^2$ and $G(x,y,\eps) = (0,1)^\transpose$. The critical manifold $S$ has a branch along $x = 0$ with a single non-trivial multiplier $\mu \equiv 0$. Thus $S$ is superstable and the map is non-invertible along $x = 0$. However, the corresponding reduced map is given by $(0,y)^\transpose = (0, y + \eps)^\transpose$, which is clearly invertible.
	\end{remark}
	
	\begin{remark}
		\label{rem:reduced_map_projection}
		The domain of definition of the reduced map \eqref{eq:reduced_map} can be smoothly extended to the set $S_{\textup{nh}} \cup S_\textup{ns}$, where
		\begin{equation}
			\label{eq:S_h}
			S_\textup{ns} := \left\{ z \in S : \mu_j(z) \in S^1 \setminus \{1\}, \ \forall j = 1, \ldots n-k \right\} ,
		\end{equation}
		\rev{and} the subscript ``ns" stands for ``Neimark-Sacker". Like $S_{\textup{nh}}$, the set $S_{\textup{nh}} \cup S_\textup{ns}$ is generically a union of $k$-dimensional submanifolds of $S$ satisfying $\overline{S_{\textup{nh}} \cup S_\textup{ns}} = S$. Such an extension is possible because the projection operator $\Pi_{\mathcal N}^{S}$ remains well-defined at \rev{points where the map is} non-normally hyperbolic with $\mu_j(z) \neq 1$. Note, however, that $\Pi_{\mathcal N}^S$ is not well-defined on $S \setminus (S_{\textup{nh}} \cup S_\textup{ns})$, since the matrix $\textup{D} f N|_{S \setminus (S_{\textrm{nh}} \cup S_\textup{ns})}$ is not invertible.
	\end{remark}

	\subsection{Formal embeddings in the normally hyperbolic regime}
	\label{sub:slow_embeddings}
	
	Our aim in this section is to derive formal embedding theorems which can be used to approximate the dynamics of the map \eqref{eq:general_maps} by the time-1 map induced by a flow in the same dimension. We state two different results. The first of these applies to the dynamics in a neighbourhood about in a given point $z \in S_{\textrm{nh}}$. The second can be used in order to relate the dynamics of the reduced map \eqref{eq:reduced_map} to the dynamics of the (continuous-time) reduced vector field defined in \eqref{eq:reduced_vf}.
	
	\
	
	We shall assume throughout this section that $0 \in S_{\textrm{nh}}$ (this can be achieved after a simple coordinate translation), and consider the expansion about $(z,\eps) = (0,0) \in \R^{n+1}$ obtained from \eqref{eq:general_maps} after appending the trivial map $\eps \to \eps$. We shall also assume that \eqref{eq:general_maps} is $C^r$-smooth with $r \geq 2$. Similarly to the proof of Proposition \ref{prop:invertibility}, we write
	\begin{equation}
		\label{eq:Taylor_expansion_nh}
		F(z,\eps) := 
		\begin{pmatrix}
			H(z,\eps) \\
			\eps
		\end{pmatrix}
		=
		A
		\begin{pmatrix}
			z \\
			\eps
		\end{pmatrix}
		+ 
		F^{(2)}(z,\eps) + \cdots + F^{(r)}(z,\eps) + o(|(z,\eps)|^r) ,
	\end{equation}
	where each $F^{(j)}(z,\eps)$ is a homogeneous polynomial of degree $j \in \{2, \ldots, r \}$, and 
	\begin{equation}
		\label{eq:A_def}
		A := \D f(0,0) = 
		\begin{pmatrix}
			\mathbb I_n + N \D f(0) & G(0, 0) \\
			\mathbb O_{1,n} & 1
		\end{pmatrix} .
	\end{equation}
	Since every linear transformation on $\R^{n+1}$ has a Jordan-Chevalley decomposition, the matrix $A$ can be written as
	\begin{equation}
		\label{eq:A}
		A = B \left( \mathbb I_{n+1} + M \right) ,
	\end{equation}
	where $B$ is semi-simple, $M$ is nilpotent, and $B M = M B$. In particular, there exists a (generally complex) invertible $(n + 1) \times (n + 1)$ matrix $\mathcal S$ such that
	\begin{equation}
		\label{eq:A2}
		A = \mathcal S^{-1} \mathcal D \mathcal S^{-1} + \mathcal S^{-1} \left( J(A) - \mathcal D \right) \mathcal S ,
	\end{equation}
	where $\mathcal D := \textup{diag} (\mu_1, \ldots, \mu_{n-k}, 1, \ldots, 1)$ is the matrix of (generally complex) multipliers of $A$, which has $k+1$ trivial multipliers equal to $1$. $J(A)$ denotes the Jordan decomposition of $A$, which is block diagonal with Jordan blocks associated to the non-trivial multipliers in the upper left part, and $\mathbb I_{k+1}$ in the bottom right block. As usual, the matrix $\mathcal S$ can be constructed in applications using the eigenvectors associated to $A$. Comparing \eqref{eq:A} with \eqref{eq:A2}, we have
	\begin{equation}
		\label{eq:BM}
		B = \mathcal S^{-1} \mathcal D \mathcal S, \qquad 
		M = \mathcal S^{-1} \left( \mathcal D^{-1} J(A) - \mathbb I_{n+1} \right) \mathcal S .
	\end{equation}
	The semi-simple matrix $B$ will play an important role in the approximation of the local dynamics of the map \eqref{eq:general_maps} by a flow. In order to state this result precisely, however, we need two more important notions and an additional assumption.
	
	\begin{definition}
		\label{def:jl_truncation}
		Consider a $C^r$-smooth function $\phi : \R^n \to \R^m$ with the following finite order Taylor expansion at $z = z_0$:
		\[
		\phi(z) = 
		\sum_{j=0}^r \phi^{(j)}(z - z_0) + o(|z-z_0|^r) ,
		\]
		where $\phi^{(0)}(z - z_0) = \phi(z_0)$ and the coefficient functions $\phi^{(j)} : \R^n \to \R^m$ 
		are vectors of degree $j$ homogeneous polynomials defined via Taylor expansion. 
		For each $l = 1, \ldots, r$ we define the $l$-jet of $\phi(z)$ and the truncation operator $j^l$ at $z = z_0$ by
		\[
		j^l \phi(z) = \sum_{j=0}^l \phi^{(j)} (z - z_0) .
		\]
	\end{definition}
	
	\begin{definition}
		\label{def:Poincare_normal_form}
		Consider a $C^r$-smooth map $\phi : \Omega \to \R^n$ such that $\phi(0) = 0$, where $\Omega$ is a neighbourhood of $0$ in $\R^n$. The map $\phi(z)$ is said to be in local Poincar\'e normal form up to order $l$ if the corresponding $l$-jet $j^l \phi(z)$ can be written as a linear combination of ``resonant monomials" of the form
		\[
		z_1^{m_1} \cdots z_n^{m_n} e_j ,
		\]
		where $z = (z_1, \ldots, z_n) \in \R^n$, $e_j$ is a standard basis vector in $\R^n$, $j \in \{1, \ldots, n\}$, and the $m_i$ are positive integers which satisfy $\sum_{i = 1}^n m_i \geq 2$ and the resonance condition
		\[
		\lambda_1^{m_1} \cdots \lambda_n^{m_n} = \lambda_j .
		\]
	\end{definition}
	
	Finally, we assume the following local invertibility property. 
	
	\begin{assumption}
		\label{ass:invertibility}
		The origin $0 \in S$ is not superstable, i.e.~the non-trivial multipliers $\mu_j(0)$ satisfy the local invertibility condition
		\[
		\mu_j(0) \neq 0 , \qquad \forall j = 1, \ldots, n-k.
		\]
	\end{assumption}
	
	If Assumption \ref{ass:invertibility} is satisfied, then Proposition \ref{prop:invertibility} ensures that \eqref{eq:general_maps} is a diffeomorphism on a suitable neighbourhood about $\mathcal U \ni 0$ in $\R^n$.
	
	\
	
	We now state and prove two formal embedding theorems, which can be used in order to approximate the dynamics of fast-slow maps close to a normally hyperbolic point \rev{of the map on $S$}. The proofs rely on the \textit{Takens embedding theorem}; see Appendix \ref{app:Takens_embedding_theorem} for a formulation based on \cite[Thm.~8.1]{Chow1994} which is suitable for our purposes.
	
	\begin{thm}
		\label{thm:nh_approximation}
		Assume that $0 \in S_{\textup{nh}}$ and consider the extended map \eqref{eq:Taylor_expansion_nh}, where Assumptions \ref{ass:fast-slow}, \ref{ass:factorisation} and \ref{ass:invertibility} apply to $H(z,\eps)$. There exists and $\eps_0 > 0$ such that for each $l \in \{1, \ldots, r\}$ there exists a neighbourhood $\mathcal U_l \times [0,\eps_0)$ about $\tilde z = (z,\eps)^\transpose = (0,0)^\transpose$ in $\R^n \times \R_{\geq 0}$, a diffeomorphism $\psi_l : \mathcal U_l \times [0,\eps_0) \to \R^{n+1}$, and a vector field $X : \mathcal U_l \times [0,\eps_0) \to \textup{T}\R^{n+1} \cong \R^{n+1}$ such that
		\begin{enumerate}
			\item $j^l(\psi_l \circ F \circ \psi_l^{-1})$ is a Poincar\'e normal form of $F$ up to order $l$;
			\item $X(B \tilde z) = B X(\tilde z)$ for any $\tilde z = (z,\eps) \in \mathcal U_l \times [0,\eps_0)$, where $B$ is the semi-simple matrix defined in \eqref{eq:BM};
			\item $j^l(\psi_l \circ F \circ \psi_l^{-1})(\tilde z) = j^l(\Phi_X^1(B \tilde z))$, where $\Phi_X^t(\tilde z)$ denotes the flow of $X(\tilde z)$ by time $t$.
			\item The $l$-jet of the vector field $j^l X(\tilde z)$ is uniquely determined by the $l$-jet of the map $j^l F(\tilde z)$.
		\end{enumerate}
		If the matrix $\textup{D} f N(0)$ has purely real eigenvalues, the preceding assertions hold for $B = A$, where $A$ is the Jacobian matrix defined in \eqref{eq:A_def}.
	\end{thm}
	
	\begin{proof}
		Assertions 1-4 follow after a direct application of Takens' embedding theorem; we refer again to Appendix \ref{app:Takens_embedding_theorem}. 
		
		If all $n-k$ eigenvalues $\lambda_j(0)$ of the matrix $\D f N(0)$ are real, then the $n-k$ non-trivial multipliers $\mu_j(0) = 1 + \lambda_j(0)$ corresponding to eigenvalues of the matrix $\mathbb I_{n-k} + \D f N(0)$ are also real. Since the remaining $k+1$ eigenvalues of $\mathbb I_n + N \D f(0)$ are identically $1$, it follows that the matrix $A$ defined in \eqref{eq:A_def} has purely real eigenvalues, implying that the semisimple part of $A$ is $A$, i.e.~that $B = A$.
	\end{proof}
	
	Theorem \ref{thm:nh_approximation} shows that in a neighbourhood of a normally hyperbolic, non-superstable point, the original map \eqref{eq:general_maps} is $C^r$-conjugate to a map which can be formally approximated by the time-1 flow of an $\eps$-dependent, $C^r$-smooth formal vector field in the same dimension. The matrix $B$ can be determined algorithmically, and the diffeomorphisms $\psi_l$ can be determined in a recursive and systematic fashion using normal form theory, see e.g.~\cite{Chow1994,Gramchev2005}. In practice, however, this is often a non-trivial and computationally expensive task, which is expected to lead to difficulties when it comes to the question of whether Theorem \ref{thm:nh_approximation} can be used as an approximation tool in applications.
	
	From an applied point of view, it may be preferable to consider approximations obtained via a formal embedding of the slow map \eqref{eq:slow_map}, which governs the dynamics on $k$-dimensional slow manifolds $S_\eps$. In order to state our results, we denote the right-hand side of the map \eqref{eq:slow_map} \rev{again} by $\mathcal G(z,\eps)$, i.e.~we write
	\begin{equation}
		\label{eq:slow_map_2}
		\mathcal G(z, \eps) = z + \eps \Pi_{\mathcal N}^{S_{\textrm{nh}}} G(z,0) + O(\eps^2) , \qquad 
		z \in \widetilde S ,
	\end{equation}
	where $\widetilde S$ is a compact submanifold of $S_{\textrm{nh}}$, and let $\Phi^t_{\Pi_{\mathcal N}^{S_{\textrm{nh}}} G(z,0)}(z,\eps)$ denote the flow induced by the reduced problem \eqref{eq:reduced_vf}. Our second approximation theorem shows that $\mathcal G(z,\eps)$ can be locally approximated by the time-$1$ map $\Phi^1_{\Pi_{\mathcal N}^{S_{\textrm{nh}}} G(z,0)}(z,\eps)$. 
	
	\begin{thm}
		\label{thm:reduced_map}
		Consider the map \eqref{eq:slow_map_2} and let $z \in S_{\textup{nh}}$. There is an $\eps_0 > 0$ and $\eps$-independent neighbourhood $\Omega \ni z$ in $S_{\textup{nh}}$ such that for all $(z,\eps) \in \Omega \times [0,\eps_0)$ we have
		\begin{equation}
			\label{eq:approximation}
			j^1 \mathcal G(z,\eps) = j^1 \Phi^1_{\Pi_{\mathcal N}^{S_{\textup{nh}}} G(z,0)}(z,\eps) , \qquad 
			j^l \mathcal G(z,\eps) = j^l \Phi^1_{\Pi_{\mathcal N}^{S_{\textup{nh}}} G(z,0)}(z,\eps) + O(\eps^2) ,
		\end{equation}
		where the latter holds for all $l = 2, \ldots, r$.
	\end{thm}
	
	
	The proof of Theorem \ref{thm:reduced_map} is deferred to Appendix \ref{app:reduced_map} for brevity. The existence of a unique formal embedding, i.e.~a unique formal vector field $X(z,\eps)$ whose time-$1$ map has a Taylor expansion coinciding with the Taylor expansion of the slow map $\mathcal G(z,\eps)$, follows from the Takens embedding theorem. 
	However, additional arguments are needed in order to gain information about the particular form of $X(z,\eps)$, and in particular, to show that the unique formal vector field $X(z,\eps)$ is locally approximated by the reduced vector field, i.e.~that
	\[
	j^l X(z,0) = j^l \Pi_{\mathcal N}^{S_{\textrm{nh}}} G(z,0)
	\]
	for each $l = 2, \ldots, r$. This is shown by a recursive argument based on the formal matching of Taylor coefficients of $H(z,\eps)$ and $\Phi^1_{\Pi_{\mathcal N}^{S_{\textrm{nh}}} G(z,0)}(z,\eps)$, where the Taylor coefficients of the latter are obtained via Picard iteration. This is similar to the procedure for approximating a map by a flow applied in e.g.~\cite{Kuznetsov2013,Kuznetsov2019}.
	
	The practical advantage of Theorem \ref{thm:reduced_map} over Theorem \ref{thm:nh_approximation} is that preliminary transformations are not necessary (there is no need to identify the transformations $\psi_l$ appearing in Theorem \ref{thm:nh_approximation}), and that the approximating vector field is known; it is precisely the reduced vector field \eqref{eq:reduced_vf}. It is also worthy to note that Theorem \ref{thm:reduced_map} still applies if the original map \eqref{eq:general_maps} is orientation-reversing, or even non-invertible (Assumption \ref{ass:invertibility} is not necessary). This is a simple consequence of the fact that the dynamics on the slow manifold $\widetilde S_{\eps}$ itself is orientation-preserving and invertible, even if \eqref{eq:general_maps} is not invertible on a neighbourhood in about $z \in S_{\textrm{nh}}$ in $\R^n$ (we refer back to Remark \ref{rem:invertibility}).
	
	\begin{remark}
		\label{rem:Lie_Algebra}
		\rev{Theorem \ref{thm:reduced_map} and its proof in Appendix \ref{app:reduced_map} can also be formulated using Lie Algebras. In this context, the aim is to show that there exists a formal vector field $V$ such that $\mathcal G = \exp(V)$. Since $\mathcal G$ is known, the vector field $V$ can be formally determined using the matrix logarithm expansion
			\[
			V = \log(\mathcal G) = \sum_{k=1}^\infty (-1)^{k+1} \frac{(\mathcal G - \mathbb I_n)^k}{k} .
			\]	
			Theorem \ref{thm:reduced_map} shows that this vector field exists and that $j^l V = j^l(\Pi_{\mathcal N}^{S_{\textup{nh}}} G)$ for each $l = 1, \ldots, r$. We refer to \cite[Sec.~6]{Gramchev2005} and in particular the formulation and proof of Proposition 6.1 therein for an example of this approach in the context of a Takens embedding theorem similar to the one Corollary \ref{cor:Takens}, Appendix \ref{app:Takens_embedding_theorem}.}
	\end{remark}
	
	\begin{remark}
		The $O(\eps^2)$ term in the second equation in \eqref{eq:approximation} is calculated explicitly in the proof of Theorem \ref{thm:reduced_map}, and given by equation \eqref{eq:eps2_correction}. In fact, the proof is formal and constructive, allowing for the calculation of corrections at $O(\eps^l)$ for all integers $l \geq 2$. \rev{Since the terms in \eqref{eq:eps2_correction} are generically non-zero, the error in \eqref{eq:approximation} is $O(\eps^2)$. However, we conjecture that $j^l \mathcal G(z,\eps)$ coincides with the time-1 map induced by a corresponding system of fast-slow ODEs restricted to a slow manifold which has been expanded up to $O(\eps^{l-1})$. The details are left for future work.}
	\end{remark}
	
	\begin{remark}
		\label{rem:local_normal_form}
		There are certain situations in which an exact embedding can be obtained, as opposed to the finite order approximations in Theorems \ref{thm:nh_approximation} and \ref{thm:reduced_map}. For example, in  \cite{Dumortier2022} the author shows that planar fast-slow maps
		\[
		\begin{pmatrix}
			x \\
			y
		\end{pmatrix} 
		\mapsto 
		\begin{pmatrix}
			x \\
			y
		\end{pmatrix}  + 
		\begin{pmatrix}
			N^x(x,y) \\
			N^y(x,y)
		\end{pmatrix}
		f(x,y) + \eps 
		\begin{pmatrix}
			G^x(x,y,\eps) \\
			G^y(x,y,\eps)
		\end{pmatrix} ,
		\]
		which are in the general form \eqref{eq:general_maps}, can be embedded into a $2$-dimensional flow in a neighbourhood $\mathcal U \subseteq \R^2$ about any point $(x,y) \in S$ for which the conditions
		\begin{equation}
			\label{eq:planar_embedding_conds}
			\mu(x,y) \in (0,1) , \qquad N^x(x,y) G^y(x,y,0) - N^y(x,y) G^x(x,y,0) \neq 0,
		\end{equation}
		are satisfied. The first condition in \eqref{eq:planar_embedding_conds} implies that the map is locally orientation-preserving and that $S$ is normally hyperbolic and attracting near $(x,y)$. The second guarantees that there are no fixed points in the reduced map. The question of whether or not exact embeddings of this kind are possible under similar assumptions in higher dimensions is not considered in this work.
	\end{remark}

	\section{Formal embeddings in the non-normally hyperbolic regime}
	\label{sec:nilpotent_embedding}
	
	
	We now turn to the study of dynamics in the non-normally hyperbolic regime. We begin in Section \ref{sub:loss_of_normal_hyperbolicity} with basic definitions and a linear classification of local singularities associated to the loss of normal hyperbolicity along $S$.

	\subsection{Non-normally hyperbolic singularities}
	\label{sub:loss_of_normal_hyperbolicity}
	
	As in continuous-time GSPT, local non-normally hyperbolic singularities can often be defined and classified on the `singular level', in our case, in terms of local conditions on the layer and reduced maps \eqref{eq:layer_map} and \eqref{eq:reduced_map} respectively. For our present purposes if suffices to consider conditions on the layer map, which can be derived using classical bifurcation theory, see e.g.~\cite{Kuznetsov2013}, since non-normally hyperbolic singularities correspond to bifurcations in the layer map if the problem is viewed in suitable (standard form) coordinates. 
	
	In the classification of codimension-$1$ non-normally hyperbolic singularities, there are three important cases; see Figure \ref{fig:classification}:
	\begin{itemize}
		\item \textit{Fold/Contact-type singularities} in the submanifold
		\[
		\mathcal C := \{ z \in S : \mu_1(z) = 1, |\mu_l(z)| \neq 1, l = 2, \ldots, n-k \}, 
		\]
		corresponding to a single real non-trivial multiplier $\mu_1(z)$ equal to $1$;
		\item \textit{Flip/Period-doubling-type singularities} in the submanifold
		\[
		\mathcal O_{\textup{f}} := \{ z \in S : \mu_1(z) = -1, |\mu_l(z)| \neq 1, l = 2, \ldots, n-k \}, 
		\]
		corresponding to a single real non-trivial multiplier $\mu_1(z)$ equal to $-1$;
		\item \textit{Neimark-Sacker/Torus-type singularities} in the submanifold
		\[
		\mathcal O_{\textup{ns}} := \{ z \in S : \mu_1(z) = \overline{\mu_2(z)} \in S^1 \setminus \{\pm1\}, |\mu_l(z)| \neq 1, l = 3, \ldots, n-k \}, 
		\]
		corresponding to a pair of complex conjugate non-trivial multipliers $\mu_{1,2}(z)$ lying on $S^1 \setminus \{\pm1\}$.
	\end{itemize}
	We refer to points in $\mathcal C$ as fold or contact-type singularities, since these are in certain sense the `least degenerate' points in $\mathcal C$. This will be made more precise below. However, we emphasise that our terminology includes singularities which would not typically be referred to as fold or contact singularities in $\mathcal C$. This includes pitchfork and transcritical singularities as important examples. Similar comments apply to distinctions between different singularity types in $\mathcal O_{\textup{f}}$ and $\mathcal O_{\textup{ns}}$. 
	
	\begin{figure}[t!]
		\centering
		\subfigure[Fold/contact.]{
			\includegraphics[width=0.3\textwidth]{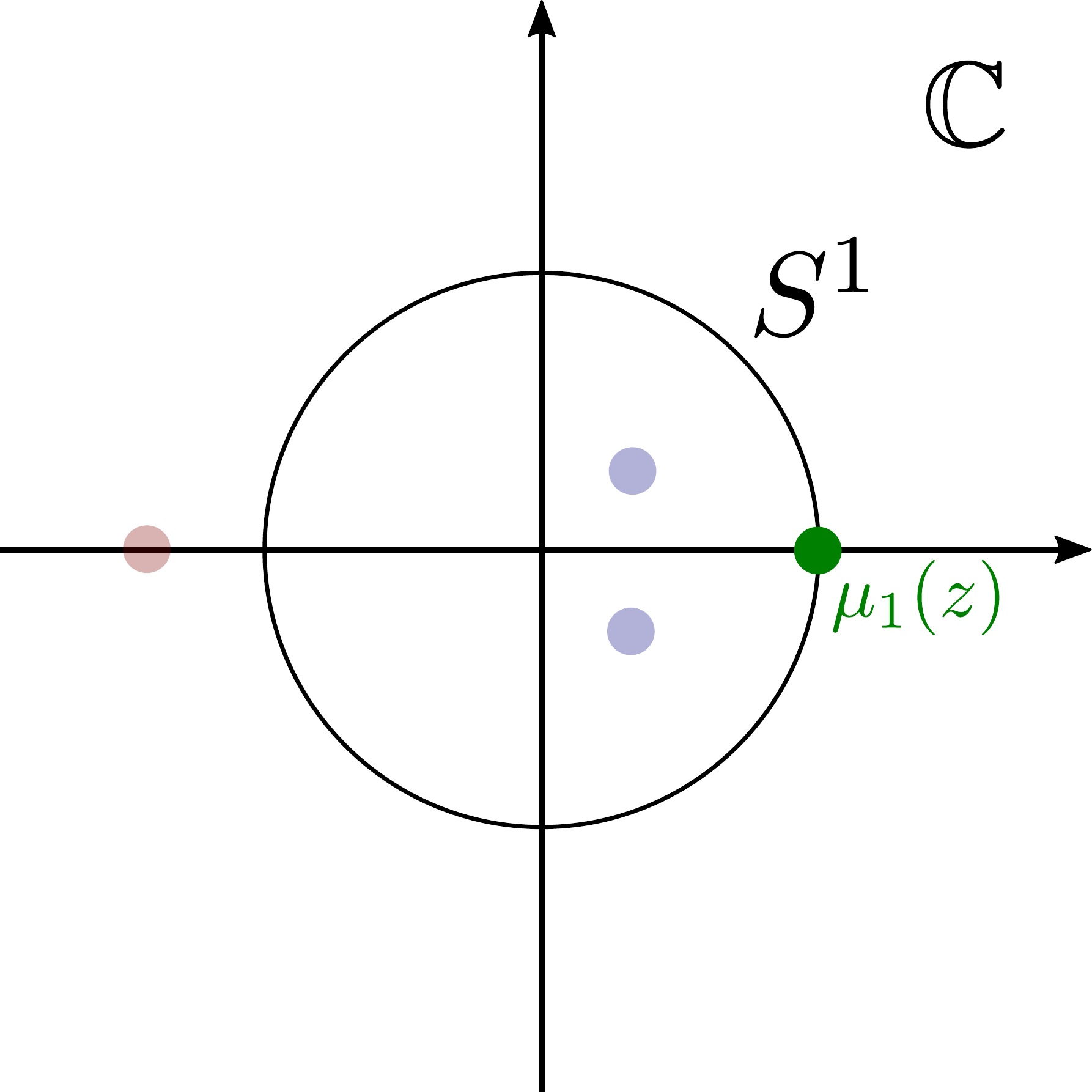}} \quad
		\subfigure[Flip/period-doubling.]{
			\includegraphics[width=0.3\textwidth]{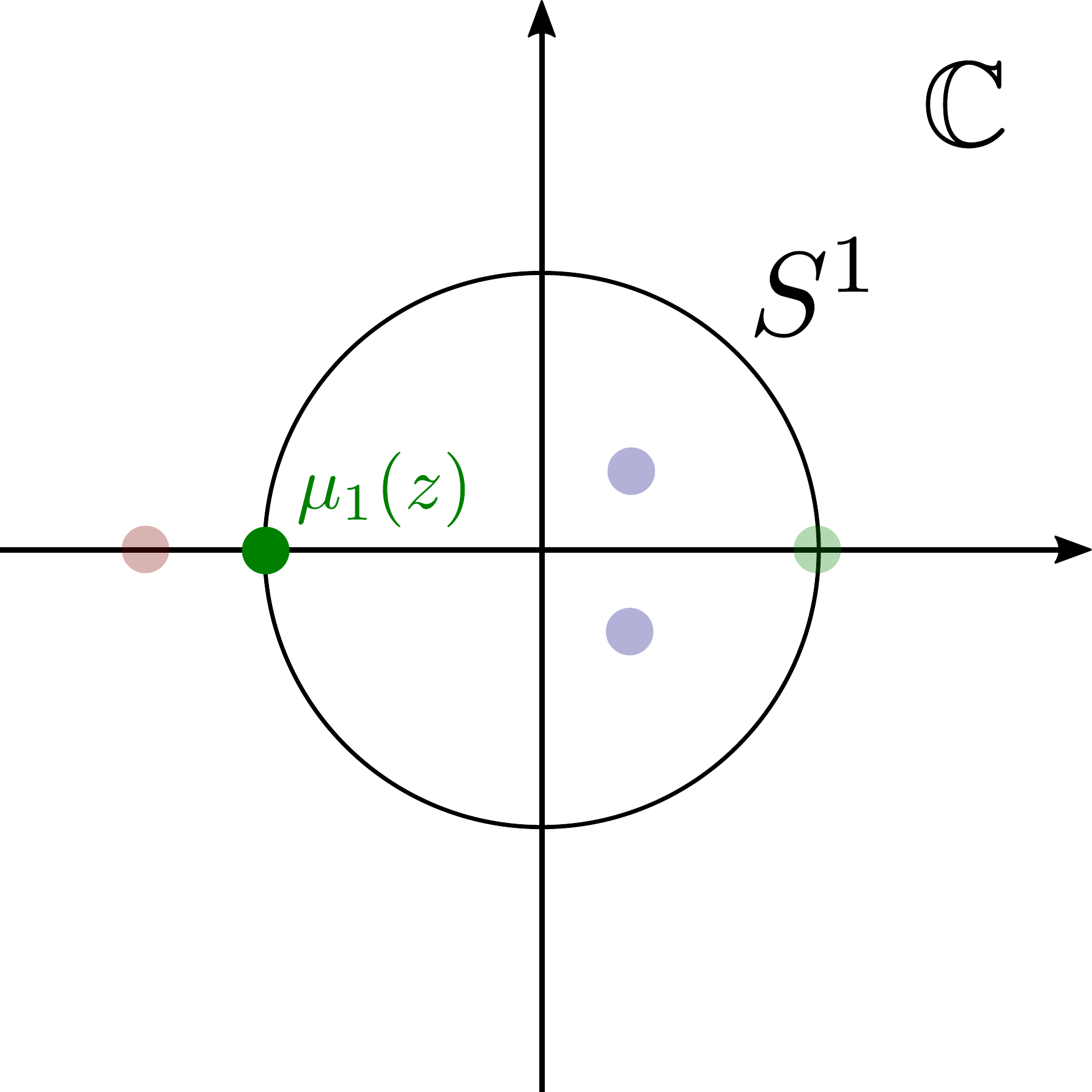}} \quad
		\subfigure[Neimark-Sacker/torus.]{
			\includegraphics[width=0.3\textwidth]{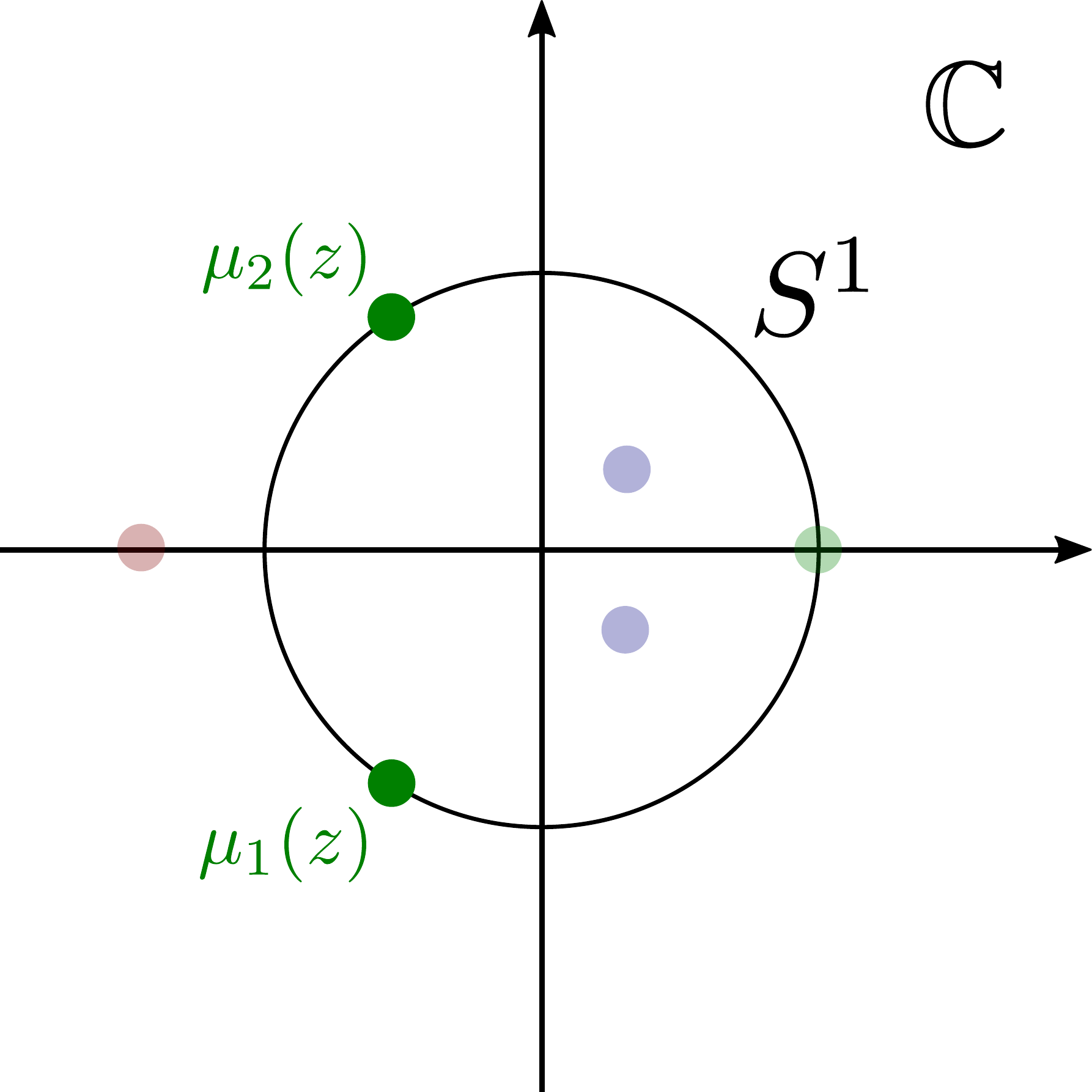}} 
		\caption{A linear classification of codimension-$1$ singularities of the map \eqref{eq:general_maps} can be given in terms of the distribution of multipliers of the matrix $\D H(z,0) = \mathbb I_n + N \D f(z)$ in the complex plane. Three different cases are distinguished according to the point at which a real non-trivial multiplier $\mu_1(z)$ or pair of complex conjugate non-trivial multipliers $\mu_{1,2}(z)$ (shown in green) intersect(s) $S^1$. (a): Fold/contact type point in $\mathcal C$ with $\mu_1(z) = 1$. (b): Flip/period-doubling type point in $\mathcal O_{\textup{f}}$ with $\mu_1(z) = -1$. (c): Neimark-Sacker/torus type point in $\mathcal C$ with complex conjugate multipliers $\mu_{1,2}(z) \in S^1 \setminus \{\pm 1\}$. In each case, we also sketch the $k$ trivial multipliers equal to $1$ in shaded green, as well as examples of stable and unstable multipliers which do not lie on $S^1$ (shaded blue and red respectively).}
		\label{fig:classification}
	\end{figure}
	
	For analytical purposes, it is often useful to group the oscillatory singularities in $\mathcal O_{\textup{f}}$ and $\mathcal O_{\textup{ns}}$ together. In fact, points in $\mathcal O_{\textup{f}} \cup \mathcal O_{\textup{ns}}$ are sometimes referred to as \textit{oscillatory} singularities \cite{Baesens1995}. Multiple authors have shown that oscillatory singularities are associated with $O(1)$ \textit{delayed stability loss} if suitable non-degeneracy conditions are satisfied and the original map is analytic (or more precisely, Gevrey-1) \cite{Baesens1991,Baesens1995,Fruchard2003,Fruchard2009,Neishtadt1996,Neishtadt1987}. Similarly to delayed Hopf bifurcation in the fast-slow ODE setting, the delay effect is lost if the map is only $C^r$ (including $r = \infty$), or if noise is introduced; see e.g.~\cite{Baer1989,Neishtadt1988,Neishtadt1987}. In either case, existing analyses of these problems depend crucially on the local invertibility of the matrix $\D fN$. Generically, however, the matrix $\D fN$ is not invertible at points in $\mathcal C$, which necessitates the need for a different approach. The remainder of this work is devoted to the map \eqref{eq:general_maps} near contact-type singularities in $\mathcal C$. 
	
	\begin{remark}
		In the present context, the `codimension' of a singularity is equal to the codimension of the submanifold of $S$ along which the relevant defining condition is satisfied. The submanifolds $\mathcal C$, $\mathcal O_{\textup{f}}$ and $\mathcal O_{\textup{ns}}$ are generically codimension-$1$ in $S$, i.e.~they are $(k-1)$-dimensional.
	\end{remark}
	
	\begin{remark}
		\label{rem:blow-up}
		See \cite{Arcidiacono2019,Engel2020b,Engel2019,Engel2022,Nipp2013,Nipp2017,Nipp2009} for detailed geometric analyses of planar fast-slow maps with singularities in class $\mathcal C$. These works demonstrate the applicability of the \textit{geometric blow-up method} in order to study (in these cases Euler) discretizations of fast-slow ODE systems with non-normally hyperbolic singularities. This approach relies on non-trivial scaling properties of the step-size parameter associated to the discretization, and does not extend directly for the study of general maps, i.e.~it does not apply to the study of maps which cannot be obtained via discretization.
	\end{remark}

	\subsection{Formal embeddings for unipotent singularities}
	\label{sub:existence_of_formal_embeddings_for_nilpotent_singularities}
	
	In this section we derive formal embedding theorems which can be used to approximate the dynamics of \eqref{eq:general_maps} near \textit{unipotent points}, which form an important subset of singularities in $\mathcal C$. Such problems are characterised by the fact that the linearisation of the layer map \eqref{eq:layer_map} along $S$, as given by \eqref{eq:layer_linearisation}, is `close' to the identity map. More precisely, we consider general fast-slow maps \eqref{eq:general_maps} under Assumptions \ref{ass:fast-slow}, \ref{ass:factorisation} and \ref{ass:invertibility}, locally near a point $z \in S$ for which the linearisation \eqref{eq:layer_linearisation} is unipotent, i.e.~for which the matrix $N \D f(z)$ is nilpotent with $(N \D f(z))^l = \mathbb O_{n,n}$ for some integer $l \geq 2$.
	
	\begin{remark}
		\label{rem:uniponent_linearisation}
		Generically, the Jacobian matrix \eqref{eq:layer_linearisation} appearing in the linearisation of the layer map associated to a fast-slow map \eqref{eq:general_maps} will not be unipotent at a point in $\mathcal C$ if $\textup{codim} (S) = n - k \geq 2$. This follows from the presence of non-zero eigenvalues of the matrix $\textup{D} f N$, which implies that $N \textup{D} f$ has non-zero eigenvalues, and therefore that $N \textup{D} f$ is not nilpotent (nilpotent matrices have only zero eigenvalues). Nevertheless, as we shall see in Section \ref{sub:center_manifold_reduction} below, the study of dynamics near singularities in $\mathcal C$ can often be reduced to the study of singularities in fast-slow maps with a codimension-$1$ critical manifold after center manifold reduction. In this case, every singularity in $\mathcal C$ has a corresponding linearisation with a unipotent Jacobian matrix \eqref{eq:layer_linearisation}.
	\end{remark}
	
	Our first result reduces computational difficulty by establishing an equivalence between nilpotency of the $n \times n$ square matrix $N\D f|S$ and the smaller $(n-k) \times (n-k)$ square matrix $\D fN|_S$. It also shows that nilpotency of $N(z)\D f(z)$ implies that $z \in \mathcal C$, i.e.~that nilpotency of $N(z)\D f(z)$ is correlated with the existence of at least one non-trivial multiplier with $\mu_j(z) = 1$ at a non-normally hyperbolic contact point.
	
	\begin{lemma}
		\label{lem:niplotency_equivalence}
		Let $z \in S$. The $n \times n$ matrix $\textup{D}(Nf)(z) = N\textup{D} f(z)$ is nilpotent of index $l+1$ if and only if the $(n-k) \times (n-k)$ matrix $\textup{D} fN(z)$ is nilpotent of index $l$, where $l \in \mathbb N_+$. In particular, $\mu_j(z) = 1$ for at least one $j \in \{1,\ldots,n-k\}$. 
	\end{lemma}
	
	\begin{proof}
		It follows from the proof of \cite[Prop.~2.8]{Jelbart2022a} that local coordinates can be chosen near $S$ such that
		\[
		N\D f|_S = 
		\begin{pmatrix}
			\mathbb O_{k,k} & \tilde N^{x} \\
			\mathbb O_{n-k,k} & \D fN
		\end{pmatrix}
		\bigg|_S ,
		\]
		where $\tilde N^x(z)$ is a $k \times (n-k)$ matrix with full column rank by Assumption \ref{ass:factorisation}. Direct matrix multiplication yields
		\[
		(N\D f)^l|_S =
		\begin{pmatrix}
			\mathbb O_{k,k} & \tilde N^{x} (\D fN)^{l-1} \\
			\mathbb O_{n-k,k} & (\D fN)^l
		\end{pmatrix} \bigg|_S ,
		\]
		for all $l \in \mathbb N_+$. It follows that $(\D fN)^l|_S = \mathbb O_{n-k,n-k} \implies (N\D f)^{l+1}|_S = \mathbb O_{n,n}$. The converse $(N\D f)^{l+1}|_S = \mathbb O_{n,n} \implies (\D fN)^l|_S = \mathbb O_{n-k,n-k}$ follows from $\tilde N^x (\D fN)^l|_S = \mathbb O_{k,n-k}$, since $\tilde N^x(z)$ has full column rank.
		
		In order to see that $\mu_j(z)=1$ for some $j \in \{1,\ldots,n-k\}$ we recall that the non-trivial multipliers have the form $\mu_j(z) = 1 + \lambda_j(z)$, where $\lambda_j(z)$ are eigenvalues of $\D fN(z)$. Since $\D fN(z)$ is nilpotent we have $\lambda_j(z) = 0$ for at least one $j \in \{1,\ldots,n-k\}$, thereby implying the result.
	\end{proof}
	
	We are now in a position to state our main result, namely, a formal embedding theorem which can be used to approximate the dynamics of the map  \eqref{eq:general_maps} in a neighbourhood of a unipotent singularity by the time-1 map of a vector field with known properties. As before, our primary tool for proving the existence of a formal embedding near such points $z \in \mathcal C \subset S$ is the Takens embedding theorem; we refer 
	again to Appendix \ref{app:Takens_embedding_theorem} for details.
	
	\begin{thm}
		\label{thm:embedding_nilpotent}
		Consider the map \eqref{eq:general_maps} denoted by $z \mapsto H(z,\eps)$ under Assumptions \ref{ass:fast-slow} and \ref{ass:factorisation}. Assume additionally that $0 \in \mathcal C \subseteq S$ and that the $(n-k) \times (n-k)$ square matrix $\textup{D} f N(0)$ is nilpotent with index $l \in \mathbb N_+$. Then there exists a neighbourhood $\mathcal U \ni 0$ and an $\eps_0 > 0$ such that for all $(z,\eps) \in \mathcal U \times [0,\eps_0)$ and $m \in \{ 1 , \ldots, r\}$ we have
		\begin{equation}
			\label{eq:Takens_embedding_nilpotent_map}
			j^m H(z,\eps) = j^m \Phi^1_V(z,\eps) ,
		\end{equation}
		where $\Phi^1_V(z,\eps)$ denotes the time-$1$ map induced by the flow of an $\eps$-family of vector fields $V(z,\eps)$ on $\R^n$ satisfying
		\begin{equation}
			\label{eq:ODE_embedding}
			j^r V(z,\eps) = \widehat N(z) j^r f(z) + \eps \widehat G(z,\eps) ,
		\end{equation}
		where $\widehat N(z)$ is a $C^r$-smooth $n \times (n-k)$ matrix with full column rank, $\widehat G(z,\eps)$ is a $C^r$-smooth column vector, and $j^r f$ is the $r$-jet associated to the Taylor expansion of $f(z)$ about $z=0$. We also have that
		\begin{equation}
			\label{eq:V_lin}
			\widehat N(0) = N(0) , \qquad \widehat G(0,0) = G(0,0) .
		\end{equation}
		The vector field \eqref{eq:ODE_embedding}
		is fast-slow with a $k$-dimensional critical manifold
		\begin{equation}
			\label{eq:jrS}
			j^r S := \left\{ z \in \R^n : j^r f(z) = 0 \right\} 
		\end{equation}
		which is $C^r$-close to $S$ on $\mathcal U$, i.e.~
		\begin{equation}
			\label{eq:dist_critical_manifolds}
			\textup{dist}(S, j^r S) := \sup_{z \in \mathcal U} | f(z) - j^r f(z) | =  o(|z|^r) .
		\end{equation}
		%
	\end{thm}
	
	\begin{proof}
		We want to apply the version of Takens' embedding theorem provided in Corollary \ref{cor:Takens}. In order to do so, we first append the map \eqref{eq:general_maps} with the trivial map $\eps \to \eps$ and consider the Taylor expansion about $(z,\eps) = (0,0)$, i.e.~
		\begin{equation}
			\label{eq:extended_expansion}
			\begin{pmatrix}
				z \\
				\eps
			\end{pmatrix}
			\mapsto
			\begin{pmatrix}
				H(z,\eps) \\
				\eps
			\end{pmatrix}
			= (\mathbb I_{n+1} + \Lambda) 
			\begin{pmatrix}
				z \\
				\eps
			\end{pmatrix}
			+ H^{(2)}(z,\eps) + H^{(3)}(z,\eps) + \ldots + H^{(r)}(z,\eps) + o(|(z,\eps)|^r) ,
		\end{equation}
		where
		\[
		\Lambda :=
		\begin{pmatrix}
			N\D f(0) & G(0,0) \\
			\mathbb O_{1,n} & 0
		\end{pmatrix}
		\]
		and the functions $H^{(l)}(z,\eps)$ are homogeneous polynomials of degree $l \geq 2$. 
		Repeated matrix multiplication leads to the expression
		\[
		\Lambda^{l+2} = 
		\begin{pmatrix}
			(N\D f)^{l+2}(0) & (N\D f)^{l+1}(0) G(0,0) \\
			\mathbb O_{1,n} & 0
		\end{pmatrix}
		= \mathbb O_{n+1,n+1} ,
		\]
		where we used the fact that $(\D fN)(0)$ is nilpotent with index $l$ 
		so that by Lemma \ref{lem:niplotency_equivalence}, $(N\D f)^{l+1}(0) = \mathbb O_{n,n}$. Thus $\Lambda$ is also nilpotent, with index $l+2$ if $(N\D f)^{l}(0) G(0,0) \neq \mathbb O_{n,1}$ and index $l+1$ if $(N\D f)^{l}(0) G(0,0) \neq \mathbb O_{n,1}$. Moreover, the linearisation $\mathbb I_{n+1} + \Lambda$ has $n+1$ eigenvalues equal to $1$. Therefore, the inverse function theorem applies, and guarantees that the map \eqref{eq:extended_expansion} is a local diffeomorphism.
		
		The preceding arguments show that Corollary \ref{cor:Takens} applies. Applying it yields the approximation property \eqref{eq:Takens_embedding_nilpotent_map} for a uniquely determined family of formal vector fields $V(z,\eps)$ satisfying
		\begin{equation}
			\label{eq:truncated_VF}
			j^r V(z,\eps) = (N\D f)(0) z + \eps G(0,0) + V^{(2)}(z,\eps) + V^{(3)}(z,\eps) + \ldots V^{(r)}(z,\eps) ,
		\end{equation}
		where the $V^{(l)}(z,\eps)$ are homogeneous polynomials of degree $l = 2, \ldots, r$.
		
		We now show that the truncated formal vector field \eqref{eq:truncated_VF} is fast-slow with general form \eqref{eq:ODE_embedding}. Notice that (i) $V(0,0) = 0$, and (ii) 
		\[
		j^r H(z,0) = j^r \Phi^1_V(z,0) = z ,
		\]
		locally for all $z \in j^r S$. These two facts imply that $j^r V(z,0) = 0$ for all $z \in j^r S$, i.e.~that the zero sets (critical manifolds) of the truncated map $j^r H(z,0)$ and the truncated vector field $j^r V(z,0)$ coincide in $\mathcal U$; both are given by \eqref{eq:jrS}. 
		Therefore, Hadamard's lemma (see again Appendix \ref{app:Hadamard}) implies that
		\[
		j^r V(z,0) = 
		\widehat N(z) j^r f(z) ,
		\]
		for an $n \times (n-k)$ matrix $\widehat N(z)$ with full column rank, as required. This shows that $j^r V(z,\eps)$ is in the form \eqref{eq:ODE_embedding}, and the linear part of \eqref{eq:truncated_VF} implies the two equalities in \eqref{eq:V_lin}. Equation \eqref{eq:dist_critical_manifolds} follows directly from the definition of $r$-jets and the truncation operator $j^r$ in Definition \ref{def:jl_truncation}.
		%
		%
		%
	\end{proof}
	
	Theorem \ref{thm:embedding_nilpotent} provides a means for approximating the dynamics of the map \eqref{eq:general_maps} near a unipotent singularity in $\mathcal C$, via the time-1 map induced by the flow of an $\eps$-dependent family of fast-slow formal vector fields $V(z,\eps)$ in the same dimension. The linear part of $V(z,\eps)$ is determined by \eqref{eq:ODE_embedding} and \eqref{eq:V_lin}, and higher order terms in the series expansion can be determined in a formal but systematic fashion; we refer to Appendices \ref{app:reduced_map} and \ref{app:partials_equality} where this procedure is carried out in detail for particular proofs. Similarly to the problem of explicit determination of the transformations in Theorem \ref{thm:nh_approximation}, this can be a computationally demanding task. Fortunately, an explicit form for the approximating vector field $V(z,\eps)$ is not needed in many cases, since we obtain a sufficient amount of information about the qualitative dynamics without an explicit representation. This will be demonstrated with numerous applications in Sections \ref{sec:2d_applications} and \ref{sec:regular_contact_points} below. At present, it suffices to emphasise and reiterate the following points:
	\begin{enumerate}
		\item[(i)] The truncated vector field $j^r V(z,\eps)$ is fast-slow in the general form \eqref{eq:ODE_embedding};
		\item[(ii)] The critical manifold $j^rS$ of the map $j^r H(z,\eps)$ coincides with the critical manifold of the vector field $j^r V(z,\eps)$, and $j^r S$ is $C^r$-close to $S$, the critical manifold of the original map $H(z,\eps)$;
		\item[(iii)] The linear part of $j^r V(z,\eps)$ is known explicitly, due to \eqref{eq:V_lin}.
	\end{enumerate}
	Properties (i)-(iii) imply that in many cases, the local defining conditions on the layer map at unipotent, non-normally hyperbolic singularities in $\mathcal C \subset S$ are satisfied if and only if a set of corresponding conditions are satisfied in the approximating vector field \eqref{eq:ODE_embedding}. 
	In particular, if the map \eqref{eq:general_maps} has a unipotent, non-normally hyperbolic singularity, then the vector field \eqref{eq:ODE_embedding} has a nilpotent, non-normally hyperbolic singularity. Moreover, due to the $C^r$-closeness of the critical manifolds $j^rS$ and $S$ and the close relationship between the reduced map and the reduced vector field described in Theorem \ref{thm:reduced_map}, the satisfaction of defining conditions for local unipotent singularities in $\mathcal C \subset S$ which are formulated in terms of $N$, $f$, $G$ and their partial derivatives with respect to $z$, are expected to imply the satisfaction of analogous conditions for vector field \eqref{eq:ODE_embedding}. Thus, unipotent singularities of the map \eqref{eq:general_maps} will typically imply nilpotent singularities of the corresponding type in the vector field \eqref{eq:ODE_embedding}. This can be a powerful tool in practice, since it 
	allows for the indirect study of the map \eqref{eq:general_maps} using methods and techniques which are only developed or applicable in the context of fast-slow ODEs.
	
	\begin{remark}
		\label{rem:embedding}
		Theorem \ref{thm:embedding_nilpotent} also applies in the $C^\infty$ case, i.e.~for $r = \infty$. In this case, equation \eqref{eq:Takens_embedding_nilpotent_map} implies that the Taylor expansions of $H(z,\eps)$ and $\Phi_V^1(z,\eps)$ about $(z,\eps) = (0,0)$ coincide. It is important to stress, however, that this does not imply an exact embedding
		\begin{equation}
			\label{eq:exact_embedding}
			H(z,\eps) = \Phi_V^1(z,\eps) .
		\end{equation}
		Rather, we have that
		\begin{equation}
			\label{eq:formal_embedding}
			H(z,\eps) = \Phi_V^1(z,\eps) + R(z,\eps) ,
		\end{equation}
		for an error term $R(z,\eps)$ which is flat (beyond all orders) in $(z,\eps)$ (if $r < \infty$ then $R(z,\eps) = o(|(z,\eps)|^r)$). An important example for which the error term cannot be removed arises in the case of saddle-node bifurcation in $1$-dimensional maps \cite{Ilyashenko1993}; see also Remark \ref{rem:Ilyashenko} below. Terminologically, we distinguish ``exact embeddings" \eqref{eq:exact_embedding} from ``formal embeddings" \eqref{eq:formal_embedding}, and we say that the latter allow for an approximation of the map by a flow. Similar terminology can be found in e.g.~\cite{Chen1965,Kuznetsov2013,Kuznetsov2019}.
	\end{remark}
	
	
	
	
	\section{Fold, transcritical and pitchfork points in 2 dimensions}
	\label{sec:2d_applications}
	
	In this section we consider dynamics near three different types of unipotent singularities in two-dimensional fast-slow maps, i.e.~in the lowest dimensional setting. This will help to demonstrate the use of Theorem \ref{thm:embedding_nilpotent} in practice. For simplicity, we consider maps in fast-slow standard form
	\begin{equation}
		\label{eq:standard_form_2d}
		\begin{pmatrix}
			x \\
			y
		\end{pmatrix}
		\mapsto H(x,y,\eps) = 
		\begin{pmatrix}
			x \\
			y
		\end{pmatrix}
		+ 
		\begin{pmatrix}
			\tilde f(x,y,\eps) \\
			\eps \tilde g(x,y,\eps)
		\end{pmatrix} ,
	\end{equation}
	in line with the earlier notation in \eqref{eq:standard_form}. We shall assume that $H$ is $C^r$-smooth with $r \geq 3$, and that $H(0,0,0) = (0,0)^\transpose$, i.e.~$(0,0) \in S$. Since $\textup{codim}(S) = 1$, every point in $\mathcal C$ is a unipotent singularity; recall Remark \ref{rem:uniponent_linearisation}. Thus, Theorem \ref{thm:embedding_nilpotent} can be applied near any point in $\mathcal C$. We return to the problem of approximation in higher dimensions, in which case $\mathcal C$ contains singularities that are not unipotent, in Section \ref{sec:regular_contact_points}.
	
	\subsection{Definitions and singular geometry}
	
	We consider non-normally hyperbolic singularities of regular fold, transcritical and pitchfork type. We begin with definitions and a description of the limiting dynamics and geometry for $\eps = 0$ in each case, starting with the fold.

	\subsubsection{Regular fold points}
	
	Regular fold points can be defined as follows.
	
	\begin{definition}
		\label{def:fold}
		The fast-slow map \eqref{eq:standard_form_2d} has a regular fold point at $(x,y) = (0,0) \in S$ if
		\begin{equation}
			\label{eq:fold_conds}
			\tilde f(0,0,0) = 0, \qquad 
			\frac{\partial \tilde f}{\partial x}(0,0,0) = 0 ,
		\end{equation}
		and
		\begin{equation}
			\label{eq:fold_genericity_conds}
			\frac{\partial^2 \tilde f}{\partial x^2}(0,0,0) \neq 0 , \qquad 
			\frac{\partial \tilde f}{\partial y}(0,0,0) \neq 0 , \qquad 
			\tilde g(0,0,0) \neq 0 .
		\end{equation}
	\end{definition}	
	
	Definition \ref{def:fold} is analogous to the definition of a regular fold point in planar fast-slow ODEs; see \cite{Krupa2001a,Kuehn2015}. The local geometry implied by the defining conditions \eqref{eq:fold_conds} and \eqref{eq:fold_genericity_conds} is sketched in Figure \ref{fig:2d_limit}(a). The critical manifold $S$ is locally quadratic at the fold point $(0,0)$, which separates two normally hyperbolic branches. We sketch the case with
	\begin{equation}
		\label{eq:fold_orientation}
		\frac{\partial^2 \tilde f}{\partial x^2}(0,0,0) > 0 , \qquad 
		\frac{\partial \tilde f}{\partial y}(0,0,0) < 0 , \qquad 
		\tilde g(0,0,0) < 0 .
	\end{equation}
	This choice fixes the orientation and stability of the branches, as well as the orientation of the reduced flow defined by the reduced vector field on $S \setminus \{(0,0)\}$. There is a normally hyperbolic and attracting (repelling) branch $S^a$ ($S^r$) in $\{x < 0\}$ ($\{x > 0\}$), and the reduced flow is oriented towards the fold point; see Figure \ref{fig:2d_limit}(a).
	
	\begin{remark}
		Due to the close relationship between the slow map \eqref{eq:slow_map} and the reduced vector field \eqref{eq:reduced_vf} described in Theorem \ref{thm:reduced_map}, we permit ourselves to speak of a `reduced flow' in the context of fast-slow maps \rev{(this can also be justified using the convergence property in \eqref{eq:convergence})}. The advantage is that the reduced vector field \eqref{eq:reduced_vf} retains important information about the dynamics on $S$ when $\eps = 0$, despite the fact that the slow map \eqref{eq:slow_map} is trivial for $\eps = 0$ (similarly to the reduced vector field on the fast time-scale, recall \eqref{eq:reduced_vf_fast}). We shall also use the reduced vector field to provide a simpler representation of the slow dynamics as $\eps \to 0$ in figures, for example in Figure \ref{fig:2d_limit}.
	\end{remark}
	
	\begin{figure}[t!]
		\centering
		\subfigure[Regular fold.]{\includegraphics[width=.47\textwidth]{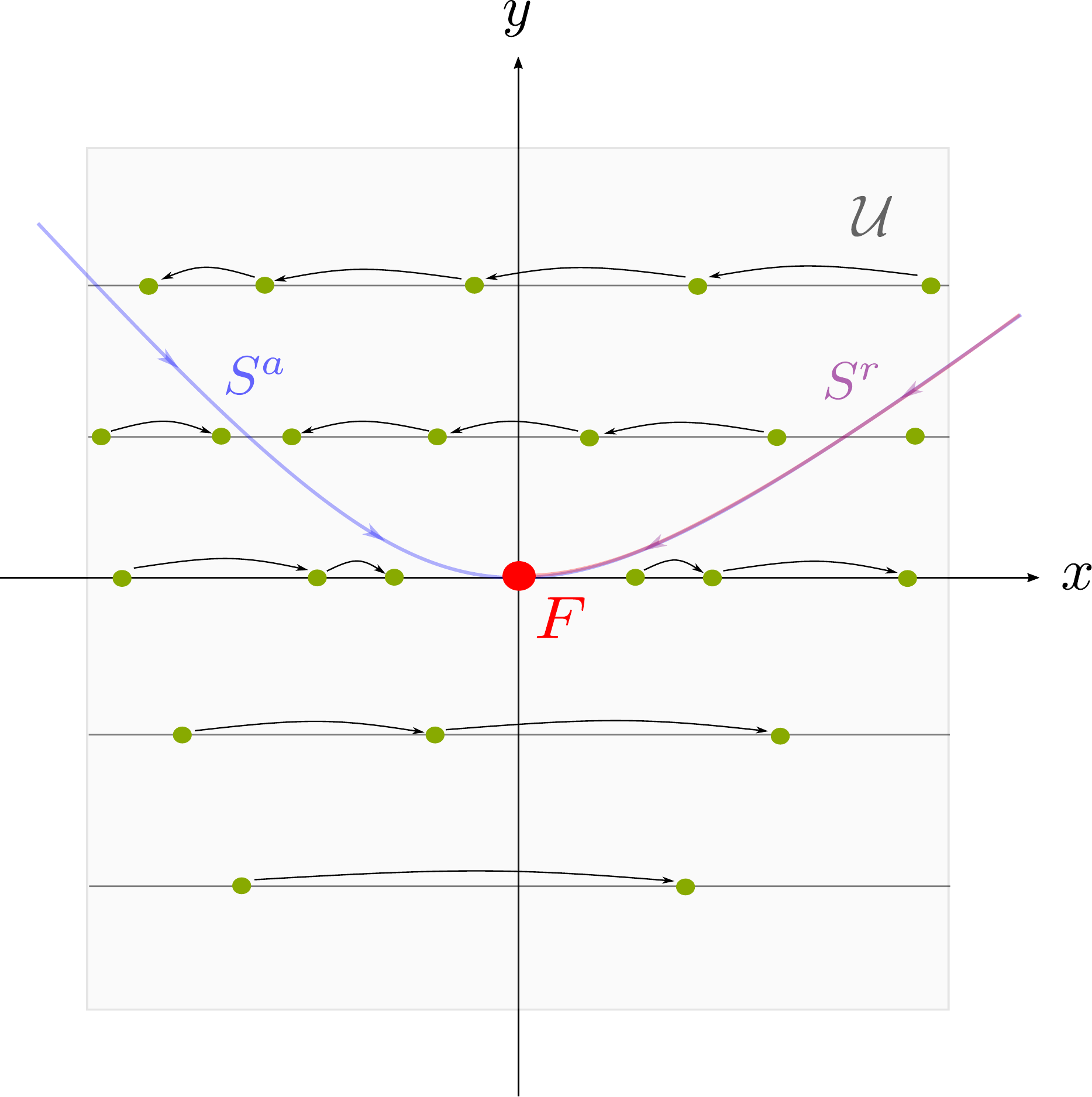}}
		\quad
		\subfigure[Transcritical.]{\includegraphics[width=.47\textwidth]{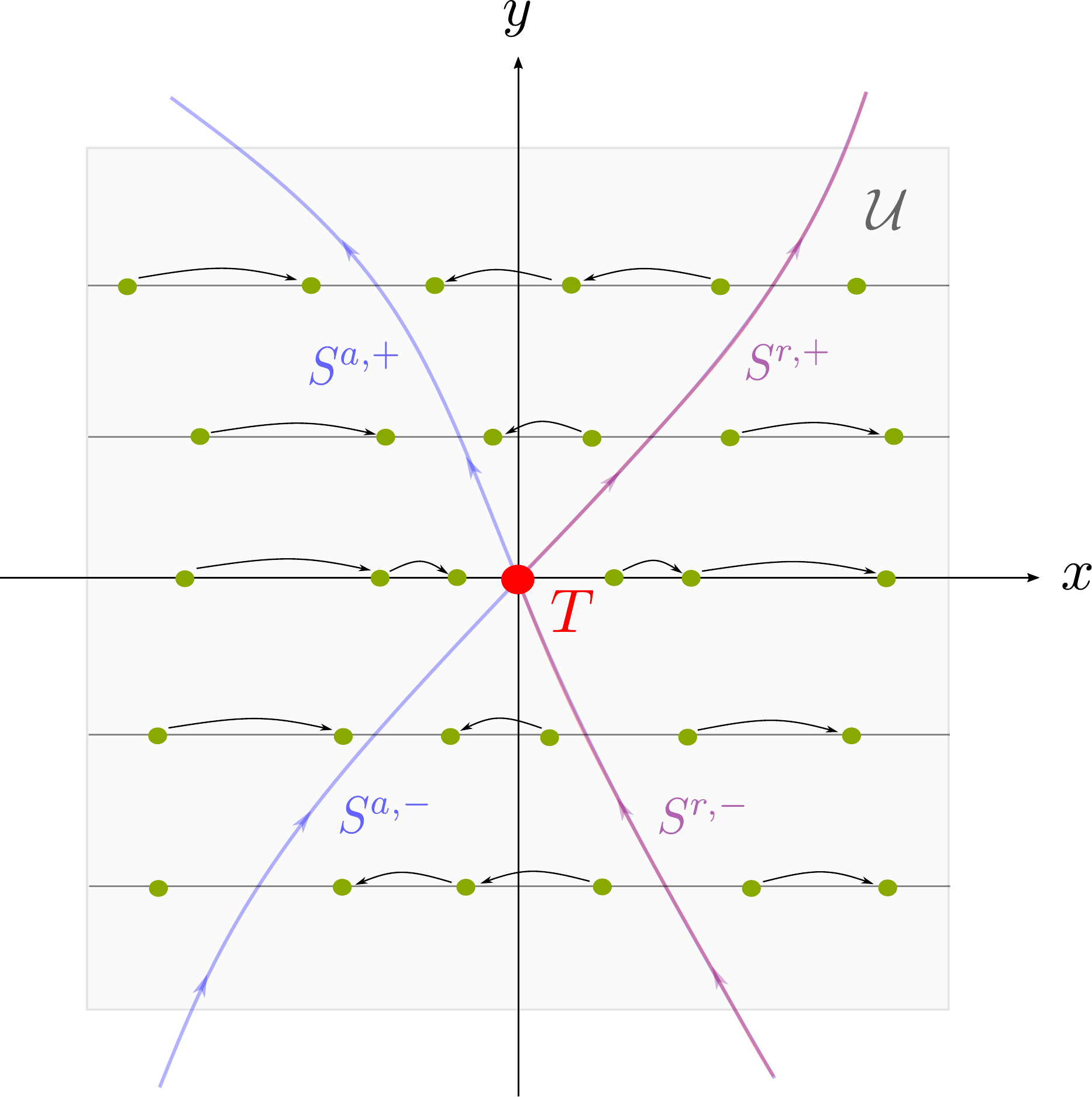}}
		\\
		\subfigure[Pitchfork.]{\includegraphics[width=.47\textwidth]{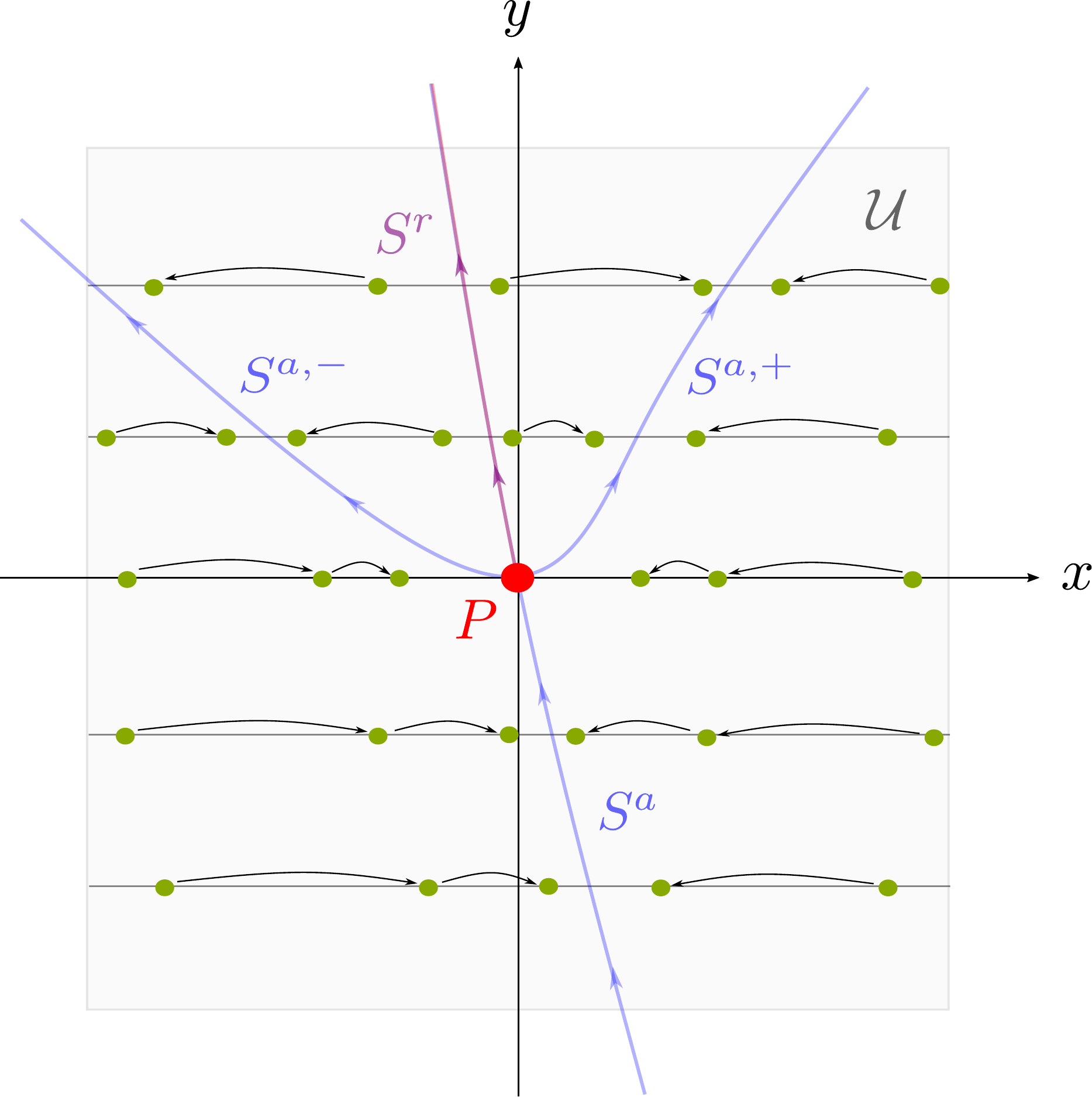}}
		\caption{Singular geometry and dynamics in a neighbourhood $\mathcal U$ (shaded grey) about a non-hyperbolic point of (a) regular fold type, (b) transcritical type, and (c) pitchfork type. Shaded blue (purple) curves indicate attracting (repelling) branches of the critical manifold $S$, and the regular fold, transcritical and pitchfork points are indicated in red and denoted by $F$, $T$ and $P$ in Figures (a), (b) and (c) respectively. The orientation of the reduced vector field on $S$ is indicated with single arrows, and sample iterates of the layer map along fast fibers are sketched in green, with black arrows to indicate the direction of forward iteration.}
		\label{fig:2d_limit}
	\end{figure}
	
	\subsubsection{Transcritical points}
	
	We now define transcritical points.
	
	\begin{definition}
		\label{def:transcritical}
		The fast-slow map \eqref{eq:standard_form_2d} has a transcritical singularity at $(x,y) = (0,0) \in S$ if
		\begin{equation}
			\label{eq:transcritical_conds}
			\tilde f(0,0,0) = 0, \qquad 
			\frac{\partial \tilde f}{\partial x}(0,0,0) = 0 , \qquad
			\frac{\partial \tilde f}{\partial y}(0,0,0) = 0 ,
		\end{equation}
		and
		\begin{equation}
			\label{eq:transcritical_genericity_conds}
			\frac{\partial^2 \tilde f}{\partial x^2}(0,0,0) \neq 0 , \qquad 
			\det
			\begin{pmatrix}
				\frac{\partial^2 \tilde f}{\partial x^2}(0,0,0) & \frac{\partial^2 \tilde f}{\partial x \partial y}(0,0,0) \\
				\frac{\partial^2 \tilde f}{\partial y \partial x}(0,0,0)  & \frac{\partial^2 \tilde f}{\partial y^2}(0,0,0) 
			\end{pmatrix} < 0 , 
			\qquad 
			\tilde g(0,0,0) \neq 0 .
		\end{equation}
	\end{definition}
	
	Definition \ref{def:transcritical} is directly analogous to the definition of a transcritical singularity in planar fast-slow ODEs; see e.g.~\cite{Krupa2001c}. The local geometry implied by the defining conditions \eqref{eq:transcritical_conds} and \eqref{eq:transcritical_genericity_conds} is sketched in Figure \ref{fig:2d_limit}(b). The critical manifold $S$ is given by the union of two curves which intersect transversally at $(0,0)$ (the transcritical point). There are four normally hyperbolic branches which emanate from (but do not include) $(0,0)$. We sketch the case with
	\begin{equation}
		\label{eq:transcritical_orientation}
		\frac{\partial^2 \tilde f}{\partial x^2}(0,0,0) > 0 , \qquad 
		\tilde g(0,0,0) > 0 .
	\end{equation}
	The first condition implies that the two branches on the left (right) are attracting (repelling). We denote the upper and lower attracting branches, shown in blue in Figure \ref{fig:2d_limit}(b), by $S^{a,+}$ and $S^{a,-}$ respectively. Upper and lower repelling branches are shown in purple and denoted by $S^{r,+}$ and $S^{r,-}$ respectively.

	\subsubsection{Pitchfork points}
	
	Finally, we define pitchfork points.
	
	\begin{definition}
		\label{def:pitchfork}
		The fast-slow map \eqref{eq:standard_form_2d} has a pitchfork singularity at $(x,y) = (0,0) \in S$ if
		\begin{equation}
			\label{eq:pitchfork_conds}
			\tilde f(0,0,0) = 0, \qquad 
			\frac{\partial \tilde f}{\partial x}(0,0,0) = 0 , \qquad
			\frac{\partial \tilde f}{\partial y}(0,0,0) = 0 , \qquad
			\frac{\partial^2 \tilde f}{\partial x^2}(0,0,0) = 0 ,
		\end{equation}
		and
		\begin{equation}
			\label{eq:pitchfork_genericity_conds}
			\frac{\partial^3 \tilde f}{\partial x^3}(0,0,0) \neq 0 , \qquad 
			\frac{\partial^2 \tilde f}{\partial x \partial y}(0,0,0) \neq 0 , \qquad 
			\tilde g(0,0,0) \neq 0 .
		\end{equation}
	\end{definition}	
	
	Definition \ref{def:pitchfork} is analogous to the definition of a pitchfork singularity in planar fast-slow ODEs; we refer again to \cite{Krupa2001c}. The local geometry implied by the defining conditions \eqref{eq:pitchfork_conds} and \eqref{eq:pitchfork_genericity_conds} is sketched in Figure \ref{fig:2d_limit}(c). The critical manifold $S$ is given by the union of two curves which intersect transversally at $(0,0)$ (the pitchfork point). There are four normally hyperbolic branches which emanate from (but do not include) $(0,0)$. We sketch the case with
	\begin{equation}
		\label{eq:pitchfork_orientation}
		\frac{\partial^2 \tilde f}{\partial x \partial y}(0,0,0) > 0 , \qquad 
		\frac{\partial^3 \tilde f}{\partial x^3}(0,0,0) < 0 .
	\end{equation}
	This choice fixes the orientation and stability of the branches, and the latter in particular restricts us to the supercritical case. The orientation of the slow dynamics can be fixed by choosing $\tilde g(0,0,0) > 0$ or $\tilde g(0,0,0) < 0$. We do not specify a choice at this point because we intend to state results for both cases. There is an attracting lower branch $S^a$ in $\{y < 0\}$, a repelling upper branch $S^r$ and two attracting outer branches $S^{a,\pm}$ in $\{y > 0\}$; see Figure \ref{fig:2d_limit}(c).

	\subsection{Formal embedding near fold, transcritical and pitchfork points}
	\label{sub:2d_embedding}
	
	In the following we state and prove a formal embedding theorem which allows for the local approximation of the map \eqref{eq:standard_form_2d} near a regular fold, transcritical or pitchfork point via the time-1 map induced by a planar system of fast-slow ODEs with a singularity of the corresponding type.
	
	\begin{proposition}
		\label{prop:2d_embedding}
		Assume that the map \eqref{eq:standard_form_2d} has a unipotent singularity at $(x,y) = (0,0) \in S$. Then there exists a neighbourhood $\mathcal U \ni (0,0)$ and an $\eps_0 > 0$ such that for all $(x,y,\eps) \in \mathcal U \times [0,\eps_0)$ and $l \in \{ 1 , \ldots, r\}$ we have
		\begin{equation}
			\label{eq:2d_embedding}
			j^l H(x,y,\eps)
			= j^l \Phi^1_V(x,y,\eps) ,
		\end{equation}
		where $\Phi^1_V(x,y,\eps)$ denotes the time-$1$ flow of an $\eps$-family vector fields $V(x,y,\eps)$ on $\R^2$. The truncated vector field $j^r V(x,y,\eps)$ is fast-slow in the standard form
		\begin{equation}
			\label{eq:2d_vf}
			j^r V(x,y,\eps) =
			\begin{pmatrix}
				\mathcal V_1(x,y,\eps) \\
				\eps \mathcal V_2(x,y,\eps)
			\end{pmatrix}
			=
			\begin{pmatrix}
				K(x,y) j^r \tilde f (x, y, 0) + O(\eps) \\
				\eps (\tilde g(0,0,0) + O(x, y, \eps) )
			\end{pmatrix} ,
		\end{equation}
		where the function $K(x,y)$ satisfies $K(0,0) = 1$. The vector field \eqref{eq:2d_vf} has a critical manifold
		\begin{equation}
			\label{eq:jrS_2d}
			j^r S
			= \left\{ (x,y) \in \mathcal U : j^r \tilde f(x,y,0) = 0 \right\} ,
		\end{equation}
		for which $(0,0) \in j^r S$ is a nilpotent singularity. In particular, $(0,0)$ is a
		\begin{enumerate}
			\item Regular fold point if $(0,0)$ is a regular fold point of the map \eqref{eq:standard_form_2d};
			\item Transcritical point if $(0,0)$ is a transcritical point of the map \eqref{eq:standard_form_2d};
			\item Pitchfork point if $(0,0)$ is a pitchfork point of the map \eqref{eq:standard_form_2d}.
		\end{enumerate}
	\end{proposition}
	
	\begin{proof}
		We want to apply Theorem \ref{thm:embedding_nilpotent}. In order to do so, we need to check Assumptions \ref{ass:fast-slow}, \ref{ass:factorisation} and \ref{ass:invertibility}, and show that $(0,0)$ is a unipotent singularity of \eqref{eq:standard_form_2d}.
		
		Following Remark \ref{rem:standard_form}, we write the \rev{map} \eqref{eq:standard_form_2d} in the general form \eqref{eq:general_maps} by defining
		\[
		N(x,y) = 
		\begin{pmatrix}
			1 \\
			0
		\end{pmatrix}
		, \qquad
		f(x,y) = \tilde f(x,y,0) , \qquad
		G(x,y,\eps) = 
		\begin{pmatrix}
			\eps^{-1} \left( \tilde f(x,y,\eps) - \tilde f(x,y,0) \right) \\
			\tilde g(0,0,0)  
		\end{pmatrix} .
		\]
		The point $(0,0)$ is a unipotent singularity if $\D fN(0,0)$ is nilpotent, which in this case ($n-k=1$) is equivalent to $(0,0) \in \mathcal C$. Using \eqref{eq:fold_conds}, \eqref{eq:transcritical_conds} or \eqref{eq:pitchfork_conds}, we have
		\[
		\D f N(0,0) = \frac{\partial \tilde f}{\partial x} (0,0,0) = 0 ,
		\]
		implying nilpotency. Assumptions \ref{ass:factorisation} and \ref{ass:invertibility} can be verified directly (the latter follows from the inverse function theorem), as can all of the conditions in Assumption \ref{ass:fast-slow} except for the requirement that $S$ can be regularly embedded as a submanifold in $\R^2$, which is not true if $\D f(0,0) = (0,0)$. Such situations arise at transcritical and pitchfork singularities due to a self-intersection at $(0,0)$ (for example). Fortunately, Theorem \ref{thm:embedding_nilpotent} applies regardless, since the proof does not rely on the property $\D f(0,0) \neq (0,0)$. After applying Theorem \ref{thm:embedding_nilpotent}, we obtain the existence of an $\eps$-dependent, $2$-dimensional vector field $V(x,y,\eps)$ whose time-1 map $\Phi^1_V(z,\eps)$ satisfies the approximation property in \eqref{eq:2d_embedding}.
		
		It follows from \eqref{eq:2d_embedding} that the truncated formal vector field $j^rV(x,y,\eps)$ takes the standard form
		\[
		j^r V(x,y,\eps) = 
		\begin{pmatrix}
			\mathcal V_1(x,y,\eps) \\
			\eps \mathcal V_2(x,y,\eps)
		\end{pmatrix} ,
		\]
		where the right-hand side is a formal power series about $(x,y,\eps) = (0,0,0)$, truncated at order $r$. It follows by Theorem \ref{thm:embedding_nilpotent} that the critical manifold is given by \eqref{eq:jrS_2d}. In order to derive the right-most expression in \eqref{eq:2d_vf}, note that Hadamard's lemma (Appendix \ref{app:Hadamard}) implies that
		\[
		\mathcal V_1(x,y,0) = K(x,y) \tilde f(x,y,0)
		\]
		for a smooth function $K : \mathcal U \to \R$, assuming that we choose the neighbourhood $\mathcal U \ni (0,0)$ in $\R^2$ sufficiently small. The right-hand side in \eqref{eq:2d_vf} follows from this, together with the fact that equation \eqref{eq:V_lin} implies
		\[
		K(0,0) = 1 , \qquad \mathcal V_2(0,0,0) = \tilde g(0,0,0) .
		\]
		The fact that $(0,0)$ is a nilpotent singularity on $j^r S$ follows from
		\[
		\frac{\partial \mathcal V_1}{\partial x}(0,0,0) =
		K(0,0) \frac{\partial \tilde f}{\partial x}(0,0,0) + \frac{\partial K}{\partial x}(0,0) \tilde f(0,0,0) = 0 .
		\]
		The right-most expression in \eqref{eq:2d_vf} can be used to directly verify that a regular fold, transcritical or pitchfork point at $(0,0)$ implies a regular fold, transcritical or pitchfork point of $j^r V(x,y,\eps)$ at $(0,0)$ respectively; one simply checks the conditions in Definitions \ref{def:fold}, \ref{def:transcritical} and \ref{def:pitchfork} for each case using $\mathcal V_1$ and $\mathcal V_2$ in place of $\tilde f$ and $\tilde g$, respectively.
	\end{proof}
	
	Thus, the local dynamics near unipotent, non-normally hyperbolic singularities of the map \eqref{eq:standard_form_2d} on $S$ can be formally approximated by the time-1 map of a planar fast-slow vector field. Moreover, if the singularity of the map is of regular fold, transcritical or pitchfork type, then the approximating vector field also has a regular fold, transcritical or pitchfork singularity respectively. Since the local dynamics near regular fold, transcritical and pitchfork singularities in planar fast-slow ODE systems are already well understood \cite{Krupa2001a,Krupa2001c}, Proposition \ref{prop:2d_embedding} can be used to approximate the dynamics of the map.
	
	\begin{remark}
		Every non-normally hyperbolic point in a planar fast-slow ODE system is nilpotent, due to the fact that there are no oscillatory singularities in planar fast-slow systems. Oscillatory singularities can occur in planar fast-slow maps, however, due to the possibility of flip/period-doubling singularities in $\mathcal O_{\textup{f}}$. Combining this observation with Proposition \ref{prop:2d_embedding}, we find that every nilpotent singularity of a planar fast-slow ODE system has a corresponding unipotent singularity in a planar fast-slow map, but the converse is not true, i.e.~not every non-normally hyperbolic singularity in a planar fast-slow map has a corresponding niplotent singularity in a planar fast-slow ODE system.
	\end{remark}
	
	\begin{remark}
		\label{rem:blow-up_2}
		Results for planar fast-slow maps with regular fold, transcritical and pitchfork singularities obtained after Euler discretizations of a planar fast-slow system with singularities of the corresponding types have been derived using a variant of the geometric blow-up method in \cite{Nipp2013,Nipp2009}, \cite{Engel2019} and \cite{Arcidiacono2019} respectively. As already noted in Remark \ref{rem:blow-up}, the step-size parameter associated to the discretization plays an important role in these analyses, and it is not presently known if similar methods can be extended to the study of general fast-slow maps without a step-size parameter, such as those considered in this work.
	\end{remark}
	
	\begin{remark}
		\label{rem:Ilyashenko}
		Consider \eqref{eq:2d_embedding} with a regular fold point at $(0,0)$. Proposition \ref{prop:2d_embedding} applies to this problem with $r = \infty$ if the original map $H(x,y,\eps)$ is $C^\infty$-smooth, but existing results for classical (parameter-dependent) fold bifurcations in 1-dimensional maps due to Ilyashenko \& Yakovenko \cite{Ilyashenko1993} imply that there is no exact embedding for the layer map $(x,y)^\transpose \mapsto H(x,y,0)$ which holds over an entire neighbourhood $\mathcal U \ni (0,0)$. A discrepancy arises because the extension of the exact embeddings about the two normally hyperbolic branches on either side of $(0,0)$ can be shown to disagree by an exponentially small amount in a particular sector of the $(x,y)$-plane containing $(0,0)$.
	\end{remark}
	
	\begin{remark}
		\label{rem:canards}
		\rev{Similarly to the continuous-time analysis in \cite{Krupa2001a}, we say that a planar fast-slow map
			\[
			\begin{pmatrix}
				x \\
				y
			\end{pmatrix}
			\mapsto H(x,y,\lambda,\eps) = 
			\begin{pmatrix}
				x \\
				y
			\end{pmatrix}
			+ 
			\begin{pmatrix}
				\tilde f(x,y,\lambda,\eps) \\
				\eps \tilde g(x,y,\lambda,\eps)
			\end{pmatrix} ,
			\] 
			defined similarly to \eqref{eq:standard_form_2d} except with dependence on an additional parameter $\lambda \in \R$, has a \textit{canard point} at $(0,0) \in S$ if the conditions in \eqref{eq:fold_conds} and \eqref{eq:fold_genericity_conds} are satisfied at $\lambda = 0$ except that
			\[
			\tilde g(0,0,0,0) = 0, \qquad 
			\frac{\partial \tilde g}{\partial \lambda}(0,0,0,0) \neq 0 .
			\]
			Proposition \ref{prop:2d_embedding} can be extended to apply to canard points, too, i.e.~Theorem \ref{thm:embedding_nilpotent} can be used in order to show that a planar fast-slow map with a canard point satisfies an approximation property of the form \eqref{eq:2d_embedding}, where the (truncated) approximating vector field is fast-slow with a canard point at $(0,0)$. Such an approximation result is sufficient to show that the local dynamics are `canard-like' insofar as an intersection of (extensions of) the attracting and repelling slow manifolds can be proven. However, the $C^r$-error which is introduced by such an approach means that the local description is too coarse to describe the relative positioning of exponentially close slow manifolds and the related `canard explosion', which occurs in an exponentially small parameter interval in $\lambda$ as $\eps \to 0$ \cite{deMaesschalck2021,Dumortier1996,Krupa2001b}. We leave the detailed treatment of canards for future work, and refer to \cite{Engel2023,Engel2020b,Engel2022,Fruchard2003} and the references therein for more on canard dynamics in fast-slow maps.}
	\end{remark}

	\subsection{Dynamics near regular fold, transcritical and pitchfork points}
	
	In the following we combine well-known results from the theory of planar fast-slow ODE systems with Proposition \ref{prop:2d_embedding} in order to describe the extension of slow manifolds through a neighbourhood of regular fold, transcritical and pitchfork points of the map \eqref{eq:standard_form_2d}. The relevant results in the fast-slow ODE setting have been derived using the geometric blow-up method in \cite{Krupa2001a,Krupa2001c}.

	\subsubsection{Regular fold dynamics}
	
	Theorem \ref{thm:slow_manifolds} implies that compact submanifolds of $S^a$ perturb to $O(\eps)$-close slow manifolds, which we denote by $S^a_\eps$ (we refer again to \cite{Jelbart2022a} for details). We are interested in the extension of $S_\eps^a$ through a sufficiently small but fixed and $\eps$-independent neighbourhood $\mathcal U \subset \R^2$ about the fold point. We assume without loss of generality that $\mathcal U = [-\rho,\rho]^2$ for a fixed constant $\rho > 0$, as in Figure \ref{fig:2d_limit}(a).
	
	\begin{corollary}
		\label{cor:fold}
		Consider the $C^r$-smooth fast-slow map \eqref{eq:standard_form_2d} with a regular fold point $(0,0) \in S$ satisfying \eqref{eq:fold_orientation}. For $\rho > 0$ sufficiently small there exists an $\eps_0 > 0$ such that for all $\eps \in (0,\eps_0)$ the extension of the attracting slow manifold $S_\eps^a$ leaves $\mathcal U$ at a point $(x,y) = (\rho,Y_{out}(\rho,\eps))$, where
		\[
		Y_{out}(\rho,\eps) = - c \eps^{2/3} + o \left( \eps \ln \eps, \rho^r \right) 
		\]
		for a constant $c > 0$.
	\end{corollary}
	
	\begin{figure}[t!]
		\centering
		\includegraphics[width=0.8\textwidth]{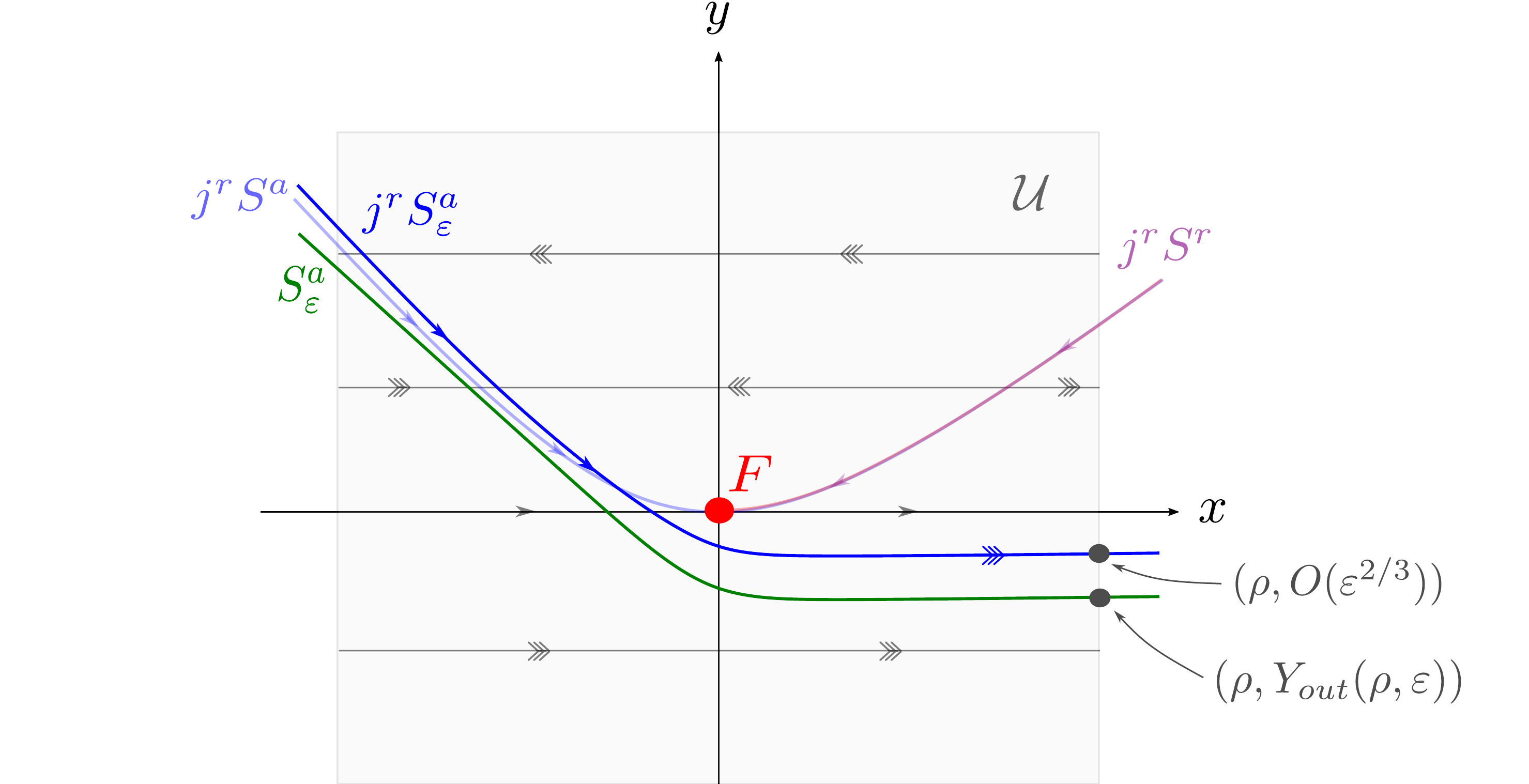}
		\caption{Local geometry and dynamics near a regular fold point $F$ (shown in red) of the map \eqref{eq:standard_form_2d}, as described by Corollary \ref{cor:fold}. The extension of the attracting slow manifold of the map, denoted $S_\eps^a$ and shown in green, is $C^r$-close to the extension of the attracting slow manifold $j^rS_\eps^a$ of the approximating vector field $j^r V(x,y,\eps)$, shown in blue, which is obtained via Proposition \ref{prop:2d_embedding}. We also sketch the singular limit (layer and reduced) dynamics for the approximating vector field, which has a critical manifold $j^rS = j^rS^a \cup F \cup j^rS^r$, where $j^rS^a$ ($j^rS^r$) is normally hyperbolic and attracting (repelling), shown here in shaded blue (purple). The intersection of the right-most boundary of $\mathcal U$ with the forward extensions of $j^rS^a_\eps$ and $S^a_\eps$ are indicated by dark grey disks. $S_\eps^a$ leaves $\mathcal U$ at a point which is $C^r$-close to $(\rho, O(\eps^{2/3}))$.}
		\label{fig:fold}
	\end{figure}
	
	\begin{proof}
		Assuming that $\rho > 0$ is fixed sufficiently small, Proposition \ref{prop:2d_embedding} implies the existence of an approximating vector field $j^r V(x,y,\eps)$ on $\mathcal U$ which is given by \eqref{eq:2d_vf}. Moreover, $j^r V(x,y,\eps)$ is fast-slow with a regular fold point at $(x,y) = (0,0)$. The local dynamics for this problem is described by \cite[Thm.~2.1]{Krupa2001a}, after applying a positive rescaling of $x$, $y$, $\eps$ and $t$ in order to obtain a local normal form. Note that the rescaling in time does not \rev{affect} the location of the slow manifold $j^r S_\eps^a$ of this ODE system, or its extension through $\mathcal U$. The above mentioned result from \cite{Krupa2001a} implies that the extension of $j^r S^a_\eps$ leaves $\mathcal U$ at a point $(\rho, - c \eps^{2/3} + O(\eps \ln \eps))$, where $c > 0$. 
		After undoing the normal form transformation (a simple rescaling) we obtain the analogous result for the truncated map $j^r H(x,y,\eps)$, which has the same attracting slow manifold $j^r S^a_\eps$ (and extension thereof) due to local invariance and the approximation property \eqref{eq:2d_embedding}.
		
		The preceding arguments show that the forward extension of the slow manifold of the truncated map $j^r H(x,y,\eps)$ leaves $\mathcal U$ at a point $(x,y) = (\rho,  - c \eps^{2/3} + O(\eps \ln \eps))$. The extension of the slow manifold $S_\eps^a$ of the original map \eqref{eq:standard_form_2d} leaves $\mathcal U$ at a point which is $C^r$-close to this point, due to the fact that
		\begin{equation}
			\label{eq:H_approx}
			| H(x,y,\eps) - j^r H(x,y,\eps)| = o(|(x,y,\eps)|^r) 
		\end{equation}
		for all $(x,y,\eps) \in \mathcal U \times [0,\eps_0)$, which implies that $\textup{dist} (S^a_\eps, j^r S^a_\eps) = o( | (x,y,\eps)|^r )$.
	\end{proof}
	
	The situation is sketched in Figure \ref{fig:fold}. As is shown by the proof, the geometry is very similar to the geometry of the continuous-time counterpart considered in \cite{Krupa2001a}. This is a direct consequence of the formal embedding result in Proposition \ref{prop:2d_embedding}.
	
	\begin{remark}
		\label{rem:fast_escape}
		Despite the similarities exhibited by discrete and continuous-time systems near a regular fold point, as described in Corollary \ref{cor:fold}, additional complications arise in the discrete setting when it comes to tracking the location of individual iterates of the map \eqref{eq:standard_form_2d}. The problem stems from the fact that for most initial conditions close to or on $S_\eps^a$, the `last iterate' in $\mathcal U$ appears in an $\eps$-dependent neighbourhood of $(0,0)$. The location of the next iterate depends on higher order terms, and therefore on global properties of the map which are not captured in a local expansion or normal form. Such an issue is expected to arise in the analysis near singularities which feature a `fast escape' or `jump'; see also \cite{Jelbart2022b} for an example which arises after considering the Poincar\'e map of a continuous-time fast-slow problem with a global singularity.
	\end{remark}

	\subsubsection{Transcritical dynamics}

	Assume that the map \eqref{eq:standard_form_2d} has a transcritical singularity at $(x,y) = (0,0) \in S$. By Theorem \ref{thm:slow_manifolds}, compact submanifolds of $S^{a,\pm}$ perturb to $O(\eps)$-close slow manifolds, which we denote by $S^{a,\pm}_\eps$. The following result describes the extension of $S_\eps^{a,-}$ through a neighbourhood $\mathcal U \subset \R^2$ about the transcritical point.
	
	\begin{corollary}
		\label{cor:transcritical}
		Consider the two-dimensional fast-slow map \eqref{eq:standard_form_2d} with a transcritical point $(0,0) \in S$ satisfying \eqref{eq:transcritical_orientation}, and let
		\[
		\lambda := \frac{1}{|g_0| \sqrt{\beta^2 - \gamma \alpha}} \left( \delta \alpha + g_0 \beta \right) ,
		\]
		where
		\[
		\alpha := \frac{1}{2} \frac{\partial^2 \tilde f}{\partial x^2}(0,0,0) , \quad
		\beta := \frac{1}{2} \frac{\partial^2 \tilde f}{\partial x \partial y}(0,0,0) , \quad
		\gamma := \frac{1}{2} \frac{\partial^2 \tilde f}{\partial y^2}(0,0,0) , \quad 
		\delta := \frac{\partial \tilde f}{\partial \eps}(0,0,0) , \quad 
		g_0 := \tilde g(0,0,0) .
		\]
		There exists an $\eps_0 > 0$ such that for all $\eps \in (0,\eps_0)$ we have the following:
		\begin{enumerate}
			\item If $\lambda < 1$ then the extension of $S_\eps^{a,-}$ is $C^r$-close to $S_\eps^{a,+}$ when it leaves $\mathcal U$.
			\item If $\lambda > 1$ then the extension of $S_\eps^{a,+}$ is $C^r$-close to a point which is $O(\eps^{1/2})$-close to the critical fiber along $y = 0$ when it leaves $\mathcal U$.
		\end{enumerate}
	\end{corollary}
	
	\begin{figure}[t!]
		\centering
		\subfigure[$\lambda < 1$ stability exchange.]{
			\includegraphics[width=0.47\textwidth]{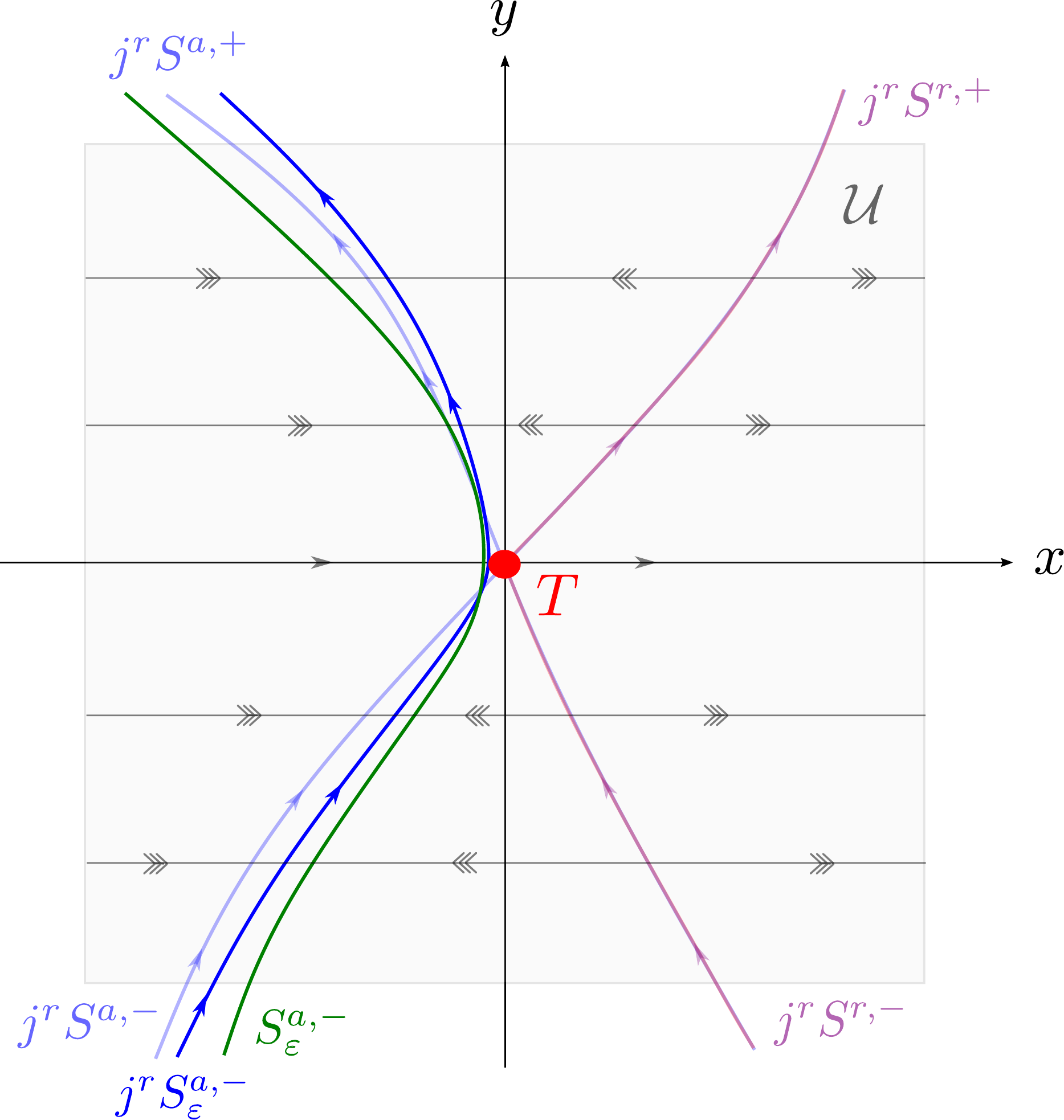}} \quad
		\subfigure[$\lambda > 1$ fast escape.]{
			\includegraphics[width=0.47\textwidth]{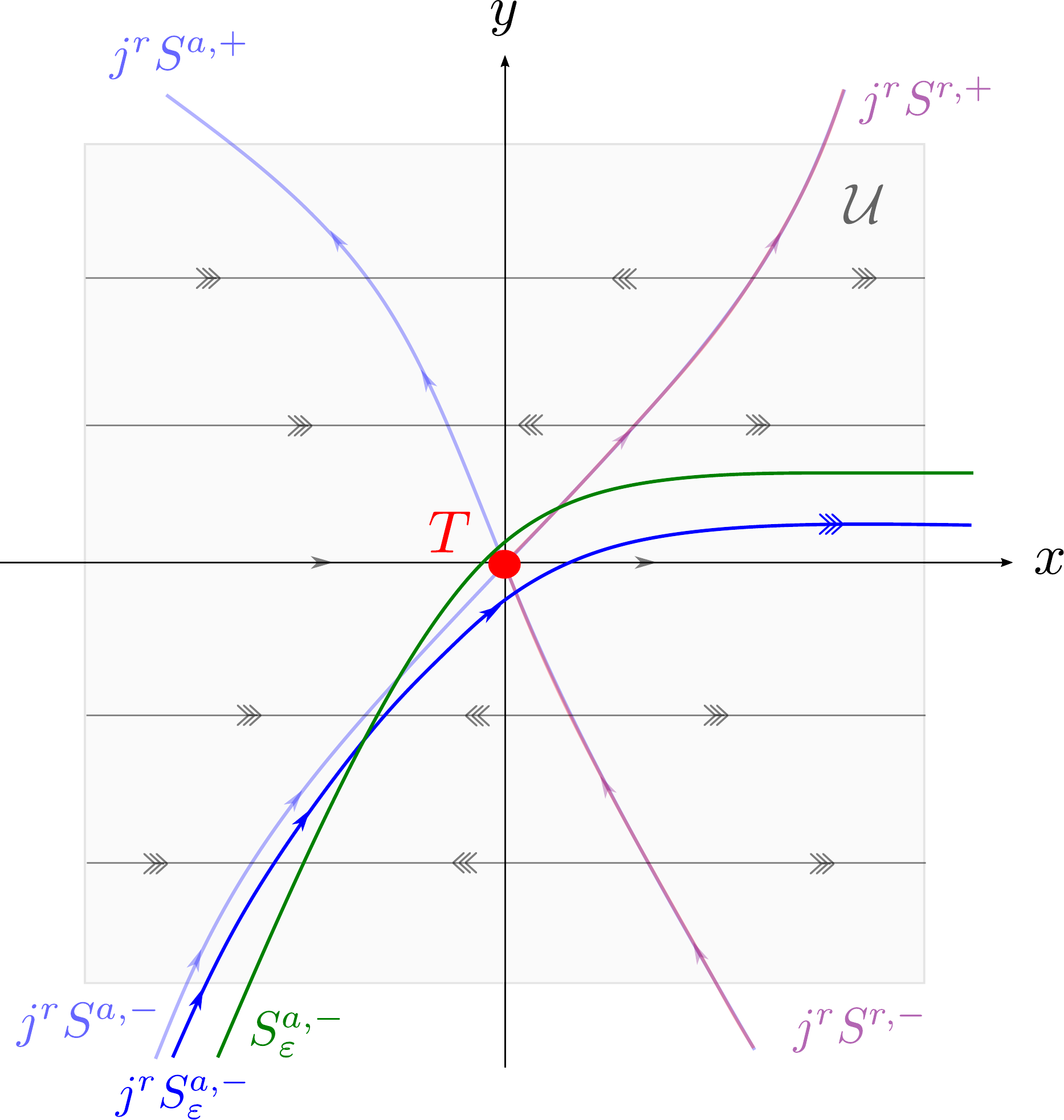}}
		\caption{Local geometry and dynamics near a transcritical point $T$ (shown in red) of the map \eqref{eq:standard_form_2d}, as described by Corollary \ref{cor:transcritical}. We sketch (a) the exchange of stability case for $\lambda < 1$, and (b) the fast escape case for $\lambda > 1$. In both cases, the extension of the attracting slow manifold $S^{a,-}_\eps$ of the map, shown here in green, is $C^r$-close to the extension of the attracting slow manifold $j^rS_\eps^{a,-}$ of the approximating vector field $j^r V(x,y,\eps)$, shown in blue, which is obtained via Proposition \ref{prop:2d_embedding}. We also sketch the singular limit (layer and reduced) dynamics for the approximating vector field, which has a critical manifold $j^rS = j^rS^{a,-} \cup j^rS^{a,+} \cup \rev{T} \cup j^rS^{r,-} \cup j^rS^{r,+}$, where $j^rS^{a,\pm}$ ($j^rS^{r,\pm}$) are normally hyperbolic and attracting (repelling), shown here in shaded blue (purple).}
		\label{fig:transcritical}
	\end{figure}
	
	\begin{proof}
		%
		By Proposition \ref{prop:2d_embedding}, the truncated map $j^r H(x,y,\eps)$ coincides locally with the time-$1$ map induced by the fast-slow vector field $j^r V(x,y,\eps)$ in \eqref{eq:2d_vf}, which has a transcritical point at $(0,0)$. The local dynamics are described after an orientation-preserving linear coordinate transformation and positive rescaling of $\eps$ by \cite[Thm.~2.1]{Krupa2001c}. This result implies the following, in the original coordinates (prior to the linear transformation and rescaling of $\eps$):
		\begin{enumerate}
			\item[1'] If $\lambda < 1$ then the forward extension of $j^r S_\eps^{a,-}$ is $O(\me^{-c/\eps})$-close to $j^r S_\eps^{a,+}$ when it leaves $\mathcal U$, for a constant $c > 0$.
			\item[2'] If $\lambda > 1$ then the forward extension of $j^r S_\eps^{a,+}$ is $O(\eps^{1/2})$-close to the critical fiber along $y = 0$ when it leaves $\mathcal U$.
		\end{enumerate}
		The approximation property \eqref{eq:2d_embedding} implies that the slow manifolds $j^r S_\eps^{a,\pm}$ and their extensions in $\mathcal U$ coincide with those of the truncated map $j^r H(x,y,\eps)$. Thus, both 1' and 2' are also true for the map $j^r H(x,y,\eps)$. Assertions 1 and 2 in the Corollary follow for the map $H(x,y,\eps)$ after accounting for the fact that an $o( |(x,y,\eps)|^r )$ error is introduced when approximating $H(x,y,\eps)$ by $j^r H(x,y,\eps)$; recall equation \eqref{eq:H_approx} in the proof of Corollary \ref{cor:fold}.
	\end{proof}
	
	The situation is sketched in Figure \ref{fig:transcritical}. Due to Proposition \ref{prop:2d_embedding}, the local geometry is similar to the local geometry of the continuous-time counterpart considered in \cite{Krupa2001c}.
	
	\begin{remark}
		If $\lambda > 1$, then the dynamics near a transcritical point of the map \eqref{eq:standard_form_2d} is complicated by the presence of a `last iterate problem' similar to that appearing near regular fold points; c.f.~Remark \ref{rem:fast_escape}.
	\end{remark}
	
	\begin{remark}
		\label{rem:transcritical_canards}
		The authors in \cite{Krupa2001c} showed that fast-slow ODEs with a transcritical point feature canard solutions due to an intersection of the attracting and repelling (extended) slow manifolds for a value of $\lambda$ which lies in an $\eps$-dependent neighbourhood about the threshold value $\lambda = 1$. Similar features are expected to arise in the map \eqref{eq:stnd_form_maps}. \rev{Similarly to the case of canards near folded critical manifolds (recall Remark \ref{rem:canards}), we defer these} considerations \rev{to future} work because (i) they would take us significantly beyond the scope of the present investigation, and (ii) canard solutions are often associated with exponentially small phenomena that may not be treatable with formal approximations of the kind developed herein.
	\end{remark}

	\subsubsection{Pitchfork dynamics}
	
	We can understand the local dynamics near a pitchfork singularity in a similar way. Consider \eqref{eq:standard_form_2d} with a pitchfork point at $(0,0)$ satisfying \eqref{eq:pitchfork_orientation}. Compact submanifolds of the attracting branches of the critical manifold $S^a$ and $S^{a,\pm}$ perturb to $O(\eps)$-close slow manifolds, which we denote by $S^a_\eps$ and $S^{a,\pm}_\eps$ respectively. We are interested in the extension of particular slow manifolds through a neighbourhood $\mathcal U \subset \R^2$ about the pitchfork point. 
	
	\begin{corollary}
		\label{cor:pitchfork}
		Consider the two-dimensional fast-slow map \eqref{eq:standard_form_2d} with a pitchfork point $(0,0) \in S$ satisfying \eqref{eq:pitchfork_orientation}, and let
		\[
		\lambda := \frac{(\delta \alpha + \beta g_0) \sqrt{- \gamma}}{\alpha |g_0| |\alpha|} ,
		\]
		where
		\[
		\alpha := \frac{\partial^2 \tilde f}{\partial x \partial y}(0,0,0) , \quad
		\beta := \frac{1}{2} \frac{\partial^2 \tilde f}{\partial y^2}(0,0,0) , \quad
		\gamma := \frac{1}{6} \frac{\partial^3 \tilde f}{\partial x^3}(0,0,0) , \quad 
		\delta := \frac{\partial \tilde f}{\partial \eps}(0,0,0) , \quad 
		g_0 := g(0,0,0) .
		\]
		There exists an $\eps_0 > 0$ such that for all $\eps \in (0,\eps_0)$ we have the following:
		\begin{enumerate}
			\item If $g_0 > 0$ and $\lambda < 0$ then the extension of $S_\eps^a$ is $C^r$-close to $S_\eps^{a,-}$ when it leaves $\mathcal U$.
			\item If $g_0 > 0$ and $\lambda > 0$ then the extension of $S_\eps^a$ is $C^r$-close to $S_\eps^{a,+}$ when it leaves $\mathcal U$.
			\item If $g_0 < 0$ then the extension of both $S_\eps^{a,\pm}$ are $C^r$-close to $S_\eps^{a}$ when they leave $\mathcal U$.
		\end{enumerate}
	\end{corollary}
	
	\begin{figure}[t!]
		\centering
		\subfigure[$g_0 > 0$, $\lambda < 0$.]{
			\includegraphics[width=0.45\textwidth]{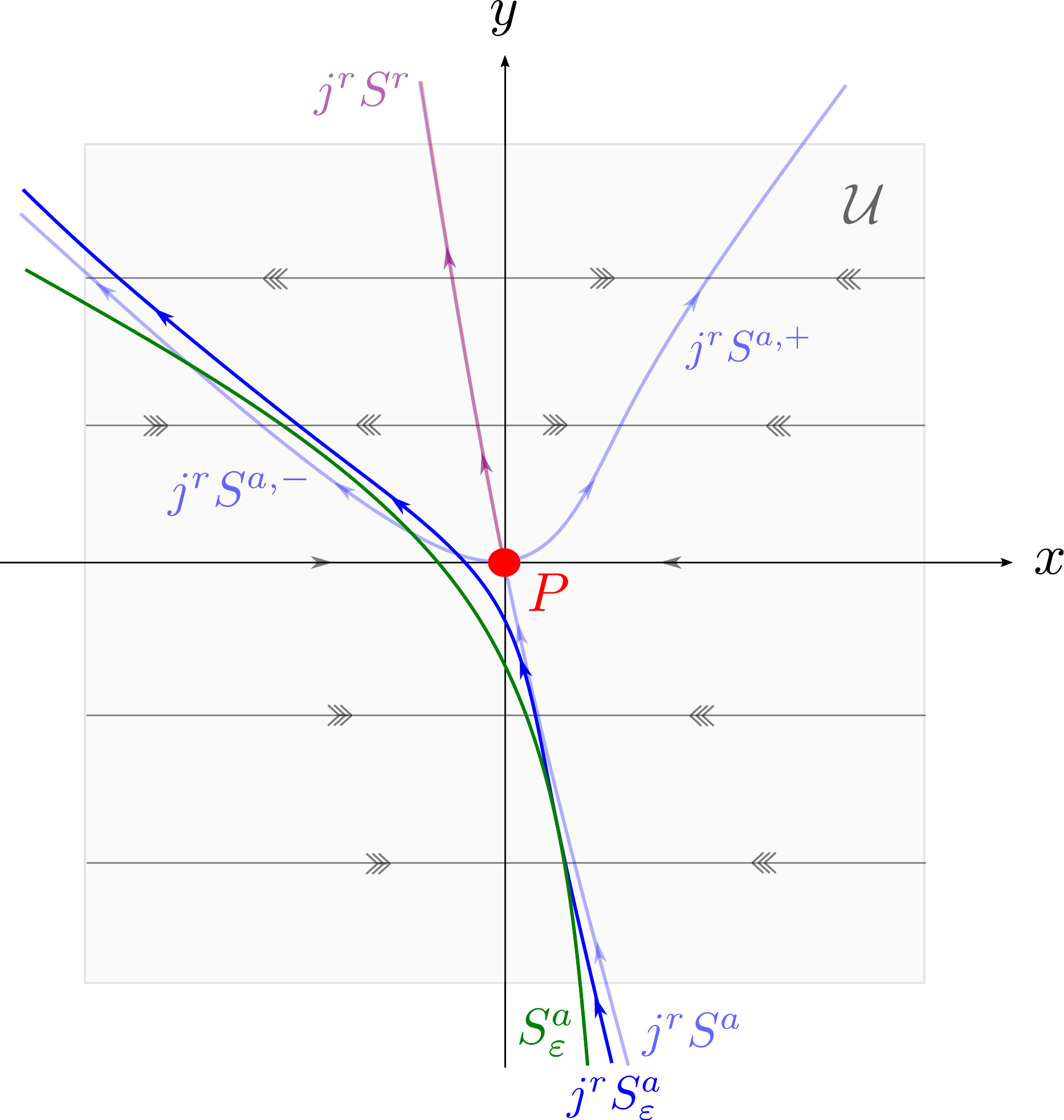}} \quad
		\subfigure[$g_0 > 0$, $\lambda > 0$.]{
			\includegraphics[width=0.45\textwidth]{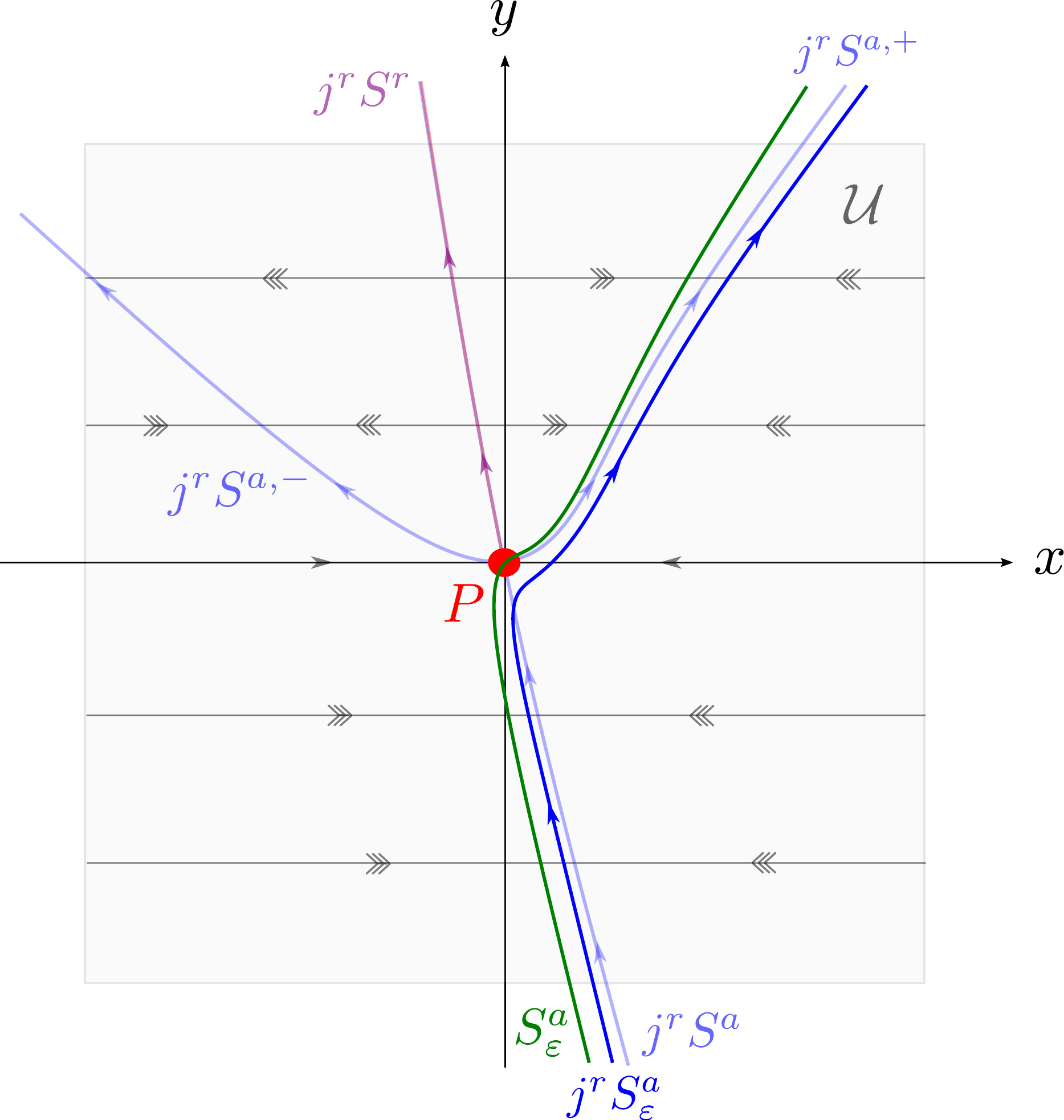}}
		
		\subfigure[$g_0 < 0$, $\lambda > 0$.]{
			\includegraphics[width=0.45\textwidth]{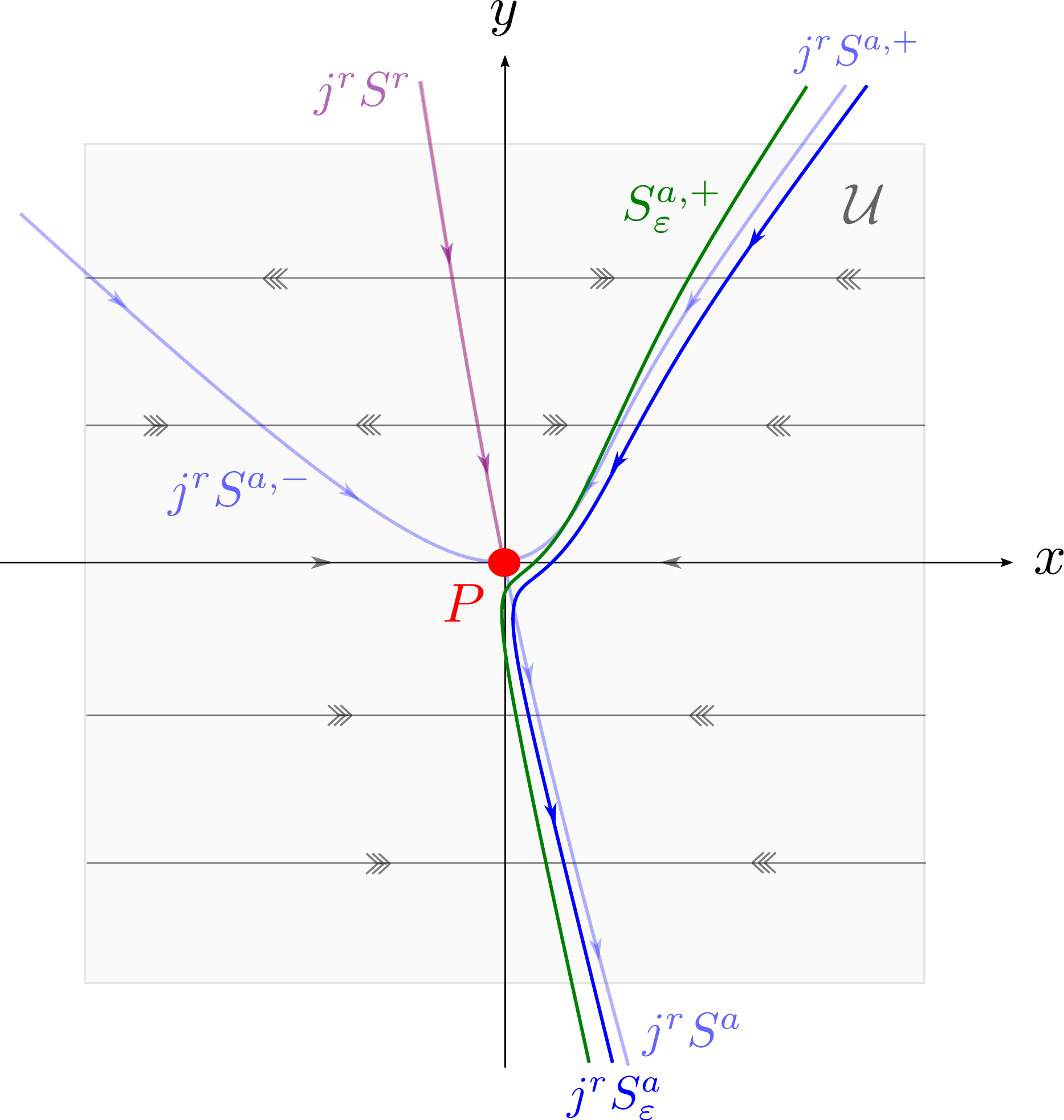}}
		\caption{Local geometry and dynamics near a pitchfork point $P$ (shown in red) of the map \eqref{eq:standard_form_2d}, as described by Corollary \ref{cor:pitchfork}. We sketch the cases (a) $g_0 > 0$, $\lambda < 0$, (b) $g_0 > 0$, $\lambda > 0$, and (c) $g_0 < 0$, $\lambda > 0$ (the case $g_0 < 0$, $\lambda < 0$ is similar to (c)). The extension of the attracting slow manifold, denoted $S^a_\eps$ in (a)-(b) and $S^{a,+}_\eps$ in (c), is shown in green, and $C^r$-close to the extension of the corresponding attracting slow manifold ($j^rS_\eps^{a}$ in (a)-(b) and $j^rS_\eps^{a,+}$ in (c)) of the approximating vector field $j^r V(x,y,\eps)$, which is shown in blue. We also sketch the singular limit (layer and reduced) dynamics for the approximating vector field, which has a critical manifold $j^rS = j^rS^{a} \cup j^rS^{a,-} \cup j^rS^{a,+} \cup P \cup j^rS^{r}$, where $j^rS^{a}$ and $j^rS^{a,\pm}$ ($j^rS^{r}$) are normally hyperbolic and attracting (repelling), shown here in shaded blue (purple).}
		\label{fig:pitchfork}
	\end{figure}
	
	\begin{proof}
		The overall approach is again similar to the proofs of Corollaries \ref{cor:fold} and \ref{cor:transcritical}, so we include fewer details. In this case, \cite[Thm.~4.1]{Krupa2001c} implies that the truncated approximating ODE system \eqref{eq:2d_vf} is such that the following assertions hold for some constant $c > 0$:
		\begin{enumerate}
			\item[1'] If $g_0 > 0$ and $\lambda < 0$ then the extension of $j^r S_\eps^a$ is $O(\me^{-c / \eps})$-close to $j^r S_\eps^{a,-}$ when it leaves $\mathcal U$.
			\item[2'] If $g_0 > 0$ and $\lambda > 0$ then the extension of $j^r S_\eps^a$ is $O(\me^{-c / \eps})$-close to $j^r S_\eps^{a,+}$ when it leaves $\mathcal U$.
			\item[3'] If $g_0 < 0$ then the extension of both $j^r S_\eps^{a,\pm}$ is $O(\me^{-c / \eps})$-close to $j^r S_\eps^a$ when it leaves $\mathcal U$.
		\end{enumerate}
		The result follows from the fact that the (extended) slow manifolds $j^r S_\eps^a$ and $j^r S_\eps^{a,\pm}$ are the same for the truncated vector field $j^r V(x,y,\eps)$ and the truncated map $j^r H(x,y,\eps)$, and the fact that $\textup{dist} (S_\eps, j^r S_\eps) = o( \| (x,y,\eps) \|^r)$.
	\end{proof}
	
	The situation is sketched in Figure \ref{fig:pitchfork}. The local geometry is similar to the local geometry of the continuous-time counterpart considered in \cite{Krupa2001c}, again due to Proposition \ref{prop:2d_embedding}.
	
	\begin{remark}
		\label{rem:pitchfork_canards}
		Similarly to the transcritical case, the authors in \cite{Krupa2001c} showed  that planar fast-slow ODEs with a pitchfork singularity feature canard solutions for a value of $\lambda$ which lies in an $\eps$-dependent interval about the threshold value $\lambda = 0$. We omit the consideration of canards near pitchfork singularities in planar fast-slow maps in this work for the reasons provided in Remark\rev{s \ref{rem:canards} and} \ref{rem:transcritical_canards}.
	\end{remark}
	
	Corollaries \ref{cor:fold}, \ref{cor:transcritical} and \ref{cor:pitchfork} describe the extension of attracting and locally invariant slow manifolds through a neighbourhood of non-hyperbolic points of regular fold, transcritical and pitchfork type in planar fast-slow maps in standard form \eqref{eq:standard_form_2d}. The qualitative similarity to the dynamics of the corresponding ODE problems in the same dimension, which have been studied in \cite{Krupa2001a} and \cite{Krupa2001c}, is a consequence of \rev{the} formal embedding result in Proposition \ref{prop:2d_embedding} (or more generally Theorem \ref{thm:embedding_nilpotent}). It demonstrates that up to an error determined by the smoothness, many important aspects of the local dynamics near unipotent singularities of fast-slow maps can be understood via the study of a corresponding ODE problem. This is advantageous, since there are many techniques, in this case geometric blow-up, which apply in the ODE setting \rev{but} not directly to maps.

	\section{Regular contact points in arbitrary dimensions}
	\label{sec:regular_contact_points}
	
	The results in Section \ref{sec:2d_applications} demonstrate the breadth and applicability of Theorem \ref{thm:embedding_nilpotent}, but they are restrictive in these sense that they only apply to two-dimensional fast-slow maps in standard form \eqref{eq:standard_form_2d}. In what follows we show that Theorem \ref{thm:embedding_nilpotent} can also be used to approximate the dynamics in higher dimensional maps, either directly in the case of maps with $\textup{codim}(S) = 1$, or indirectly after center manifold reduction in the case of maps with $\textup{codim}(S) \geq 2$. We focus on the dynamics near an important class of contact points which can be viewed as the higher dimensional counterpart of regular fold points in (generally non-standard form) fast-slow maps on $\R^n$, but our approach can be generalised in order to study other codimension-1 singularities in $\mathcal C$.
	
	\
	
	Regular contact points can be defined analogously to regular contact points in continuous-time GSPT; see \cite{Lizarraga2020c} and \cite[Sec.~4.1 and Sec.~4.2]{Wechselberger2019}.
	
	\begin{definition}
		\label{def:regular_contact_point}
		\textup{(Regular contact point)} Consider a $C^r$-smooth map \eqref{eq:general_maps} under Assumptions \ref{ass:fast-slow} and \ref{ass:factorisation}, where $r \geq 3$. We say that $z \in \mathcal C \subset S$ is a regular contact point if
		\begin{equation}
			\label{eq:contact_cond}
			\textup{rk} \left( \textup{D} f N (z) \right) = n-k-1 
		\end{equation}
		and the following genericity conditions are satisfied:
		\begin{itemize}
			\item Transversality, i.e.~$\textup{rk} (\textup{D} f) (z) = n-k$; 
			\item Nondegeneracy, i.e.~$l \cdot ((\textup{D}^2f)(Nr,Nr) + \textup{D} f \textup{D} N (Nr,r))(z) \neq 0$;
			\item Slow regularity, i.e.~$N r l \textup{D} f G(z,0) \neq \mathbb O_{n,1}$,
		\end{itemize}
		where $l$ and $r$ denote left and right null vectors of the matrix $\textup{D} f N(z)$ respectively, 
		and $(D^2f)(Nr,Nr)$ and $\textup{D} f \textup{D} N (Nr,r)$ are bilinear forms with the following componentwise definitions:
		\[
		\begin{split}
			[\textup{D}^2f]_i(Nr,Nr) &= \sum_{l,m=1}^n \sum_{j,s=1}^{n-k} \frac{\partial^2 f_i}{\partial z_l \partial z_m} (N_{mj} r_j) (N_{ls} r_s) , \\
			[\textup{D} f \textup{D} N]_i(Nr,r) &= \sum_{l,m=1}^n \sum_{j,s=1}^{n-k} \frac{\partial f_i}{\partial z_l} \frac{\partial N_{lj}}{\partial z_m} (N_{ms} r_s) r_j,
		\end{split}
		\]
		for each $i = 1, \ldots, n-k$. The definition extends to submanifolds, i.e.~a submanifold $\mathcal C_r \subset S$ is called a regular contact (sub)manifold if every $z \in \mathcal C_r$ is a regular contact point.
	\end{definition}
	
	The case of regular contact along a 1-dimensional submanifold \rev{$\mathcal C^r$} of a 2-dimensional critical manifold $S$ in $\R^3$ is sketched in Figure \ref{fig:contact_point}. The rank condition \eqref{eq:contact_cond} implies that there is a single non-trivial multiplier, which we may denote by $\mu_1$, which satisfies $\mu_1(z) = 1$ along a $(k-1)$-dimensional regular contact manifold $\mathcal C^r \subset S$. The algebraic multiplicity of the zero eigenvalue of the matrix $N \D f$ is $k+1$ at a regular contact point, which is one greater than the geometric multiplicity $k$. Geometrically, this corresponds to a tangency between $S$ and the curve(s) containing orbits of the layer problem \eqref{eq:layer_map}. More precisely, if $z \in S$ is a regular contact point, then $\textup{range} (N(z))$ and $\textup{T}_z S$ are tangent along the $1$-dimensional line given by the span of the (generalised) null vector $\mathfrak N_1(z)$ of the matrix $\D f N(z)$ which corresponds to the critical multiplier $\mu_1(z)$. The transversality condition implies that $S$ is a regularly embedded submanifold of $\R^n$, the nondegeneracy condition implies that the tangency described above is quadratic when considered on the nonlinear level, and the slow regularity condition implies that orbits of the reduced vector field $\Pi^S_{\mathcal N} G(z,0)|_{z \in S}$, which locally approximates the slow dynamics (recall Theorem \ref{thm:reduced_map}), extend to transversal intersections with $\mathcal C^r$; see again Figure \ref{fig:contact_point}.
	
	\begin{figure}[t!]
		\centering
		\includegraphics[width=0.55\textwidth]{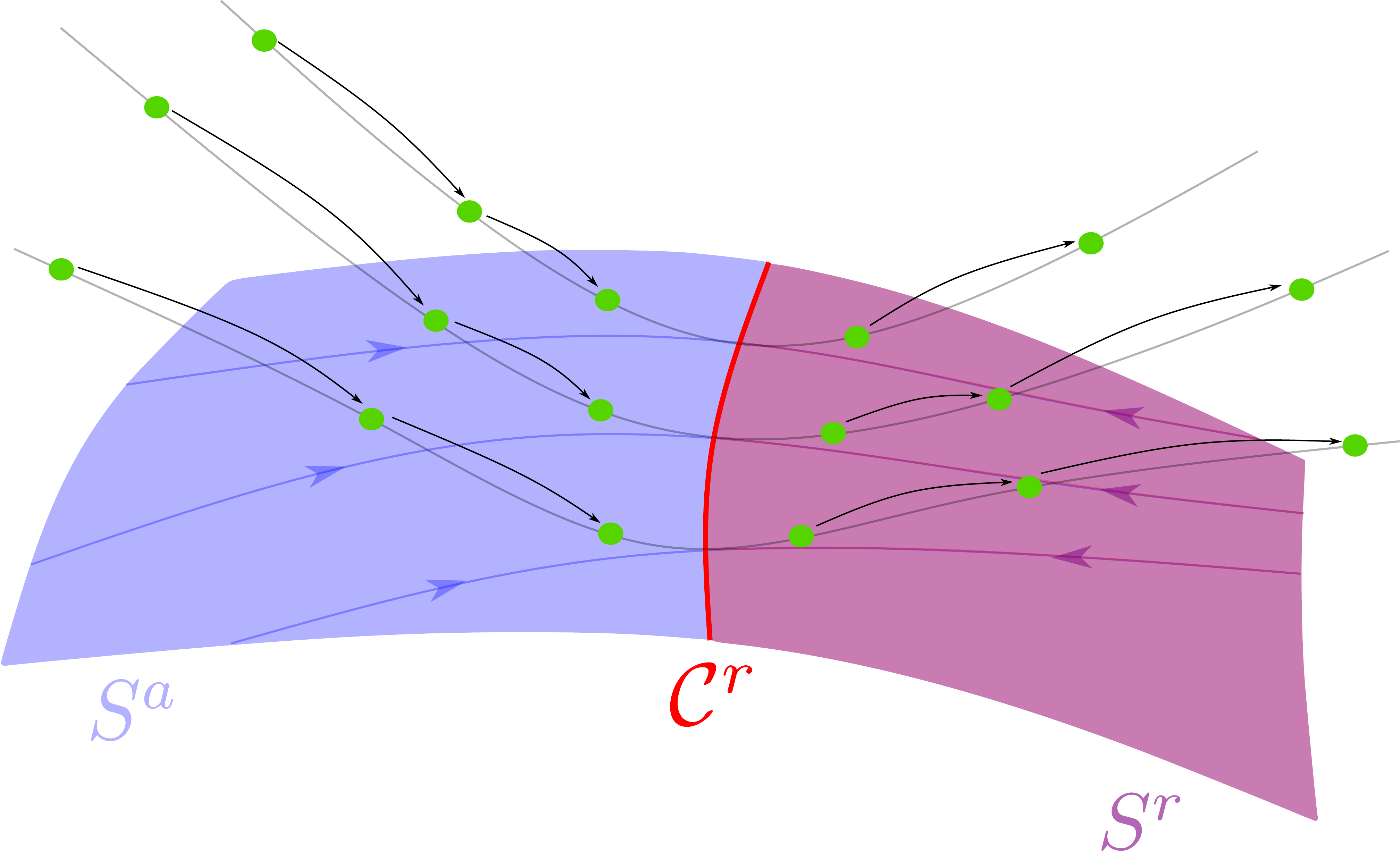}
		\caption{Singular geometry and dynamics near a $2$-dimensional critical manifold $S$ in $\R^3$ with a normally hyperbolic attracting and repelling submanifold $S^a$ and $S^r$, shown in shaded blue and purple respectively, separated by a $1$-dimensional regular contact submanifold $\mathcal C^r \subset S$ shown in red. We also sketch a number of representative trajectories of the reduced flow which approximates the dynamics on $S$ (shaded blue and purple on $S^a$ and $S^r$ respectively), as well as sample iterates of the layer map (green with arrows to indicate forward iteration), which lie on invariant curves which contain orbits of the layer map (shaded grey). These curves are tangent to $S$ along $\mathcal C^r$ due to the contact condition \eqref{eq:contact_cond} in Definition \ref{def:regular_contact_point}, and the tangency is locally quadratic due to the nondegeneracy condition. The slow regularity condition ensures that orbits of the reduced problem extend to a transversal intersection with $\mathcal C^r$.}
		\label{fig:contact_point}
	\end{figure}
	
	\begin{remark}
		The slow regularity condition appearing in \cite[Def.~4.4]{Wechselberger2019} is
		\begin{equation}
			\label{eq:slow_regularity}
			N \textup{adj} (\textup{D} fN) \textup{D} f G (z,0) \neq \mathbb O_{n,1} ,
		\end{equation}
		where $\textup{adj} (\textup{D} f N)(z)$ denotes the adjoint/adjugate matrix, which is defined by the relation $\textup{adj} (A) = \det (A) A$. The slow regularity condition in Definition \ref{def:regular_contact_point} is equivalent to \eqref{eq:slow_regularity} as long as the rank condition \eqref{eq:contact_cond} is satisfied, since
		\[
		\textup{rk} (\textup{D} f N) = n-k-1 \quad 
		\implies \quad
		\textup{rk} \left( \textup{adj} (\textup{D} f N) \right) = 1 \quad 
		\implies \quad 
		\textup{adj} (\textup{D} f N) = \alpha r l ,
		\]
		where $l$ and $r$ are null vectors of $\textup{D} f N$ and $\alpha \neq 0$ is a constant.
	\end{remark}
	
	\begin{remark}
		For fast-slow maps in standard form \eqref{eq:standard_form}, Definition \ref{def:regular_contact_point} reduces to a definition which closely resembles that the well-known definition of a regular fold point in continuous-time GSPT (see e.g.~\cite{Wechselberger2012}). More precisely, a regular contact point for the map \eqref{eq:standard_form} is a point $z \in S$ such that
		\[
		\textup{rk} \ (\textup{D}_y \rev{\tilde f}(z)) = n-k-1 ,
		\]
		\rev{and}
		\[
		l \cdot \textup{D}_x\rev{\tilde f}(z) \neq \mathbb O_{k,k} , \qquad
		l \cdot \textup{D}^2_{yy} \rev{\tilde f}(r,r) \rev{(z)} \neq 0 , \qquad 
		\rev{\textup{adj} (\textup{D}_x \tilde f) (\textup{D}_y \tilde f) \tilde g (z,0) \neq \mathbb O_{n-k} ,}
		\]
		where $l$ and $r$ are left and right null vectors of $\textup{D}_yf(z)$ respectively, and $\textup{D}^2_{yy}\rev{\tilde f}(r,r)$ is the bilinear form given component-wise by
		\[
		[\textup{D}_{yy}^2\rev{\tilde f}]_i(r,r) = \sum_{j,l=1}^{n-k} \frac{\partial^2 \rev{\tilde f}_i}{\partial y_j \partial y_l} \bigg|_z r_j r_l , \qquad i = 1, \ldots, n-k.
		\]
	\end{remark}
	
	\begin{remark}
		\label{rem:transversality}
		The transversality condition $\textup{rk} (\D f)(z) = n-k$ is included in Definition \ref{def:regular_contact_point} in order to clarify the comparison to classical bifurcation theory, but it is automatically satisfied due to Assumption \ref{ass:fast-slow}. Assumption \ref{ass:factorisation} is also not necessary because the definition is local; we refer again to the discussion following the statement of Assumption \ref{ass:factorisation}. We opt to keep it in order to simplify the formulation.
	\end{remark}
	
	\begin{remark}
		\label{rem:canards_again}
		\rev{By analogy to \cite{Wechselberger2019}, canard points can be defined as points where the slow regularity condition is violated, i.e.~points $z \in \mathcal C^r \subseteq S$ such that
			\[
			N \textup{adj} (\textup{D} f N) \textup{D} f G(z, 0) = \mathbb O_{n,1} .
			\]
			Further consideration of canards are left for future work.}
	\end{remark}
	
	In order to understand the dynamics near a regular contact point (or submanifold), we would like to apply Theorem \ref{thm:embedding_nilpotent}. Unfortunately, Theorem \ref{thm:embedding_nilpotent} does not apply directly if $\textup{codim} (S) \geq 2$, in which case the linearisation in \eqref{eq:layer_linearisation} is not unipotent due to the presence of a non-trivial multiplier $\mu_2(z)$ satisfying $|\mu_2(z)| \neq 1$. Nevertheless, we can apply Theorem \ref{thm:embedding_nilpotent} to a restricted map after applying the center manifold theorem along $\mathcal C^r$.

	\subsection{Center manifold reduction}
	\label{sub:center_manifold_reduction}
	
	It was shown in \cite{Wechselberger2012,Wechselberger2019} that the mathematical analysis near a regular fold submanifold $\mathcal C^r \subset S$ in fast-slow ODE systems can be simplified after a preliminary reduction to a $(k+1)$-dimensional local local center manifold along $\mathcal C^r$. A similar feature appears in the context of regular fold submanifolds in fast-slow maps in general form \eqref{eq:general_maps}.
	
	We consider maps in general form \eqref{eq:general_maps} under Assumptions \ref{ass:fast-slow}-\ref{ass:factorisation}, which are $C^r$-smooth with $r \geq 3$, $\textup{codim}(S) \geq 2$, and a regular contact point at $z = 0 \in S$. Using Assumption \ref{ass:fast-slow}, in particular the assumption that $S$ is a regularly embedded submanifold of $\R^n$, we may choose local coordinates 
	$z = (x,y) \in \R^k \times \R^{n-k}$ such that the matrix $\D_yf$ is regular (full rank and invertible) in a neighbourhood of $z = 0$. We write
	\begin{equation}
		\label{eq:general_maps_coords_1}
		\begin{pmatrix}
			x \\
			y
		\end{pmatrix}
		\mapsto H(x,y,\eps) = 
		\begin{pmatrix}
			x \\
			y
		\end{pmatrix}
		+ 
		\begin{pmatrix}
			N^x(x,y) \\
			N^y(x,y)
		\end{pmatrix} 
		f(x,y) + \eps 
		\begin{pmatrix}
			G^x(x,y,\eps) \\
			G^y(x,y,\eps)
		\end{pmatrix} ,
	\end{equation}
	where $N^x(x,y)$ and $N^y(x,y)$ are matrices of size $k \times (n-k)$ and $(n-k) \times (n-k)$ respectively, and $G^x(x,y,\eps)$ and $G^y(x,y,\eps)$ are column vectors of size $k \times 1$ and $(n-k) \times 1$ respectively. Before we apply the center manifold theorem, it is helpful to isolate the $(k+1)$-dimensional center eigenspace $E^{\textup{c}}(0)$. 
	This can be achieved in a two step procedure which is directly analogous to the corresponding derivation in the context of fast-slow ODEs in \cite[Ch.~4.4]{Wechselberger2019}.
	
	\begin{lemma}
		\label{lem:cm_normal_form}
		The map \eqref{eq:general_maps_coords_1} is locally $C^r$-conjugate to the map
		\begin{equation}
			\label{eq:cm_normal_form}
			\widehat H(x,u,w,\eps) = 
			\begin{pmatrix}
				x \\
				u \\
				w
			\end{pmatrix}
			+ 
			\begin{pmatrix}
				\widehat N^x(x,u,w) \\
				\widehat N^u(x,u,w) \\
				\widehat N^w(x,u,w)
			\end{pmatrix} 
			\begin{pmatrix}
				u \\
				w
			\end{pmatrix}
			+ \eps 
			\begin{pmatrix}
				\widehat G^x(x,u,w,\eps) \\
				\widehat G^u(x,u,w,\eps) \\
				\widehat G^w(x,u,w,\eps)
			\end{pmatrix} 
			+ R(x,u,w,\eps)
			,
		\end{equation}
		where $(x, u, w) \in \R^k \times \R \times \R^{n-k-1}$, and the remainder term satisfies
		\[
		R(x, u, w, \eps) = 
		\begin{pmatrix}
			\mathbb O_{k,1} \\
			O(|(u, w, \eps)|^2)
		\end{pmatrix}
		.
		\]
		The critical manifold of the map \eqref{eq:cm_normal_form} is contained within the hyperplane defined by $u = 0$, $w = 0$, and the generalised eigenspace $E^{\textup{c}}(0)$ is spanned by the $(x,u)$-coordinates.
	\end{lemma}
	
	\begin{proof}
		This can be proven with minor adaptations to the proof of \cite[Thm.~4.1]{Wechselberger2019}. The first step is to rectify $S$ with a nonlinear coordinate transformation
		\[
		v = f(x,y) , \qquad
		y = K(x,v) .
		\]
		Note that a local inverse exists due to the implicit function theorem, since $\D_y f$ is regular. In particular, we have $K(x,v) = (\D_yf)^{-1} \left( v - \D_xf x \right) + O(|(x,v)|^2)$. Taylor expansion about $v = 0$ and $\eps = 0$ in the new coordinates leads to the map
		\begin{equation}
			\label{eq:general_map_coords_flat_S}
			\widetilde H(x,v,\eps) = 
			\begin{pmatrix}
				x \\
				v
			\end{pmatrix}
			+ 
			\begin{pmatrix}
				N^x (x,K(x,v)) \\
				\D f N (x,K(x,v))
			\end{pmatrix}
			v + \eps
			\begin{pmatrix}
				G^x (x,K(x,v),\eps) \\
				\D f G (x,K(x,v),\eps)
			\end{pmatrix} 
			+ R(v,\eps) ,
		\end{equation}
		where the remainder function satisfies $R(v, \eps) = (\mathbb O_{k,1}, O(|(v, \eps)|^2))^\transpose$. The advantage is that $S$ has been rectified along $v = 0$. In particular, the tangent space $\textup{T}_0S$ is now spanned by the $x$-coordinates, since the Jacobian matrix has the form
		\[
		\D \widetilde H(0,0,0) = \mathbb I_n + 
		\begin{pmatrix}
			\mathbb O_{k,k} & N^x (x,K(x,v)) \\
			\mathbb O_{n-k,k} & \D fN (x,K(x,v)) 
		\end{pmatrix} .
		\]
		
		The second step is to extract the fast direction along which contact occurs. This can be achieved with a linear transformation of the form
		\[
		\begin{pmatrix}
			u \\
			w
		\end{pmatrix}
		=
		\begin{pmatrix}
			l \\
			Q
		\end{pmatrix}
		v , \qquad 
		v =
		\begin{pmatrix}
			r & P 
		\end{pmatrix}
		\begin{pmatrix}
			u \\
			w
		\end{pmatrix} 
		= r u + P w ,
		\]
		where $u \in \R$, $w \in \R^{n-k-1}$, $l$ and $r$ are left and right null vectors of $\D f N$ at $z = 0$, and $Q$, $P$ are matrices of size $(n-k-1) \times (n-k)$, $(n-k) \times (n-k-1)$ such that
		\[
		\begin{pmatrix}
			r & P
		\end{pmatrix}
		\begin{pmatrix}
			l \\
			Q
		\end{pmatrix}
		=
		\begin{pmatrix}
			l \\
			Q
		\end{pmatrix}
		\begin{pmatrix}
			r & P
		\end{pmatrix}
		= \mathbb I_{n-k} .
		\]
		Expressing \eqref{eq:general_maps_coords_1} in $(x,u,w)$ coordinates, we obtain
		\begin{equation}
			\label{eq:general_maps_coords_2}
			\widehat H(x,u,w,\eps) = 
			\begin{pmatrix}
				x \\
				u \\
				w
			\end{pmatrix}
			+ 
			\begin{pmatrix}
				N^x \\
				l \D f N \\
				Q \D f N
			\end{pmatrix} 
			\begin{pmatrix}
				r & P
			\end{pmatrix}
			\begin{pmatrix}
				u \\
				w
			\end{pmatrix}
			+ \eps 
			\begin{pmatrix}
				G^x \\
				l \D f G \\
				Q \D f G
			\end{pmatrix} 
			+ R(x,u,w,\eps)
			,
		\end{equation}
		where $R(u, w, \eps) = (\mathbb O_{k,1}, O(|(u, w, \eps)|^2))^\transpose$ and we have omitted the argument notation in the right-hand side for notational simplicity. The map \eqref{eq:general_maps_coords_2} is in the desired form \eqref{eq:cm_normal_form}.
		
		Finally, the Jacobian matrix is upper block triangular of the form
		\begin{equation}
			\label{eq:hat_H_Jacobian}
			\D \widehat H(0,0,0,0) = \mathbb I_n + 
			\begin{pmatrix}
				\mathbb O_{k,k} & N^x r  & N^x P \\
				\mathbb O_{1,k} & 0 & \mathbb O_{1,n-k-1} \\
				\mathbb O_{n-k-1,k} & \mathbb O_{n-k-1,1} & Q \D f N P
			\end{pmatrix} ,
		\end{equation}
		which shows that the $(k+1)$-dimensional generalised center eigenspace $E^{\textup{c}}(0)$ is spanned by the $(x,u)$-coordinates.
	\end{proof}
	
	
	We now state the center manifold theorem which applies near a regular contact point of the map \eqref{eq:general_maps}, in the local coordinates given by \eqref{eq:cm_normal_form}. This result closely resembles the continuous-time counterpart for fast-slow ODEs in \cite[Thm.~4.1]{Wechselberger2019}.
	
	\begin{lemma}
		\label{lem:center_manifold}
		Consider the map \eqref{eq:cm_normal_form} 
		with a regular contact point $0 \in \mathcal C_r \subseteq S$. 
		There exists a $C^r$-smooth, $(k+1)$-dimensional local center manifold $W_{\textup{loc}}^{\textup{c}}(0)$ which is tangent to the center subspace $E^{\textup{c}}(0)$. 
		There is an $\eps_0 > 0$ such that for each $\eps \in [0,\eps_0)$, $W_{\textup{loc}}^{\textup{c}}(0)$ has the local graph representation
		\begin{equation}
			\label{eq:center_manifold_graph}
			W_{\textup{loc}}^{\textup{c}}(0) = \left\{(x, u, W(x,u,\eps)) : (x,u) \in \mathcal W \right\} , \qquad 
			W(x,u,\eps) = u W_{0}(x,u) + \eps W_{\textup{rem}}(x,u,\eps) ,
		\end{equation}
		where $\mathcal W$ is a neighbourhood of $0$ in $\mathbb R^{k+1}$, and the $C^r$-smooth function $W : \mathcal W \times [0,\eps_0) \to \R^{n-k-1}$ satisfies $W(x,0,0) = 0$. The restricted map $H|_{W_{\textup{loc}}^{\textup{c}}(0)} : \mathcal W \times [0,\eps_0) \to \R^{k+1}$ defines an $\eps$-family of local diffeomorphism\rev{s} of the form
		\begin{equation}
			\label{eq:restricted_map}
			\begin{pmatrix}
				x \\
				u
			\end{pmatrix}
			\mapsto
			\widehat H|_{W_{\textup{loc}}^{\textup{c}}(0)}(x,u,\eps)
			=
			\begin{pmatrix}
				x \\
				u
			\end{pmatrix}
			+
			\begin{pmatrix}
				\widetilde N^x(x,u) \\
				\widetilde N^u(x,u)
			\end{pmatrix}
			\widetilde f(x,u) + \eps
			\begin{pmatrix}
				\widetilde G^x(x,u,\eps) \\
				\widetilde G^u(x,u,\eps)
			\end{pmatrix} ,
		\end{equation}
		where $\widetilde N : = (\widetilde N^x, \widetilde N^u)^\transpose$ is a $(k+1) \times 1$ column vector given by
		\begin{equation}
			\label{eq:N_cm}
			\begin{pmatrix}
				\widetilde N^x \\
				\widetilde N^u
			\end{pmatrix} 
			= 
			\begin{pmatrix}
				N^x \left( r + P W_0(x,u) \right) \\
				l \textup{D} f N \left( r + P W_0(x,u) \right)
			\end{pmatrix} ,
		\end{equation}
		the function $\widetilde f : \mathcal W \to \R$ is $C^r$-smooth and satisfies
		\begin{equation}
			\label{eq:tilde_f_properties}
			\widetilde f(x,0) = 0 , \qquad 
			\frac{\partial \widetilde f}{\partial u}(x,0) = 1 ,
		\end{equation}
		i.e.~$\widetilde f(x,u) = u + u (1 + O(|x|, u) )$, and
		$\widetilde G : = (\widetilde G^x, \widetilde G^u)^\transpose$ is a $(k+1) \times 1$ column vector such that
		\begin{equation}
			\label{eq:G_cm}
			\widetilde G(0,0,0) = 
			\begin{pmatrix}
				\widetilde G^x(0,0,0,0) \\
				\widetilde G^u(0,0,0,0)
			\end{pmatrix}
			=
			\begin{pmatrix}
				G^x(0,0,0,0) \\
				l \textup{D} f G(0,0,0,0) 
			\end{pmatrix} .
		\end{equation}
		%
		%
		The restricted map \eqref{eq:restricted_map} is fast-slow with a $k$-dimensional critical manifold $S$ and a unipotent regular contact point at $(x,u) = (0,0) \in \R^{k+1}$.
	\end{lemma}
	
	\begin{proof}
		The existence of a $(k+1)$-dimensional $C^r$-smooth center manifold follows from the center manifold theorem; see e.g.~\cite{Kuznetsov2013}. 
		The $(n-k-1) \times (n-k-1)$ matrix $Q \D f N P$ appearing in the bottom-right block of the Jacobian in \eqref{eq:hat_H_Jacobian} is regular; its eigenvalues $\lambda_j(z)$ have non-zero real part because the matrix $\mathbb I_{n-k-1} + Q \D f N P$ encodes the non-trivial multipliers $\mu_j(z) = 1 + \lambda_j(z)$ for $j = 2, \ldots , n-k$ which are not on the unit circle (and therefore not equal to $1$). Thus, the implicit function theorem implies that $W_{\textrm{loc}}^c(0)$ can be written as a graph over the $(x,u)$-coordinates. The particular form of $W(x,u,\eps)$ in \eqref{eq:center_manifold_graph} follows after Taylor expansion in $\eps$, and the factorisation $W(x,u,0) = u W_0(x,u)$ follows from the fact that $S$ lies in the hyperplane defined by $u = 0$ and $w = W(x,0,0) = 0$.
		
		Directly restricting \eqref{eq:general_maps_coords_2} (which is equivalent to \eqref{eq:cm_normal_form}) to $W_{\textrm{loc}}^{\textup{c}}(0)$ leads to the $(k+1)$-dimensional map in \eqref{eq:restricted_map}, including the properties and expressions in \eqref{eq:N_cm}, \eqref{eq:tilde_f_properties} and \eqref{eq:G_cm}, after further Taylor expansion in $\eps$. 
		
		It remains to show that the restricted map \eqref{eq:restricted_map} is a diffeomorphism with a unipotent regular contact point at 
		$(0,0) \in \R^{k+1}$. Evaluating the Jacobian matrix when $\eps = 0$ at $(0,0)$, we obtain
		\begin{equation}
			\label{eq:linearisation_restricted_map}
			\D \widehat H|_{W_{\textrm{loc}}^{\textup{c}}(0)}(0,0,0) = 
			\mathbb I_{k+1} + \widetilde N(0,0) \D \widetilde f(0,0) 
			= \mathbb I_{k+1} + 
			\begin{pmatrix}
				\mathbb O_{k,k} & \widehat N^x(0,0) \\
				\mathbb O_{1,k} & \widehat N^u(0,0)
			\end{pmatrix}
			, 
		\end{equation}
		which has $k$ trivial multipliers equal to $1$ and a single non-trivial multiplier satisfying
		\begin{equation}
			\label{eq:mu_1}
			\mu_1(0,0) 
			= 1 + \widetilde N^u(0,0) 
			= 1 + l \D f N \left( r + P W_0(0,0) \right) (0,0) = 1 ,
		\end{equation}
		where we used the fact that $l \D f N (0,0) = 0$ ($l$ is a left null vector of $\D f N$ at $z = 0$). Thus $(0,0)$ is a contact point for the map \eqref{eq:reduced_map}. Since the matrix appearing on the right-hand side of \eqref{eq:linearisation_restricted_map} is nilpotent (its square is the zero matrix), $(0,0)$ is a unipotent contact point for \eqref{eq:restricted_map}. The fact that $(0,0)$ is a regular contact point follows after verifying the genericity conditions in Definition \ref{def:regular_contact_point}. The transversality condition can be checked directly, while the nondegeneracy and slow regularity conditions can be shown to follow from the satisfaction of the nondegeneracy and slow regularity conditions in the original map \eqref{eq:cm_normal_form} on $\R^n$; the details are standard, and omitted for brevity.
		
		Finally, equation \eqref{eq:mu_1} is also sufficient to apply the inverse function theorem. This implies that $\mathcal W$ and $\eps_0 > 0$ can be fixed sufficiently small for $H|_{W_{\textrm{loc}}^{\textup{c}}(0)}$ to be a diffeomorphism on $\mathcal W \times [0,\eps_0)$.
		%
		%
	\end{proof}
	
	It is significant to note that the restricted map \eqref{eq:restricted_map} describing dynamics on the local center manifold $W^{\textup{c}}_{\textup{loc}}(0)$ is a diffeomorphism, and that the Jacobian matrix \eqref{eq:linearisation_restricted_map} is unipotent. Thus, Theorem \ref{thm:embedding_nilpotent} applies, even though it does not apply directly in the original map on $\R^n$ if $\textup{codim} (S) \geq 2$, i.e.~prior to restriction to $W_{\textrm{loc}}^{\textup{c}}(0)$.
	

	\subsection{Formal embedding and dynamics on the center manifold}
	
	Applying Theorem \ref{thm:embedding_nilpotent}, we obtain a formal embedding result which can be used to approximate the dynamics of the map \eqref{eq:restricted_map} on $W_{\textrm{loc}}^{\textup{c}}(0)$.
	
	\begin{thm}
		\label{thm:center_manifold_embedding}
		Consider the restricted map \eqref{eq:restricted_map} on $W_{\textrm{loc}}^{\textup{c}}(0)$, as described in Lemma \ref{lem:center_manifold}. There exists a neighbourhood $\mathcal W \ni 0$ in $\R^{k+1}$ and an $\eps_0 > 0$ such that for all $(x,u,\eps) \in \mathcal W \times [0,\eps_0)$ and $m \in \{1, \ldots, r\}$, where $r \geq 3$, we have
		\begin{equation}
			\label{eq:contact_pt_approx}
			j^m \widehat H|_{W_{\textup{loc}}^{\textup{c}}(0)}(x,u,\eps) = j^m \Phi_V^1(x,u,\eps) ,
		\end{equation}
		where $\Phi_V^1(x,u,\eps)$ denotes the time-$1$ flow of an $\eps$-family of vector fields $V(x,u,\eps)$ on $\R^{k+1}$ satisfying
		\begin{equation}
			\label{eq:contact_pt_embedding}
			j^r V(x,u,\eps) = \mathfrak N(x,u) j^r \widetilde f(x,u) + \eps \mathfrak G(x,u,\eps) ,
		\end{equation}
		where $\mathfrak N$ and $\mathfrak G$ have the same dimensions, regularity and smoothness as $(\widetilde N^x, \widetilde N^u)^\transpose$ and $(\widetilde G^x, \widetilde G^u)^\transpose$ respectively, and $j^r \widetilde f$ is the $r$-jet associated to the Taylor expansion of $\widetilde f(x,u)$ about $(x,u) = (0,0)$. Moreover,
		\begin{equation}
			\label{eq:contact_pt_linear_part}
			\mathfrak N(0,0) = 
			\begin{pmatrix}
				\widetilde N^x(0,0) \\
				\widetilde N^u(0,0) 
			\end{pmatrix} ,
			\qquad 
			j^r \widetilde f(x,u) = u \left(1 + O(|x|, u) \right) ,
			\qquad 
			\mathfrak G(0,0,0) = 
			\begin{pmatrix}
				\widetilde G^x(0,0,0) \\
				\widetilde G^u(0,0,0) 
			\end{pmatrix} .
		\end{equation}
		The vector field \eqref{eq:contact_pt_embedding} is fast-slow with a $k$-dimensional critical manifold
		\begin{equation}
			\label{eq:contact_pt_S}
			j^r S := \left\{ (x,u) \in \R^{k+1} : j^r \widetilde f(x,u) = 0 \right\} ,
		\end{equation}
		and a regular contact point at $(0,0) \in j^r S$.
	\end{thm}
	
	\begin{proof}
		The proof is similar to the proof of Proposition \ref{prop:2d_embedding}. 
		Since Assumptions \ref{ass:fast-slow}-\ref{ass:factorisation} can be checked directly, and $(0,0)$ is a unipotent contact point by Lemma \ref{lem:center_manifold}, we can apply Theorem \ref{thm:embedding_nilpotent} directly. 
		Doing so yields the approximation property \eqref{eq:contact_pt_approx} and the expressions in \eqref{eq:contact_pt_embedding}, \eqref{eq:contact_pt_linear_part} and \eqref{eq:contact_pt_S}.
		
		It remains to show that $(0,0)$ is a regular contact point for the truncated vector field \eqref{eq:contact_pt_embedding}. It suffices to 
		verify the defining conditions from Definition \ref{def:regular_contact_point} at $(0,0)$, which are analogous for fast-slow ODE systems; we refer again to \cite[Sec.~4.1 and 4.2]{Wechselberger2019} for definitions. The rank condition \eqref{eq:contact_cond} is satisfied since
		\[
		D (j^r \widetilde f) \mathfrak N(0,0) = D \widetilde f \widetilde N(0,0) = 0 \quad 
		\implies \quad 
		\textup{rk}\ \left( D (j^r \widetilde f) \mathfrak N(0,0) \right) = (k+1)-k-1 = 0 ,
		\]
		where we used the left-most expression in \eqref{eq:contact_pt_linear_part} and the fact that $D \widetilde f(0,0) = D (j^r \widetilde f)(0,0)$, since $r \geq 3 > 1$.
		
		The transversality condition can also be checked directly. Using the expression for $j^r \widetilde f(x,u)$ in \eqref{eq:contact_pt_linear_part} we obtain
		\[
		\textup{rk} (D (j^r \widetilde f) (0,0) ) = 
		\textup{rk} 
		\begin{pmatrix}
			\mathbb O_{1,k} & 1 
		\end{pmatrix}
		= 1 ,
		\]
		as required. 
		
		The nondegeneracy condition reduces to the requirement that
		\begin{equation}
			\label{eq:nondegeneracy}
			\sum_{l,m = 1}^{k+1} \left( \frac{\partial^2 (j^r \tilde f)}{\partial \xi_l \partial \xi_m}(0,0) \mathfrak N_m(0,0) \mathfrak N_l(0,0) + \frac{\partial (j^r \widetilde f)}{\partial \xi_l}(0,0) \frac{\partial \mathfrak N_l}{\partial \xi_m}(0,0) \mathfrak N_m(0,0) \right) \neq 0 ,
		\end{equation}
		where we introduced the componentwise notation $\xi = (\xi_1, \ldots, \xi_k, \xi_{k+1}) = (x_1, \ldots, x_k, u)$ and $\mathfrak N = (\mathfrak N^x_1, \ldots, \mathfrak N^x_k, \mathfrak N^u)^\transpose$. It follows from the form of $j^r \widetilde f(x,u)$ in \eqref{eq:contact_pt_linear_part} and the fact that $\mathfrak N^u (0,0) = \widetilde N^u(0,0) = 0$ that
		\[
		\frac{\partial^2 (j^r \tilde f)}{\partial \xi_l \partial \xi_m}(0,0) \mathfrak N_m(0,0) \mathfrak N_l(0,0) = 0
		\]
		for all $l, m = 1, \ldots, k+1$. Moreover, the terms on the right can be rewritten as
		\[
		\frac{\partial (j^r \widetilde f)}{\partial \xi_l}(0,0) \frac{\partial \mathfrak N_l}{\partial \xi_m}(0,0) \mathfrak N_m(0,0)
		= \frac{\partial \widetilde f}{\partial \xi_l}(0,0) \frac{\partial \mathfrak N_l}{\partial \xi_m}(0,0) \widetilde N_m(0,0) ,
		\]
		for all $l, m = 1, \ldots, k+1$. Thus, the nondegeneracy condition in \eqref{eq:nondegeneracy} reduces to
		\begin{equation}
			\label{eq:nondegeneracy_2}
			\sum_{l,m = 1}^{k+1} \frac{\partial \widetilde f}{\partial \xi_l}(0,0) \frac{\partial \mathfrak N_l}{\partial \xi_m}(0,0) \widetilde N_m(0,0) \neq 0 .
		\end{equation}
		Finally, one can show that
		\begin{equation}
			\label{eq:partials_equality}
			\frac{\partial \mathfrak N_l}{\partial \xi_m}(0,0) = \frac{\partial \widetilde N_l}{\partial \xi_m}(0,0) ,
		\end{equation}
		after which \eqref{eq:nondegeneracy_2} reduces to the nondegeneracy condition for the restricted map \eqref{eq:restricted_map}, which is satisfied. The equality in \eqref{eq:partials_equality} can be obtained by matching coefficients in truncated Taylor expansions of \eqref{eq:restricted_map} and the time-1 map in \eqref{eq:contact_pt_approx}. This calculation is standard but lengthy, so we defer it to Appendix \ref{app:partials_equality} for expository reasons.
		
		
		It remains to check the slow regularity condition. Since $\D (j^r \widetilde f) \widehat N (0,0)$ is a scalar we have that $r = l = 1$. Using this together with \eqref{eq:contact_pt_linear_part}, the slow regularity condition reduces on the vector field \eqref{eq:contact_pt_embedding} reduces to the slow regularity condition on the map \eqref{eq:restricted_map}, and we obtain
		\[
		\mathfrak N \D (j^r \widetilde f) \mathfrak G(0,0,0) = \widetilde N \D \widetilde f \widetilde G(0,0,0) \neq 0 ,
		\]
		as required.
	\end{proof}
	
	Theorem \ref{thm:center_manifold_embedding} shows that the dynamics on the center manifold $W_{\textrm{loc}}^{\textup{c}}(0)$ can be locally approximated by the time-$1$ map induced by a $(k+1)$-dimensional fast-slow ODE system with a regular fold point at $(0,0)$. It is significant to note that Theorem \ref{thm:center_manifold_embedding} also allows for an approximation of the dynamics of the original map \eqref{eq:general_maps_coords_2} in $\R^n$, if the non-critical non-trivial multipliers $\mu_j(z)$ satisfy
	\[
	|\mu_j(0)| < 1 , \qquad \forall j = 2, \ldots, n-k,
	\]
	due to the local exponential attractivity of $W_{\textrm{loc}}^{\textup{c}}(0)$ in this case. The dynamics near regular fold points in general $(k+1)$-dimensional fast-slow ODE systems in canonical local normal forms have been described in \cite{Wechselberger2012,Wechselberger2019}  (the case $k = 2$ was already treated in \cite{Szmolyan2004}). Although we do not consider the geometry and dynamics in detail here, we conjecture that the approximation result in Theorem \ref{thm:center_manifold_embedding} can be combined with the results in \cite{Wechselberger2012,Wechselberger2019} in order to describe certain properties of the geometry and dynamics close to the contact point. In particular, we expect that the extension of the attracting slow manifold can be described using arguments which are similar to those applied in the context of regular fold, transcritical and pitchfork points in Section \ref{sec:2d_applications}; recall in particular the proofs for Corollaries \ref{cor:fold}, \ref{cor:transcritical} and \ref{cor:pitchfork}.

	\section{Summary and Outlook}
	\label{sec:summary_and_outlook}
	
	The following three features have been fundamental to the success of continuous-time GSPT as an approach to the study of fast-slow ODE systems: 
	\begin{itemize}
		\item[(I)] A systematic geometric formalism for identifying and analysing simpler limiting problems associated to each time-scale;
		\item[(II)] A set of perturbation theorems in the normally hyperbolic regime, as provided by Fenichel theory;
		\item[(III)] A set of independent results which describe the local dynamics close to non-normally hyperbolic singularities.
	\end{itemize}
	The primary aim of this manuscript has been to continue the development of DGSPT, based on the principle that (I)-(III) should also be viewed as cornerstones for a geometric approach to the study of discrete fast-slow systems induced by repeated iteration of a map. One can also view this work as a natural continuation of the work in \cite{Jelbart2022a}, which focused on (I)-(II) via the development of DGSPT for general fast-slow maps \eqref{eq:general_maps} in the normally hyperbolic regime. 
	
	In Section \ref{sub:slow_embeddings} we presented two new results in the normally hyperbolic regime, namely Theorems \ref{thm:nh_approximation} and \ref{thm:reduced_map}. Both of these results can be used in order to approximate the dynamics on slow manifolds via the time-$1$ map induced by the flow of a vector field in the same dimension. Theorem \ref{thm:reduced_map} in particular describes a close relationship to the reduced vector field which appears as the limiting problem for the dynamics on the critical manifold in the continuous-time setting. Not only does this elucidate a close connection to continuous-time theory on the slow time-scale, but, for many purposes, it provides a means of reducing the study of the map to the study of a vector field which is known.
	
	Our primary focus, however, was to contribute to the understanding of dynamics near non-normally hyperbolic points, i.e.~we focused primarily on (III). As we saw in Section \ref{sub:loss_of_normal_hyperbolicity}, codimension-$1$ non-normally hyperbolic singularities can be split into three important types: fold/contact points in $\mathcal C \subset S$, flip/period-doubling points in $\mathcal O_{\textup{f}} \subset S$, and torus/Neimark-Sacker points in $\mathcal O_{\textup{ns}} \subset S$. These correspond to points at which the matrix $\mathbb I_{n-k} + \D fN$ has a single multiplier at $+1$, $-1$, or a pair of complex conjugate multipliers on $S^1 \setminus \{\pm 1\}$, respectively. The so-called `oscillatory singularities' in $\mathcal O_{\textup{f}} \cup \mathcal O_{\textup{ns}}$ can often be analysed using techniques which rely on the invertibility of the matrix $\D f N$, we refer again to \cite{Baesens1991,Baesens1995,Fruchard2003,Fruchard2009,Neishtadt1996,Neishtadt1987} and the references therein. We therefore chose to focus on the loss of normal hyperbolicity near singularities in $\mathcal C$, where the matrix $\D fN$ is singular.
	
	Our most important theoretical result on the dynamics near singularities in $\mathcal C$ is Theorem \ref{thm:embedding_nilpotent}, which shows that the dynamics near unipotent singularities, which form an important subset of $\mathcal C$, can be approximated by the time-$1$ map of a formal vector field in the same dimension. This vector field is fast-slow with a nilpotent singularity and a critical manifold $j^r S$ that is $C^r$-close to the critical manifold of the original map. 
	The proof of Theorem \ref{thm:embedding_nilpotent} relied on direct arguments and a suitable application of the Takens embedding theorem. 
	
	In order to demonstrate the applicability of Theorem \ref{thm:embedding_nilpotent}, we used it in order to generate results on the geometry and dynamics near non-normally hyperbolic points of regular fold, transcritical and pitchfork type in planar fast-slow maps in standard form; recall Section \ref{sec:2d_applications}. We showed in Proposition \ref{prop:2d_embedding} that $2$-dimensional fast-slow maps in standard form \eqref{eq:standard_form_2d} can be approximated by the time-$1$ map of a planar fast-slow ODE system in standard form, and that the approximating vector field has a non-normally hyperbolic point of regular fold, transcritical or pitchfork type if the original map has a non-normally hyperbolic point of the corresponding type. Using this result and the established results from continuous-time theory in \cite{Krupa2001a,Krupa2001c}, we characterised the extension of attracting slow manifolds through a neighbourhood of regular fold, transcritical and pitchfork points in the map \eqref{eq:standard_form_2d} in Corollaries \ref{cor:fold}, \ref{cor:transcritical} and \ref{cor:pitchfork} respectively. It is significant to note that the established results in the continuous-time setting, i.e.~those in \cite{Krupa2001a,Krupa2001c}, were derived using the geometric blow-up method, which is either not applicable or not yet developed for the study of singularities in general fast-slow maps (with the exception of discretized systems, recall Remarks \ref{rem:blow-up} and \ref{rem:blow-up_2}). This points to a major advantage of approximation results like Proposition \ref{prop:2d_embedding}, namely, that they provide a means for analysing certain features of the map using methods that are only developed or applicable in the continuous-time setting.
	
	Finally in Section \ref{sec:regular_contact_points}, we showed that Theorem \ref{thm:embedding_nilpotent} could be also used to approximate the dynamics near regular contact points of fast-slow maps in general nonstandard form \eqref{eq:general_maps} in $\R^n$, even though contact points in $\mathcal C$ are not generically unipotent if $\textup{codim}(S) \geq 2$, i.e.~if $n - k \geq 2$. The key observation is that fast-slow maps with a codimension-1 regular contact point admit of a local center manifold reduction. The center manifold $W_{\textrm{loc}}^{\textup{c}}$ is $(k+1)$-dimensional, and $\textup{codim}(S) = 1$ in the restricted map $H|_{W_{\textrm{loc}}^{\textup{c}}}$ which governs dynamics on $W_{\textrm{loc}}^{\textup{c}}$; this was shown in Lemma \ref{lem:center_manifold}. The map $H|_{W_{\textrm{loc}}^{\textup{c}}}$ is also fast-slow with a regular contact point, but in this case, the contact point is also unipotent. This allowed for the application of Theorem \ref{thm:embedding_nilpotent}, which in turn allowed us to show that the dynamics of $H|_{W_{\textrm{loc}}^{\textup{c}}}$ can be approximated by the time-$1$ map of a $(k+1)$-dimensional fast-slow system with a regular contact point. This is summarised in Theorem \ref{thm:center_manifold_embedding}.
	
	\
	
	There are many open questions remaining. For example, Corollaries \ref{cor:fold}, \ref{cor:transcritical} and \ref{cor:pitchfork} describe the extension of attracting slow manifolds through a neighbourhood of different types of non-normally hyperbolic points, but we did not present results on the dynamics of larger sets of initial conditions. Another open question pertains to the extent to which the approximation error appearing in all of the formal embedding theorems presented herein can be improved. In this regard it is worthy to reiterate that exponentially small errors may be unavoidable, recall Remark \ref{rem:Ilyashenko}, which suggests that the details of exponentially small separations which are typical in canard theory, for example, cannot be simply described with formal approximations of the kind presented herein.
	
	More generally, the development of DGSPT is closely tied to Guckenheimer's program for the development of a `global' GSPT \cite{Guckenheimer1996} which is general enough to apply to the study of oscillatory fast-slow systems, i.e.~fast-slow ODE systems featuring a manifold or manifolds of limit cycles in the layer problem. The connection arises because DGSPT provides a geometric framework for studying the fast-slow Poincar\'e maps which arise naturally in this setting. We believe that \cite{Jelbart2022a} and the present manuscript constitute significant progress towards a rigorous geometric theory for global oscillatory fast-slow systems, which has implications for the geometric study of models of complex oscillatory dynamics in a wide range of areas including mathematical neuroscience, biology and chemistry. The exact nature of this connection is left for future work.

	\subsection*{Acknowledgements}
	
	SJ and CK acknowledge funding from the SFB/TRR 109 Discretization and Geometry in Dynamics, i.e.~Deutsche Forschungsgemeinschaft (DFG - German Research Foundation) - Project-ID 195170736 - TRR109. \rev{SJ recieved additional funding via a European Union Marie Sk{\l}odowska-Curie Postoctoral Fellowship - Grant Agreement ID 101103827.} CK thanks the VolkswagenStiftung for support via a Lichtenberg Professorship.

	\bibliographystyle{siam}
	\bibliography{extending_dgspt}

	\appendix
	
	\section{Hadamard's lemma}
	\label{app:Hadamard}
	
	Here we state \rev{a} version of \textit{Hadamard's lemma}, see e.g.~\cite{Nestruev2003} for a formulation suited to our purposes. A simple proof can be given using Taylor's theorem.
	
	\begin{lemma}
		\textup{(Hadamard's lemma)} 
		Let $f = (f_1, \ldots, f_m)^\transpose : \R^n \to \R^{m}$ be a $C^1$-smooth function at $z \in \R^n$ which satisfies
		\[
		f(z) = 0 , \qquad 
		\textup{rk} (\textup{D} f(z)) = m .
		\]
		For every $C^1$-smooth function $\rho : U \to \R$, where $U$ is a neighbourhood of $z$ in $\R^n$, there exists a neighbourhood $\mathcal U \subseteq U$ of $z$ and continuous functions $\sigma_j : \mathcal U \to \R$, $j = 1, \ldots, m$ such that
		\[
		\rho = \sum_{j = 1}^{m} \sigma_j f_j .
		\]
		If $\rho$ and $f$ are $C^r$-smooth, there exists a $C^r$-smooth choice of $\sigma_j$'s.
	\end{lemma}

	\section{Takens' embedding theorem}
	\label{app:Takens_embedding_theorem}
	
	In the following we state the version of Takens' embedding theorem given in \cite[Theorem 8.1]{Chow1994} (see also \cite[Prop.~6.1]{Gramchev2005}), as well as a corollary which specialises this result for unipotent points. \rev{Takens' own result can be found in \cite[Thm.~4]{Takens1973}. This is similar to the version that we use, except that it applies to scalar diffeomorphisms only. Numerous (albeit simpler) versions of the result predate Takens' result:} An early reference \rev{where a similar result for scalar diffeomorphisms with linear part identity appears} without proof 
	in \cite{Bouten1916}\rev{, and later with a proof in \cite{Lewis1939}. An early formulation of a similar result in $n$-dimensions} 
	can be found in \cite{Chen1965}.
	
	\
	
	Consider a $C^r$-diffeomorphism $F : \R^n \to \R^n$, with $r \geq 2$ and a fixed point at the origin, i.e.~$F(0) = 0$. Taylor expansion gives
	\begin{equation}
		\label{eq:F}
		F(x) = A x + F^2(x) + F^3(x) + \dots + F^r(x) + o(|x|^r) ,
	\end{equation}
	where the $F^l$ are homogeneous polynomials of order $l \in \{2, \ldots, r\}$ in $x \in \R^n$. The extension to parameter-dependent maps is straightforward; one simply considers the parameters $\lambda \in \R^p$ as variables by appending the trivial map(s) $\lambda \mapsto \lambda$ to $F$. Using a Jordan-Chevalley decomposition, the Jacobian matrix $A$ can be written as
	\[
	A = B (\mathbb I_n + M) ,
	\]
	where the matrices $B$ and $M$ are semi-simple and nilpotent respectively, and $B M = M B$. With this identification, Takens' theorem can be stated as follows.
	
	\begin{thm}
		\label{thm:Takens}
		\textup{(Takens embedding theorem)}
		Consider the diffeomorphism \eqref{eq:F}. For each $l \in \{2, \ldots, r\}$ there exists a diffeomorphism $\psi_l : \Omega \subseteq \R^n \to \R^n$ and a vector field $X : \R^n \to \textup{T}\R^n$ such that
		\begin{enumerate}
			\item[(i)] $j^l(\psi_l \circ F \circ \psi_l^{-1})$ is a Poincar\'e normal form of $F$ up to order $l$;
			\item[(ii)] $X(Bx) = B X(x)$ for all $x \in \R^n$;
			\item[(iii)] $j^l(\psi_l \circ F \circ \psi_l^{-1}) = j^l(\Phi_X^1(Bx))$, where $\Phi^t_X(x)$ denotes the time-$t$ flow map induced by $X(x)$;
			\item[(iv)] The truncated vector field $j^l X(x)$ is uniquely determined by $j^l F(x)$.
		\end{enumerate}
	\end{thm}
	
	In addition to Theorem \ref{thm:Takens}, we present a well-known corollary for the unipotent case $B = \mathbb I_n$. This case arises frequently throughout this work. The following formulation can be found in \cite[Lem.~8.2]{Chow1994}.
	
	\begin{corollary}
		\label{cor:Takens}
		Assume the same conditions as in Theorem \ref{thm:Takens} and, additionally, that $B = \mathbb I_n$ so that $A = \mathbb I_n + M$. Then there exists a vector field $X : \R^n \to \textup{T}\R^n$ such that $j^r X(x)$ is uniquely determined by $j^r F(x)$, and the flow $\Phi^t_X(x)$ satisfies
		\[
		j^r \Phi^1_X(x) = j^r F(x) .
		\]
	\end{corollary}

	\section{Proof of Theorem \ref{thm:reduced_map}}
	\label{app:reduced_map}
	
	Theorem \ref{thm:reduced_map} can be proven formally and systematically by matching coefficients in Taylor expansions for the slow map \eqref{eq:slow_map} and the time-$1$ map induced by a (presently unknown) vector field in the same dimension. 
	
	We consider the slow map \eqref{eq:slow_map} 
	in an extended space including $\eps$. We write
	\begin{equation}
		\label{eq:slow_map_extended}
		\begin{pmatrix}
			z \\ 
			\eps
		\end{pmatrix}
		\mapsto 
		H(z,\eps) =
		\begin{pmatrix}
			z \\
			\eps
		\end{pmatrix}
		+ \eps
		\begin{pmatrix}
			\Pi_{\mathcal N}^{S} G(z,0) + O(\eps)  \\
			0
		\end{pmatrix} .
	\end{equation}
	The function $H$ should not be confused with that appearing in the main part of the text, e.g.~in \eqref{eq:general_maps}. We assume without loss of generality that $z = 0 \in S$, and we write $x := (z,\eps) \in \mathbb R^{n+1}$ so that \eqref{eq:slow_map_extended} can be written as the following expansion about $x = 0$:
	\begin{equation}
		\label{eq:slow_map_expansion}
		x \mapsto H(x) = A x + H^{(2)}(x) + \ldots + H^{(r)}(x) + o(|x|^r) ,
	\end{equation}
	where
	\begin{equation}
		\label{eq:Lambda}
		A = \mathbb I_{n+1} + \Lambda , \qquad
		\Lambda = 
		\begin{pmatrix}
			\mathbb O_{n,n} & G(0) \\
			\mathbb O_{1,n} & 0
		\end{pmatrix} ,
	\end{equation}
	and each $H^{(j)}$ is a homogeneous polynomial of degree $j \in \{2, \ldots, r\}$. Our aim is to compare \eqref{eq:slow_map_expansion}, which is known, with an unknown vector field in the same dimension which has an equilibrium at $x=0$ and a formal expansion of the form
	\[
	x' = F(x) = \Lambda x + F^{(2)}(x) + \ldots + F^{(r)}(x) + o(|x|^r) ,
	\]
	where each $F^{(j)}$ is a 
	homogeneous polynomials of degree $j$. We will often use componentwise notation/definitions $H^{(k)} = (H^{(k)}_1, \ldots, H^{(k)}_n, 0)^\transpose$ and $F^{(k)} = (F^{(k)}_1, \ldots, F^{(k)}_n, 0)^\transpose$, where
	\begin{equation}
		\label{eq:HF_series}
		\begin{split}
			H_i^{(k)}(x) &= \sum_{j_1 + \ldots + j_{n+1} = k} b^{(k)}_{i, j_1 \cdots j_{n+1}} z_1^{j_1} \cdots z_n^{j_n} \eps^{j_{n+1}} , \\
			F_i^{(k)}(x) &= \sum_{j_1 + \ldots + j_{n+1} = k} B^{(k)}_{i, j_1 \cdots j_{n+1}} z_1^{j_1} \cdots z_n^{j_n} \eps^{j_{n+1}} ,
		\end{split}
	\end{equation}
	for each $i = 1, \ldots, n$ and $k = 2, \dots, r$. Note that each $b^{(k)}_{i, j_1 \cdots j_n 0} = 0$, since the nonlinear part of \eqref{eq:slow_map_extended} is factored by $\eps$. 
	
	Since $\Lambda$ is nilpotent with $\Lambda^2 = \mathbb O_{n+1,n+1}$, Corollary \ref{cor:Takens} 
	implies that there is a unique choice of $F(x)$ such that (i) $A = \me^\Lambda = \mathbb I_{n+1} + \Lambda$, and (ii) $H(x)$ is locally approximated up to an error of size $o(|x|^r)$, i.e.~
	\[
	H(x) = \Phi_F^1(x) + R(x) , 
	\]
	where $\Phi_F^1(x)$ denotes the time-1 map of the flow induced by the vector field $F(x)$ and the error of this approximation is given by $R(x) = o(|x|^r)$. 
	
	The claim in Theorem \ref{thm:reduced_map} is that to leading order in $\eps$, the vector field $F(x)$ coincides with the slow vector field
	\[
	\begin{split}
		z' &= \eps \Pi_{\mathcal N}^{S} G(z,0) + O(\eps^2) , \\
		\eps' &= 0 .
	\end{split}
	\]
	In order to show this, we match $l$-jets of $F(x)$ and $\Phi^1(x)$ for all $l = 1, \ldots, r$, where the $l$-jet of the time-1 map $\Phi^1(x)$ is obtained via Picard iteration; see \cite{Chow1994,Gramchev2005,Kuznetsov2013,Kuznetsov2019} for similar approaches in differing contexts.
	
	The $1$-jet of $\Phi^{(t)}(x)$ about $x=0$ is
	\begin{equation}
		\label{eq:1-jet}
		x^{(1)}(t) = e^{\Lambda t} x = (\mathbb I_{n+1} + \Lambda t) x =
		\begin{pmatrix}
			z + \eps a t \\
			\eps
		\end{pmatrix} ,
	\end{equation}
	and each $(l \geq 2)$-jet is defined via the recursive formula
	\begin{equation}
		\label{eq:l_jet}
		\begin{split}
			x^{(l)}(t) &= e^{\Lambda t} x + \int_0^t e^{\Lambda (t - \tau)}	 \left( F^{(2)}(x^{(l-1)} (\tau)) + \ldots + F^{(l)}(x^{(l-1)} (\tau)) \right) d\tau \\
			&= 
			\begin{pmatrix}
				z + \eps a t \\
				\eps
			\end{pmatrix}
			+
			\int_0^t (\mathbb I_{n+1} + \Lambda (t - \tau) )
			\begin{pmatrix}
				\tilde F^{(2)}(x^{(l-1)} (\tau)) + \ldots + \tilde 	F^{(l)}(x^{(l-1)} (\tau)) \\ 
				0
			\end{pmatrix}
			d\tau \\
			&= 
			\begin{pmatrix}
				z + \eps a t \\
				\eps
			\end{pmatrix}
			+
			\begin{pmatrix}
				\mathbb I_n & \mathbb O_{n, 1} \\
				\mathbb O_{1, n} & 0
			\end{pmatrix}
			\int_0^t \left( \tilde F^{(2)}(x^{(l-1)} (\tau)) + \ldots + \tilde F^{(l)}(x^{(l-1)} (\tau)) \right) d\tau  ,
		\end{split}
	\end{equation}
	where we used the expression for $\Lambda$ in \eqref{eq:Lambda}, and the $\tilde F^{(l)}$ are defined via $F^{(l)}(x) := (\tilde F^{(l)}(x) , 0)^T$ for each $l = 2, \ldots, r$.
	
	\
	
	We shall prove the following result, which is sufficient to prove both expressions in \eqref{eq:approximation} and therefore Theorem \ref{thm:reduced_map}.
	
	\begin{lemma}
		\label{lem:red_problem_coefficients}
		For each $l = 2, \ldots , r$ and $i = 1, \ldots, n+1$ we have
		\[
		H_i^{(l)}(z,0) + \eps \frac{\partial H_i^{(l)}}{\partial \eps} (z,0) =
		F_i^{(l)}(z,0) + \eps \frac{\partial F_i^{(l)}}{\partial \eps} (z,0) .
		\]
		Equivalently, for each $l = 2 \ldots, r$ and $i = 1, \ldots , n+1$ we have equality between terms with linear or no dependence on $\eps$, i.e.~
		\begin{equation}
			\label{eq:coefficient_equality}
			B^{(l)}_{i, j_1 \cdots j_{n+1}} = b^{(l)}_{i, j_1 \cdots j_{n+1}} ,
		\end{equation}
		whenever $j_{n+1} \in \{0,1\}$. In particular, $B^{(l)}_{i, j_1 \cdots j_n 0} = b^{(l)}_{i, j_1 \cdots j_n 0} = 0$.
	\end{lemma}
	
	\begin{proof}
		We prove the latter claim (equality of coefficients) by induction.
		
		\medskip
		
		\noindent \textbf{Base case:} $l=2$. At quadratic order $l=2$ we need to solve
		\[
		\int_0^1 e^{\Lambda (1-\tau)}F^{(2)}(x^{(1)}(\tau)) d\tau = H^{(2)}(x) 
		\]
		for the coefficients $B^{(2)}_{i, j_1 \cdots j_{n+1}}$ with $j_{n+1} \in \{0,1\}$. Substituting explicit expressions for $\me^{\Lambda (t - \tau)} = \mathbb I_{n+1} + \Lambda (1 - \tau)$ and $x^{(1)}(\tau)$, recall \eqref{eq:Lambda} and \eqref{eq:1-jet}, leads to the simplified equation
		\[
		\int_0^1 F^{(2)}(z + \eps a \tau) d\tau = H^{(2)}(z,\eps) ,
		\]
		where we introduced $a = (a_1, \ldots, a_{n+1}) := (G(0),0)^\transpose$. After substituting the series expansions for components of $F$ and $H$ in \eqref{eq:HF_series} we obtain
		\begin{equation}
			\label{eq:base_match}
			\begin{split}
				\sum_{j_1 + \ldots + j_{n+1} = 2} \int_0^1  B^{(2)}_{i, j_1 \cdots j_{n+1}} (z_1 + \eps a_1 \tau)^{j_1} & \cdots (z_n +	 \eps a_n \tau)^{j_n} \eps^{j_{n+1}} \\
				&= \sum_{j_1 + \ldots + j_{n+1} = 2} b^{(2)}_{i, j_1 \cdots j_{n+1}} z_1^{j_1} \cdots z_n^{j_n} \eps^{j_{n+1}} .
			\end{split}
		\end{equation}
		Terms in the left-hand side can be integrated directly, with the form of the resulting expression differing in the cases $j_{n+1} = 0, 1$ or $2$. Terms with $j_{n+1} = 0$ have the form
		\[
		B^{(2)}_{i, j_1 \cdots j_n 0} \int_0^1 (z_{s_1} + \eps a_{s_1} \tau) (z_{s_2} + \eps a_{s_2} \tau) d\tau =
		B^{(2)}_{i, j_1 \cdots j_n 0} \left( z_{s_1} z_{s_2} + \frac{1}{2} \eps (z_{s_1} a_{s_2} + z_{s_1} a_{s_2} ) + \frac{1}{3} \eps^2 a_{s_1} a_{s_2} \right) ,
		\]
		where $s_1, s_2 \in \{1,\ldots,n\}$ satisfy $j_1 + \ldots + j_n = j_{s_1} + j_{s_2} = 2$. Terms with $j_{n+1} = 1$ have the form
		\[
		\eps B^{(2)}_{i, j_1 \cdots j_n 1} \int_0^1 (z_{s} + \eps a_{s} \tau) d\tau =
		\eps B^{(2)}_{i, j_1 \cdots j_n 1} \left( z_s + \frac{1}{2} \eps a_s \right) ,
		\]
		for $s \in \{1,\ldots,n\}$ such that $j_1 + \ldots + j_n = j_s = 1$. Although we will not need them, we note for completeness that terms with $j_{n+1} = 2$ have the form
		\[
		\int_0^t B^{(2)}_{i, 0 \cdots 0 2} \eps^2 d\tau = 
		\eps^2 B^{(2)}_{i, 0 \cdots 0 2} ,
		\]
		where $j_1 = \ldots = j_n = 0$ (there is only one such term). Using the above, we may rewrite \eqref{eq:base_match} as follows:
		\begin{equation}
			\label{eq:base_series}
			\begin{split}
				\eps^2 B^{(2)}_{i, 0 \cdots 0 2} \ &+ \  
				\eps \sum_{j_1 + \ldots + j_n = j_s = 1} 
				B^{(2)}_{i, j_1 \cdots j_n 1} \left( z_s + \frac{1}{2} \eps a_s \right) + \\
				&\sum_{j_1 + \ldots + j_n = j_{s_1} + j_{s_2} = 2} B^{(2)}_{i, j_1 \cdots j_n 0} \left( z_{s_1} z_{s_2} + \frac{1}{2} \eps (z_{s_1} a_{s_2} + z_{s_1} a_{s_2} ) + \frac{1}{3} \eps^2 a_{s_1} a_{s_2} \right) \\
				&= \eps^2 b^{(2)}_{i, 0 \cdots 0 2} + \eps \sum_{j_1 + \ldots + j_n = 1} b^{(2)}_{i, j_1 \cdots j_n 1} z_1^{j_1} \cdots z_n^{j_n} ,
			\end{split}
		\end{equation}
		where the terms factoring $\eps^0$ in the final line do not appear since $H(z,0) = (z,0)^\transpose$ is linear. 
		Matching terms with $j_{n+1} = 0$, we obtain
		\[
		\sum_{j_1 + \ldots + j_n = j_{s_1} + j_{s_2} = 2} B^{(2)}_{i, j_1 \cdots j_n 0} z_{s_1} z_{s_2} = 0 \quad \implies \quad
		B^{(2)}_{i, j_1 \cdots j_n 0} = 0 .
		\]
		Substituting $B^{(2)}_{i, j_1 \cdots j_n 0} = 0$ into \eqref{eq:base_series} and matching terms with $j_{n+1} = 1$, we obtain
		\[
		\eps \sum_{j_1 + \ldots + j_n = j_s = 1} B^{(2)}_{i, j_1 \cdots j_n 1} z_s = \eps \sum_{j_1 + \ldots + j_n = j_s = 1} b^{(2)}_{i, j_1 \cdots j_n 1} z_1^{j_1} \cdots z_n^{j_n} 
		\quad \implies \quad 
		B^{(2)}_{i, j_1 \cdots j_n 1} = b^{(2)}_{i, j_1 \cdots j_n 1} ,
		\]
		whenever $j_1 + \cdots + j_n = 1$. This proves the equality of coefficients requirement in \eqref{eq:coefficient_equality} in the base case $l=2$.
		
		For completeness, we note that substituting the preceding expressions and matching terms at $O(\eps^2)$ gives
		\begin{equation}
			\label{eq:eps2_correction}
			B^{(2)}_{i, 0 \cdots 0 2} = b^{(2)}_{i, 0 \cdots 0 2} - \sum_{j_1 + \ldots + j_n = j_s = 1} b^{(2)}_{i, j_1 \cdots j_n 1} \frac{a_s}{2} .
		\end{equation}
		
		\medskip
		
		\noindent \textbf{Induction step:} Now assume as an induction hypothesis that \eqref{eq:coefficient_equality} holds for all $l = 2, \ldots , k$, for some $k < r$. We need to show that it also holds for $l = k+1$. It follows from the form of equation \eqref{eq:l_jet} that the matching requirement at order $k+1$ is
		\begin{equation}
			\label{eq:IS_int}
			H^{(k+1)}_i(x) - 
			\int_0^1 \tilde F_i^{(k)}(x^{(k)} (\tau)) d\tau = 0,
		\end{equation}
		for each $i = 1, \ldots, n$. 
		In order to evaluate \eqref{eq:IS_int}, we need to simplify $x^{(k)}(\tau) = (z_1^{(k)}(\tau), \ldots z_n^{(k)}(\tau) , \eps)^\transpose$. At order $k$ we have
		\[
		H^{(k)}_i(x) - 
		\int_0^1 \tilde F_i^{(k)}(x^{(k-1)} (\tau)) d\tau = 0 ,
		\]
		for each $i = 1, \ldots, n$. Using the series expressions in \eqref{eq:HF_series}, applying the induction hypothesis and splitting the series into parts which factor $O(\eps^0)$, $O(\eps)$ and $O(\eps^2)$ as in the proof of the base case above, we obtain
		\begin{equation}
			\label{eq:k_matched_series}
			\begin{split}
				& \sum_{j_1 + \cdots + j_n = k} b_{i, j_1 \cdots j_n 0}^{(k)} \left( z_1^{j_1} \cdots z_n^{j_n} - \int_0^1 z_1^{(k)}(\tau)^{j_1} \cdots z_n^{(k)}(\tau)^{j_n} d\tau \right) \\
				&+ \eps \sum_{j_1 + \cdots + j_n = k-1}  b_{i, j_1 \cdots j_n 1}^{(k)} \left( z_1^{j_1} \cdots z_n^{j_n} - \int_0^1 z_1^{(k)}(\tau)^{j_1} \cdots z_n^{(k)}(\tau)^{j_n} d\tau \right) + O(\eps^2) = 0 ,
			\end{split}
		\end{equation}
		which implies that
		\begin{equation}
			\label{eq:Picard_match}
			z_1^{j_1} \cdots z_n^{j_n} = \int_0^1 z_1^{(k)}(\tau)^{j_1} \cdots z_n^{(k)}(\tau)^{j_n} d\tau ,
		\end{equation}
		whenever $j_1 + \cdots + j_n \in \{k-1, k\}$.
		
		\ 
		
		We now return to \eqref{eq:IS_int}. Using the series expressions in \eqref{eq:HF_series}, equation \eqref{eq:IS_int} can be written in a form similar to \eqref{eq:k_matched_series}, namely
		\begin{equation}
			\label{eq:k1_matched_series}
			\begin{split}
				& \sum_{j_1 + \cdots + j_n = k+1} \left( b_{i, j_1 \cdots j_n 0}^{(k+1)} z_1^{j_1} \cdots z_n^{j_n} - B_{i, j_1 \cdots j_n 0}^{(k+1)} \int_0^1 z_1^{(k)}(\tau)^{j_1} \cdots z_n^{(k)}(\tau)^{j_n} d\tau \right) + \\
				& \eps \sum_{j_1 + \cdots + j_n = k} \left( b_{i, j_1 \cdots j_n 1}^{(k+1)} z_1^{j_1} \cdots z_n^{j_n} - B_{i, j_1 \cdots j_n 1}^{(k+1)} \int_0^1 z_1^{(k)}(\tau)^{j_1} \cdots z_n^{(k)}(\tau)^{j_n} d\tau \right) + O(\eps^2) = 0 .
			\end{split}
		\end{equation}
		Since $H(z,0) = (z,0)^\transpose$ is linear, recall \eqref{eq:slow_map_extended}, we have that
		\begin{equation}
			\label{eq:B1}
			b_{i, j_1 \cdots j_n 0}^{(k+1)} = 0 \qquad 
			\implies \qquad 
			B_{i, j_1 \cdots j_n 0}^{(k+1)} = 0,
		\end{equation}
		for all $i = 1, \ldots, n$ and $\{j_l\}_{l=1}^n$ such that $j_1 + \cdots + j_n = k+1$. Moreover, combining \eqref{eq:Picard_match} with the second term in \eqref{eq:k1_matched_series} shows that
		\begin{equation}
			\label{eq:B2}
			B_{i, j_1 \cdots j_n 1}^{(k+1)} = b_{i, j_1 \cdots j_n 1}^{(k+1)} ,
		\end{equation}
		for all $i = 1, \ldots, n$ and $\{j_l\}_{l=1}^n$ such that $j_1 + \cdots + j_n = k$. Equations \eqref{eq:B1} and \eqref{eq:B2} prove the equality of coefficients claim in \eqref{eq:coefficient_equality} at order $k+1$. Lemma \ref{lem:red_problem_coefficients}, and therefore also Theorem \ref{thm:reduced_map}, follows by induction.
	\end{proof}
	
	\begin{remark}
		\rev{The (rather direct) method of proof adopted in the section and the next is inspired by approaches to the approximation of maps by flows applied in e.g.~\cite{Kuznetsov2013,Kuznetsov2019}. We note, however, that several elements of the proofs can be `streamlined' by viewing the time-1 map as an exponential map in the Lie Algebra sense, which allows for the formal computation of the approximating vector field using a matrix logarithm expansion (as described in Remark \ref{rem:Lie_Algebra}). We refer again to \cite[Sec.~6]{Gramchev2005} for an example of this approach.}
	\end{remark}

	\section{Completing the proof of Theorem \ref{thm:center_manifold_embedding}}
	\label{app:partials_equality}
	
	In order to complete the proof of Theorem \ref{thm:center_manifold_embedding}, we need to prove the equality \eqref{eq:partials_equality}, restated here for convenience:
	\begin{equation}
		\label{eq:partials_equality_2}
		\frac{\partial \mathfrak N^u}{\partial x_m}(0,0) = \frac{\partial \widetilde N^u}{\partial x_m}(0,0) ,
	\end{equation}
	for all $m = 1, \ldots, k$. This can be shown after deriving the form of $j^2 V(x,y,0)$, the approximating vector field for $\eps = 0$ up to quadratic order, via a systematic approach based on formal matching which is similar to that applied in the proof of Theorem \ref{thm:reduced_map} in Appendix \ref{app:reduced_map}. We are particularly interested in the $u$-component of $j^2 V(x,y,0)$, which can be written as
	\begin{equation}
		\label{eq:j2V_1}
		j^2 V_2(x,y,0) = 
		\sum_{j=1}^k \frac{\partial \mathfrak N^u}{\partial x_j} (0,0) x_j u + O(u^2) ,
	\end{equation}
	where $x = (x_1, \ldots, x_k)$.
	
	We consider the layer map associated with \eqref{eq:restricted_map}, which can be rewritten as a formal expansion of the form
	\begin{equation}
		\label{eq:restricted_map_expansion}
		\widehat H_{W_{\textrm{loc}}^c(0)}(x,u,0) = 
		\left( \mathbb I_{k+1} + \Lambda \right) 
		\begin{pmatrix}
			x \\
			u
		\end{pmatrix}
		+ H^{(2)}(x,u) + \cdots + H^{(r)}(x,u) + o(|(x,u)|^r) ,
	\end{equation}
	where the linear part
	\[
	\Lambda = 
	\begin{pmatrix}
		\mathbb O_{k,k} & \widetilde N^x(0,0) \\
		\mathbb O_{1,k} & 0 
	\end{pmatrix} 
	\]
	is nilpotent with $\Lambda^2 = \mathbb O_{k+1,k+1}$ and the $u$-component of the quadratic part is
	\begin{equation}
		\label{eq:H2_quadratic}
		H_2^{(2)}(x,u) = \sum_{j=1}^k \frac{\partial \widetilde N^u}{\partial x_j}(0,0) x_j u .
	\end{equation}
	We want to determine the form of the truncated vector field $j^2 V(x,u,0)$, which we write as
	\begin{equation}
		\label{eq:j2V}
		j^2 V(x,u,0) = \Lambda 
		\begin{pmatrix}
			x \\
			u
		\end{pmatrix}
		+ F^{(2)}(x,u) ,
	\end{equation}
	where $F^{(2)}$ is a homegeneous quadratic function of $(x,u)$. Similarly to the proof of Theorem \ref{thm:reduced_map}, this can be done by matching coefficients of \eqref{eq:restricted_map_expansion} with coefficients in the Taylor expansion of the time-$1$ map induced by \eqref{eq:j2V}.
	
	We have that $\me^\Lambda = \mathbb I_{k+1} + \Lambda$, implying that the $1$-jet of the associated time-$1$ map $\Phi_V^1(x,u)$ is
	\[
	j^1 \Phi_V^1(x,u) = 
	\begin{pmatrix}
		x^{(1)}(t) \\
		u^{(1)}(t)
	\end{pmatrix}
	= \me^{\Lambda t} 
	\begin{pmatrix}
		x \\
		u
	\end{pmatrix}
	=
	\begin{pmatrix}
		\mathbb I_k & \widetilde N^x(0,0) t \\
		\mathbb O_{1,k} & 1
	\end{pmatrix}
	\begin{pmatrix}
		x \\
		u
	\end{pmatrix} 
	= 
	\begin{pmatrix}
		x + t \widetilde N^x(0,0) u \\
		u
	\end{pmatrix}.
	\]
	This expression can be used to simplify the expression for the $2$-jet, which is given by
	\[
	\begin{split}
		j^2 \Phi_V^1(x,u) &= 
		\begin{pmatrix}
			x^{(2)}(t) \\
			u^{(2)}(t)
		\end{pmatrix}
		= \me^{\Lambda t} 
		\begin{pmatrix}
			x \\
			u
		\end{pmatrix}
		+ \int_0^t \me^{\Lambda (t - \tau)} F^{(2)}(x^{(1)}(\tau), u^{(1)}(\tau) ) d \tau \\
		&= \me^{\Lambda t} 
		\begin{pmatrix}
			x \\
			u
		\end{pmatrix}
		+ \int_0^t
		\begin{pmatrix}
			F_1^{(2)} (x + \tau \widetilde N^x(0,0) u, u) + \widetilde N^x(0,0) (t - \tau) F_2^{(2)}(x + \tau \widetilde N^x(0,0) u, u) \\
			F_2^{(2)}(x + \tau \widetilde N^x(0,0) u, u)
		\end{pmatrix}
		d \tau .
	\end{split}
	\]
	Since \eqref{eq:partials_equality_2} only involves partial derivatives of $\mathfrak N^u$ and $\widetilde N^u$, it suffices to consider the matching equation obtained at quadratic order in the $u$-component, which is given by
	\begin{equation}
		\label{eq:H2eqn}
		H_2^{(2)}(x,u) - \int_0^1 F^{(2)}_2(x + \tau \widetilde N^x(0,0) u , u) d \tau = 0 .
	\end{equation}
	We write $H^{(2)}_2$ and $F_2^{(2)}$ in the form
	\[
	\begin{split}
		H_2^{(2)}(x,u) = \sum_{j_1 + \ldots + j_{k+1} = 2} b^{(2)}_{2, j_1 \cdots j_{k+1}} x_1^{j_1} \cdots x_k^{j_k} u^{j_{k+1}} , \\
		F_2^{(2)}(x,u) = \sum_{j_1 + \ldots + j_{k+1} = 2} B^{(2)}_{2, j_1 \cdots j_{k+1}} x_1^{j_1} \cdots x_k^{j_k} u^{j_{k+1}} ,
	\end{split}
	\]
	so that \eqref{eq:H2eqn} becomes
	\begin{equation}
		\label{eq:H2eqn_2}
		\begin{split}
			\sum_{j_1 + \cdots j_{k+1} = 2} &\bigg( b^{(2)}_{2, j_1 \cdots 	j_{k+1}} x_1^{j_1} \cdots x_k^{j_k} u^{j_{k+1}} - \\
			& \int_0^1 B^{(2)}_{2, j_1 \cdots j_{k+1}} (x_1 + \tau \widetilde N_1^x(0,0) u )^{j_1} \cdots (x_k + \tau \widetilde N_k^x(0,0) u)^{j_k} u^{j_{k+1}} \bigg) = 0 ,
		\end{split}
	\end{equation}
	where $\widetilde N^x = (\widetilde N^x_1 , \ldots , \widetilde N^x_k )^\transpose$. Note that \eqref{eq:j2V_1} and \eqref{eq:H2_quadratic} imply that
	\begin{equation}
		\label{eq:coefficients}
		B^{(2)}_{2, j_1 \cdots j_k 1} = \frac{\partial \mathfrak N^u}{\partial x_j}(0,0) , \qquad 
		b^{(2)}_{2, j_1 \cdots j_k 1} = \frac{\partial \widetilde N^u}{\partial x_j}(0,0) ,
	\end{equation}
	respectively for all combinations of $j_i$ such that $j_1 + \ldots + j_k = 1$.
	
	The series in \eqref{eq:H2eqn_2} can be decomposed into parts with $j_{k+1} = 0, 1$ or $2$, after which the integrals can be evaluated explicitly. The details are similar to those presented in the proof of Theorem \ref{thm:reduced_map} (the base case in particular), and so omitted for brevity. Matching coefficients in the resulting expression leads to the following formulae for the coefficients:
	\begin{equation}
		\label{eq:N_coefficients}
		\begin{aligned}
			B^{(2)}_{2, j_1 \cdots j_k 0} &= 0, && j_1 + \ldots + j_k = 2 , \\
			B^{(2)}_{2, j_1 \cdots j_k 1} &= b^{(2)}_{2, j_1 \cdots j_k 1} , && j_1 + \ldots + j_k = 1 , \\
			B^{(2)}_{2, 0 \cdots 0 2} &= b^{(2)}_{2,0 \cdots 0 2} + \frac{1}{2}	 \sum_{s = 1}^k b^{(2)}_{2, 0 \cdots 0 j_s 0 \cdots 0 1} \widetilde N^x_s(0,0) , \quad && \  
		\end{aligned}
	\end{equation}
	which uniquely determine $F^{(2)}_2(x,u)$. 
	
	Combining \eqref{eq:N_coefficients} with the second expression in \eqref{eq:coefficients} yields the desired equality in \eqref{eq:partials_equality_2}.
	\qed

\end{document}